\documentclass[11pt, reqno]{amsart}

\usepackage[top=3.75cm, bottom=3cm, left=3cm, right=3cm]{geometry}
\usepackage{amsmath}
\usepackage{amsthm}
\usepackage{amsfonts}
\usepackage{amssymb}
\usepackage{eucal}
\usepackage{bm}
\usepackage{hyperref}
\usepackage{color}
\usepackage[all,cmtip]{xy}
\usepackage{enumitem}
\usepackage{tikz}  
\usepackage{graphicx}
\usepackage{scalerel}

\numberwithin{equation}{section}

\theoremstyle{plain}
\newtheorem{theorem}{Theorem}[section]
\newtheorem{lemma}[theorem]{Lemma}
\newtheorem{proposition}[theorem]{Proposition}
\newtheorem{corollary}[theorem]{Corollary}

\theoremstyle{definition}
\newtheorem{definition}[theorem]{Definition}
\newtheorem{example}[theorem]{Example}
\newtheorem{remark}[theorem]{Remark}
\newtheorem{notation}[theorem]{Notation}

\setlist[itemize]{leftmargin=*, itemsep={2pt}}
\setlist[enumerate]{leftmargin=*, itemsep={2pt}}

\makeatletter
\newcommand*{\da@rightarrow}{\mathchar"0\hexnumber@\symAMSa 4B }
\newcommand*{\da@leftarrow}{\mathchar"0\hexnumber@\symAMSa 4C }
\newcommand*{\xdashrightarrow}[2][]{%
  \mathrel{%
    \mathpalette{\da@xarrow{#1}{#2}{}\da@rightarrow{\,}{}}{}%
  }%
}
\newcommand{\xdashleftarrow}[2][]{%
  \mathrel{%
    \mathpalette{\da@xarrow{#1}{#2}\da@leftarrow{}{}{\,}}{}%
  }%
}
\newcommand*{\da@xarrow}[7]{%
  \sbox0{$\ifx#7\scriptstyle\scriptscriptstyle\else\scriptstyle\fi#5#1#6\m@th$}%
  \sbox2{$\ifx#7\scriptstyle\scriptscriptstyle\else\scriptstyle\fi#5#2#6\m@th$}%
  \sbox4{$#7\dabar@\m@th$}%
  \dimen@=\wd0 %
  \ifdim\wd2 >\dimen@
    \dimen@=\wd2 %
  \fi
  \count@=2 %
  \def\da@bars{\dabar@\dabar@}%
  \@whiledim\count@\wd4<\dimen@\do{%
    \advance\count@\@ne
    \expandafter\def\expandafter\da@bars\expandafter{%
      \da@bars
      \dabar@ 
    }%
  }%
  \mathrel{#3}%
  \mathrel{%
    \mathop{\da@bars}\limits
    \ifx\\#1\\%
    \else
      _{\copy0}%
    \fi
    \ifx\\#2\\%
    \else
      ^{\copy2}%
    \fi
  }%
  \mathrel{#4}%
}
\makeatother

\newcommand{\bijartop}[1][]{%
 \ar[#1]
 \ar@<0.7ex>@{}[#1]|-*=0[@]{\sim}} 
 
 \newcommand{\bijarbottom}[1][]{%
 \ar[#1]
 \ar@<-0.95ex>@{}[#1]|-*=0[@]{\sim}} 

\newcommand{\st}{\mid} 
\newcommand{\set}[1]{\left\{ \, #1 \, \right\}}

\newcommand{\Perf}{\mathrm{Perf}}
\newcommand{\Db}{\mathrm{D^b_{coh}}}

\newcommand{\Fun}{\mathrm{Fun}}

\newcommand{\coh}{\mathrm{coh}}

\DeclareMathOperator{\length}{\mathrm{length}}

\newcommand{\hpd}{{\natural}}
\newcommand{\svee}{\scriptscriptstyle\vee}
\newcommand{\cAd}{\cA^\hpd}
\newcommand{\dcA}{{^\hpd}\cA}
\newcommand{\cJd}{\mathcal{J}^{\hpd}}
\newcommand{\barcJ}{\bar{\mathcal{J}}}

\newcommand{\llangle}{\left \langle}
\newcommand{\rrangle}{\right \rangle}
\newcommand{\sotimes}{\otimes}

\newcommand{\QCoh}{\mathrm{D_{qc}}}

\newcommand{\Sec}{\mathbf{Sec}}

\newcommand{\Spec}{\mathrm{Spec}}

\newcommand{\Gr}{\mathrm{Gr}}

\newcommand{\OGrp}{\mathrm{OGr_+}}
\newcommand{\OGrm}{\mathrm{OGr_-}}
\newcommand{\Spin}{\mathrm{Spin}}

\newcommand{\wtilde}{\widetilde}

\newcommand{\tJ}{{\tilde\bJ}}
\newcommand{\tJv}{{{\tilde\bJ}^{\svee}}}
\newcommand{\tJJ}{{\widetilde{\bJ\bJ}}}

\newcommand{\Bl}{\mathrm{Bl}}

\newcommand{\teps}{\tilde{\eps}}

\newcommand{\eps}{\varepsilon}
\newcommand{\tp}{{\tilde{p}}}

\DeclareMathOperator{\Sym}{Sym}
\DeclareMathOperator{\Proj}{Proj}
\DeclareMathOperator{\Pic}{Pic}

\newcommand{\stimes}{\times}

\newcommand{\tM}{{\wtilde{\cM}}}
\newcommand{\omegatp}{\omega_{\tp}}
\newcommand{\omegap}{\omega_{p}}
\newcommand{\cUv}{\mathcal{U}^{\svee}}

\DeclareMathOperator{\Cl}{\mathsf{Cliff}}

\newcommand{\vV}{V^{\svee}}
\newcommand{\vW}{W^{\svee}}
\newcommand{\tV}{\tilde{V}}

\newcommand{\tbE}{\tilde{\mathbf{E}}}

\DeclareMathOperator{\Tor}{Tor}

\newcommand{\id}{\mathrm{id}}
\newcommand{\pr}{\mathrm{pr}}

\newcommand{\rank}{\mathrm{rank}}

\newcommand{\cO}{\mathcal{O}}
\newcommand{\cA}{\mathcal{A}}
\newcommand{\cB}{\mathcal{B}}
\newcommand{\cC}{\mathcal{C}}
\newcommand{\cD}{\mathcal{D}}
\newcommand{\cE}{\mathcal{E}}

\newcommand{\cJ}{\mathcal{J}}
\newcommand{\cK}{\mathcal{K}}
\newcommand{\cL}{\mathcal{L}}
\newcommand{\cM}{\mathcal{M}}

\newcommand{\cP}{\mathcal{P}}

\newcommand{\cR}{\mathcal{R}}

\newcommand{\cT}{\mathcal{T}}
\newcommand{\cU}{\mathcal{U}}

\newcommand{\ccX}{\CMcal{X}}

\newcommand{\ccP}{\CMcal{P}}

\newcommand{\rR}{\mathrm{R}}

\newcommand{\fa}{\mathfrak{a}}
\newcommand{\fj}{\mathbf{j}}

\newcommand{\bE}{\mathbf{E}}
\newcommand{\bH}{\mathbf{H}}
\newcommand{\bJ}{\mathbf{J}}
\newcommand{\bZ}{\mathbf{Z}}
\newcommand{\bP}{\mathbf{P}}

\newcommand{\bk}{\mathbf{k}}
\newcommand{\bj}{\mathbf{j}}

\newcommand{\sS}{\mathsf{S}}
\newcommand{\sSs}{\mathsf{S}_{16}}

\begin{document}

\title{Categorical joins}

\author{Alexander Kuznetsov}
\address{{\sloppy
\parbox{0.9\textwidth}{
Steklov Mathematical Institute of Russian Academy of Sciences,\\
8 Gubkina str., Moscow 119991 Russia
\\[5pt]
National Research University Higher School of Economics, Moscow, Russia
}\smallskip}}
\email{akuznet@mi-ras.ru \medskip}

\author{Alexander Perry}
\address{Department of Mathematics, University of Michigan, Ann Arbor, MI 48109 \smallskip}
\email{arper@umich.edu}

\thanks{A.K. was partially supported by the HSE University Basic Research Program, Russian Academic Excellence Project ``5-100''. 
A.P. was partially supported by NSF postdoctoral fellowship DMS-1606460, 
NSF grant DMS-1902060/DMS-2002709, and the Institute for Advanced Study. 
}

\begin{abstract}
We introduce the notion of a categorical join, which can be thought of as a categorification 
of the classical join of two projective varieties. 
This notion is in the spirit of homological projective duality, 
which categorifies classical projective duality. 
Our main theorem says that 
the homological projective dual category of the categorical join 
is naturally equivalent to the categorical join of the homological projective dual categories. 
This categorifies the classical version of this assertion and has many applications, 
including a nonlinear version of the main theorem of homological projective duality.
\end{abstract}

\maketitle

\tableofcontents


\section{Introduction}
\label{section-intro}

 The theory of homological projective duality (HPD) introduced in~\cite{kuznetsov-hpd} is a powerful tool 
 for understanding the structure of 
 derived categories of algebraic varieties. 
 Given a smooth projective variety~$X$ with a morphism to a projective space $\bP(V)$
 and a special semiorthogonal decomposition
 of its derived category, 
 HPD associates a (noncommutative) smooth projective 
 variety $X^{\hpd}$ --- called the HPD of $X$ --- 
 with a morphism to the dual projective space $\bP(\vV)$. 
 This operation provides a categorification of classical projective duality.
 The main theorem of HPD
 describes the derived categories of linear sections of $X$ in terms of those of $X^\hpd$. 
 
 Since many interesting varieties can be expressed as linear sections of ``simple'' varieties, 
 e.g. many Fano threefolds are linear sections of homogeneous varieties, 
 this gives a potent strategy for studying derived categories: 
 \begin{enumerate}
 \item \label{HPD-strategy1} Obtain an explicit geometric description of the HPD of a simple variety. 
 \item \label{HPD-strategy2} Pass to 
 information about linear sections using the main theorem of HPD. 
 \end{enumerate}
 This strategy has been fruitfully carried out in a number of cases, 
 see~\cite{kuznetsov2014semiorthogonal, thomas2015notes} for surveys. 
 However, step \eqref{HPD-strategy1} is typically very difficult. 
 This raises the question of coming up with operations that allow us to construct 
 new solutions to \eqref{HPD-strategy1} from known ones. 

The main goal of this paper is to give an answer to this question. 
Namely, we define an operation called the \emph{categorical join}, which is modeled 
on (and can be thought of as a noncommutative resolution of) its classical geometric counterpart. 
Our main theorem says 
the formation of categorical joins commutes with HPD, parallel to the classical situation. 
In particular, if a description of the HPD of two varieties is known, this gives a description 
of the HPD of their categorical join. 
As a consequence of our main theorem, we prove a nonlinear generalization of the main theorem of HPD which greatly extends the scope of step \eqref{HPD-strategy2} above. 
These results have many concrete applications
detailed in \S\ref{intro-applications}. 

\subsection{Background on HPD} 
\label{subsection-NCHPD} 

The input for HPD is a \emph{Lefschetz variety} over $\bP(V)$.
This consists of a smooth projective variety $X$ equipped with a morphism $f \colon X \to \bP(V)$ 
and a \emph{Lefschetz decomposition} of the category $\Perf(X)$ 
of perfect complexes, namely a semiorthogonal decomposition of the form 
\begin{equation*}
\Perf(X) = \llangle \cA_0, 
\cA_1 \otimes f^*\cO_{\bP(V)}(1), \dots, \cA_{m-1} \otimes f^*\cO_{\bP(V)}(m-1) \rrangle 
\end{equation*}
where 
$0 \subset \cA_{m-1} \subset \dots \subset \cA_1 \subset \cA_0 \subset \Perf(X)$
is a chain of triangulated subcategories. 
Such a decomposition is in fact determined by the subcategory $\cA_0 \subset \Perf(X)$ \cite[Lemma~2.18]{kuznetsov2008lefschetz}, 
which we call the \emph{Lefschetz center}.

A Lefschetz variety $f \colon X \to \bP(V)$ is \emph{moderate} if $m < \dim(V)$;
this condition is essentially always satisfied in practice (see Remark~\ref{remark-moderate}),
and is important for running HPD.

\begin{example}
\label{example-linear-HPD}
If $W \subset V$ is a subspace of dimension $m$, then $\bP(W)$ is a Lefschetz variety over $\bP(V)$ 
with Lefschetz decomposition 
\begin{equation}
\label{eq:lefschetz-standard}
\Perf(\bP(W)) = \llangle \cO_{\bP(W)}, \cO_{\bP(W)}(1), \dots, \cO_{\bP(W)}(m-1) \rrangle.
\end{equation} 
Here the Lefschetz center is the subcategory $\langle \cO_{\bP(W)} \rangle$ 
generated by the structure sheaf~$\cO_{\bP(W)}$.
We call this the \emph{standard Lefschetz structure} on $\bP(W) \subset \bP(V)$. 
\end{example} 

The output of HPD applied to a moderate Lefschetz variety $f \colon X \to \bP(V)$
is a new moderate Lefschetz variety $f^\hpd \colon X^{\hpd} \to \bP(\vV)$, called the HPD of $X$.
The main theorem of HPD then describes the derived categories of linear sections of $X$ 
in terms of those of orthogonal linear sections of $X^{\hpd}$. 
More precisely, it says that if $L \subset V$ is a subspace, 
the derived categories of $X \times_{\bP(V)} \bP(L)$ and $X^{\hpd} \times_{\bP(\vV)} \bP(L^\perp)$
have a distinguished semiorthogonal component in common, while the other components 
come from the Lefschetz decompositions of $X$ and $X^{\hpd}$;
see Theorem~\ref{theorem-HPD} for a precise statement.

\begin{example}
\label{example:HPD-linear}
The HPD variety of $\bP(W) \subset \bP(V)$ with the standard Lefschetz structure 
is~$\bP(W^\perp) \subset \bP(\vV)$, where $W^\perp = \ker(\vV \to W^{\svee})$ is the orthogonal of $W$,
again with the standard Lefschetz structure.
\end{example}

In contrast with classical projective duality, HPD applies to morphisms $f \colon X \to \bP(V)$ 
that are not necessarily embeddings and preserves smoothness of $X$. 
However, in general, the associated HPD variety $f^{\hpd} \colon X^\hpd \to \bP(\vV)$ is noncommutative. 
This means that~$X^\hpd$ consists only of the data of a 
(suitably enhanced) triangulated category~$\Perf(X^\hpd)$  
equipped with a~$\bP(\vV)$-linear structure,  
i.e. an action of the monoidal category~$\Perf(\bP(\vV))$. 
In some cases,  
$\Perf(X^\hpd)$ can be identified with the derived category of a 
variety (as in Example~\ref{example:HPD-linear}), or with some other category of geometric origin, 
like the derived category of sheaves of modules over a finite sheaf of algebras on a variety. 
In these cases, $X^\hpd$ is called ``commutative'' or ``almost commutative'', 
but in general there is no underlying variety, and the symbol $X^\hpd$ 
just serves as a notational device. 

In~\cite{kuznetsov-hpd} HPD is developed under the assumption that~$X^\hpd$ is commutative. 
This was generalized in~\cite{NCHPD} where both the input~$X$ and output~$X^\hpd$ 
are allowed to be noncommutative (and not necessarily smooth or proper), making the theory completely symmetric; 
in this setting, Lefschetz varieties over~$\bP(V)$ are replaced with \emph{Lefschetz categories} over~$\bP(V)$, 
which are~$\bP(V)$-linear categories equipped with a Lefschetz decomposition.

In~\S\ref{section-HPD} we briefly review this framework. 
In the Introduction, 
for simplicity we stick as much as possible to the language of Lefschetz varieties, 
but in the body of the paper we work in the general setting of Lefschetz categories.

\subsection{Categorical joins} 
\label{subsection:intro-cat-joins}
Now let us outline the construction of a categorical join. 

Given two smooth projective varieties $X_1 \subset \bP(V_1)$ and $X_2 \subset \bP(V_2)$, their \emph{classical join} 
\begin{equation*}
\bJ(X_1, X_2) \subset \bP(V_1 \oplus V_2) 
\end{equation*} 
is the union of all the lines between points of $X_1$ and $X_2$ regarded as subvarieties of $\bP(V_1 \oplus V_2)$. 
The join is usually very singular (along the union $X_1 \sqcup X_2$), 
unless both~$X_1$ and $X_2$ are linear subspaces in $\bP(V_1)$ and $\bP(V_2)$, i.e. 
unless 
\begin{equation}
\label{eq:x12-linear}
X_1 = \bP(W_1) \subset \bP(V_1)
\qquad\text{and}\qquad
X_2 = \bP(W_2) \subset \bP(V_2),
\end{equation} 
in which case $\bJ(X_1,X_2) = \bP(W_1 \oplus W_2) \subset \bP(V_1 \oplus V_2)$.

The main problem with running HPD on the classical join is that in general $\bJ(X_1,X_2)$ does not 
have a natural Lefschetz structure. 
We would like to construct one from Lefschetz structures on $X_1$ and $X_2$. 
For this purpose, we pass to the \emph{resolved join}, defined by  
\begin{equation}
\label{eq:resolved-join-intro}
\tJ(X_1,X_2) = \bP_{X_1 \times X_2}(\cO(-H_1) \oplus \cO(-H_2)) 
\end{equation}
where $H_k$ is the pullback to $X_1 \times X_2$ of the hyperplane class of $\bP(V_k)$.
The resolved join is smooth since it is a $\bP^1$-bundle over $X_1 \times X_2$, 
and the canonical embedding 
\begin{equation*}
\cO(-H_1) \oplus \cO(-H_2) \hookrightarrow (V_1 \otimes \cO) \oplus (V_2 \otimes \cO) = 
(V_1 \oplus V_2) \otimes \cO 
\end{equation*} 
induces a morphism $\tJ(X_1,X_2) \to \bP(V_1 \oplus V_2)$ which factors birationally 
through the classical join. 
The morphism $\tJ(X_1, X_2) \to \bJ(X_1, X_2)$ blows down the 
two disjoint divisors 
\begin{equation}
\label{eq:intro-ek}
\eps_k \colon \bE_k(X_1, X_2) = \bP_{X_1 \times X_2}(\cO(-H_k)) \hookrightarrow \tJ(X_1,X_2), 
\quad 
k = 1,2, 
\end{equation} 
to $X_k \subset \bJ(X_1, X_2)$. 
Note that both $\bE_k(X_1, X_2)$ are canonically isomorphic to $X_1 \times X_2$. 

There are several advantages of the resolved join over the classical join: 
\begin{enumerate} 
\item Via the morphism $\tJ(X_1, X_2) \to X_1 \times X_2$ the resolved join 
is more simply related to $X_1$ and $X_2$ than the classical join. 

\item \label{rJ-advantage-2} The definition~\eqref{eq:resolved-join-intro} of the resolved join extends 
verbatim to the more general situation where the morphisms $X_1 \to \bP(V_1)$ and $X_2 \to \bP(V_2)$ are not necessarily embeddings, 
and with some work even to the case where $X_1$ and $X_2$ are replaced with 
$\bP(V_1)$- and~$\bP(V_2)$-linear categories (see Definition~\ref{definition-tJ}).

\item \label{rJ-advantage-3} The resolved join $\tJ(X_1, X_2)$ is smooth. 
\end{enumerate}
Points~\eqref{rJ-advantage-2} and~\eqref{rJ-advantage-3} should be thought of as parallel 
to the fact that HPD takes as input varieties that are not necessarily embedded in projective space 
(or even categories), and preserves smoothness.

However, there are still two issues with the resolved join $\tJ(X_1, X_2)$: 
it is ``too far'' from the classical join, 
and is not equipped with a natural Lefschetz structure over $\bP(V_1 \oplus V_2)$. 
To illustrate these problems, consider the simplest case~\eqref{eq:x12-linear} 
where~$X_1$ and~$X_2$ are linear subspaces of~$\bP(V_1)$ and~$\bP(V_2)$. 
Then the classical join \mbox{$\bJ(X_1, X_2) = \bP(W_1 \oplus W_2)$} is already smooth 
and comes with the standard Lefschetz structure~\eqref{eq:lefschetz-standard}. 
However, the resolved join~$\tJ(X_1, X_2)$, being a blowup of~$\bJ(X_1, X_2)$, contains 
in its derived category a copy of~$\Perf(\bJ(X_1, X_2))$, but also some 
extra components coming from the blowup centers. 
These extra components are irrelevant for the geometry of the join, and 
prevent $\Perf(\tJ(X_1,X_2))$ from having a natural Lefschetz structure. 
We thus need a 
general procedure for eliminating these extra components. 
The solution is the categorical join. 

\begin{definition}
Let $X_1 \to \bP(V_1)$ and $X_2 \to \bP(V_2)$ be Lefschetz 
varieties with Lefschetz centers 
$\cA^1_0 \subset \Perf(X_1)$ and $\cA^2_0 \subset \Perf(X_2)$. 
The \emph{categorical join} of $X_1$ and $X_2$ is  
the full subcategory of~$\Perf(\tJ(X_1, X_2))$ defined by 
\begin{equation*}
\label{intro-cat-join}
\cJ(X_1,X_2) = \left\{ C \in \Perf(\tJ(X_1, X_2)) \ \left|\ 
\begin{aligned}
\eps_1^*(C) &\in \Perf(X_1) \otimes \cA^2_0 \subset \Perf(\bE_1(X_1, X_2))  \\ 
\eps_2^*(C) &\in \cA^1_0 \otimes \Perf(X_2) \subset \Perf(\bE_2(X_1, X_2)) 
\end{aligned}
\right.\right\} ,
\end{equation*}
where $\eps_1$ and $\eps_2$ are the morphisms \eqref{eq:intro-ek} and 
\begin{align*}
&\Perf(X_1) \otimes \cA^2_0 \subset \Perf(X_1 \times X_2) = \Perf(\bE_1(X_1,X_2)),\\
&\cA^1_0 \otimes \Perf(X_2) \subset \Perf(X_1 \times X_2) = \Perf(\bE_2(X_1,X_2)) , 
\end{align*}
are the subcategories generated by objects of the form $F_1 \boxtimes F_2$ with $F_1 \in \Perf(X_1)$, 
$F_2 \in \cA^2_0$ in the first case, 
and with $F_1 \in \cA^1_0$, $F_2 \in \Perf(X_2)$ in the second case.
\end{definition}

If $X_1$ and $X_2$ are linear subspaces~\eqref{eq:x12-linear} 
with their standard Lefschetz structures~\eqref{eq:lefschetz-standard}, 
then there is an equivalence 
\begin{equation*}
\cJ(\bP(W_1), \bP(W_2)) \simeq \Perf(\bP(W_1 \oplus W_2)), 
\end{equation*} 
so the categorical join gives the desired category (see Example~\ref{example:cj-pp} for details). 

In general, the categorical join $\cJ(X_1, X_2)$ is not equivalent to the derived category of a variety, 
i.e. is not commutative. Rather, $\cJ(X_1, X_2)$ should be thought of as a noncommutative 
birational modification of the resolved join $\tJ(X_1, X_2)$. 
Indeed, away from the exceptional divisors \eqref{eq:intro-ek}, the categorical join coincides with the resolved join, which in turn coincides with the classical join if the morphisms $X_1 \to \bP(V_1)$ and $X_2 \to \bP(V_2)$ are embeddings (see Lemma~\ref{lemma-tJ}, Proposition~\ref{proposition-cJT}, and Remark~\ref{remark-cJ-vs-classical} for precise statements). 
In particular, if~$X_1 \to \bP(V_1)$ and $X_2 \to \bP(V_2)$ are embeddings, 
then~$\cJ(X_1, X_2)$ can be thought of as a noncommutative resolution 
of singularities of the classical join~$\bJ(X_1, X_2)$. 

We prove in Lemma~\ref{lemma-tJ} that the categorical join is a~$\bP(V_1 \oplus V_2)$-linear subcategory of~$\Perf(\tJ(X_1, X_2))$.
The following shows that it also has a Lefschetz decomposition, i.e. the  structure of a Lefschetz category over~$\bP(V_1 \oplus V_2)$, and 
hence gives a suitable input for HPD.

\begin{theorem}[{Theorem~\ref{theorem-join-lef-cat}}]
\label{theorem-ld-cJ} 
Let $X_1 \to \bP(V_1)$ and $X_2 \to \bP(V_2)$ be moderate Lefschetz varieties with 
Lefschetz centers \mbox{$\cA^1_0 \subset \Perf(X_1)$} and $\cA^2_0 \subset \Perf(X_2)$. 
Then the categorical join~$\cJ(X_1, X_2)$ has the structure of a moderate Lefschetz category 
over $\bP(V_1 \oplus V_2)$, with Lefschetz center equivalent to the subcategory 
$\cA^1_0 \otimes \cA^2_0 \subset \Perf(X_1 \times X_2)$ generated by objects of the form 
$F_1 \boxtimes F_2$ for $F_1 \in \cA^1_0$, $F_2 \in \cA^2_0$. 
\end{theorem} 

Our main result describes the HPD of a categorical join. 

\begin{theorem}[{Theorem~\ref{theorem-joins-HPD}}]
\label{main-theorem-intro}
Let $X_1 \to \bP(V_1)$ and $X_2 \to \bP(V_2)$ be moderate Lefschetz varieties.  
Let~$X^\hpd_1 \to \bP(\vV_1)$ and $X^\hpd_2 \to \bP(\vV_2)$ be the HPD varieties. 
Then there is an equivalence 
\begin{equation*}
\cJ(X_1, X_2)^{\hpd} \simeq \cJ(X^\hpd_1, X^\hpd_2) 
\end{equation*} 
of Lefschetz categories over $\bP(\vV_1 \oplus \vV_2)$, i.e. a\/ $\bP(\vV_1 \oplus \vV_2)$-linear equivalence 
identifying the Lefschetz centers on each side.
\end{theorem}

Theorem~\ref{main-theorem-intro} can be thought of as a categorification 
of the classical result that the operations of classical join and projective 
duality commute, i.e. for $X_1 \subset \bP(V_1)$ and $X_2 \subset \bP(V_2)$ 
we have 
\begin{equation*}
\bJ(X_1, X_2)^{\svee} = \bJ(X_1^{\svee}, X_2^{\svee}),
\end{equation*}
where $(-)^{\svee}$ denotes the operation of classical projective duality. 

In the paper we extend the construction of categorical joins to Lefschetz categories, 
and prove Theorem~\ref{main-theorem-intro} in this context. 
This is needed even if $X_1$ and $X_2$ are varieties, because in 
general the HPD varieties $X^\hpd_1$ and $X^\hpd_2$ will be noncommutative, i.e. only exist as Lefschetz categories. 

\subsection{The nonlinear HPD theorem} 
\label{ss:intro-nonlinear-hpd}
Given closed subvarieties $X_1 \subset \bP(V)$ and~$X_2 \subset \bP(V)$ of the \emph{same} projective space, 
the classical join $\bJ(X_1, X_2)$ (as defined in~\S\ref{subsection:intro-cat-joins}) is a subvariety of~$\bP(V \oplus V)$.  
Let~$W \subset V \oplus V$ be the graph of an isomorphism $\xi \colon V \to V$ 
given by scalar multiplication; e.g. $W \subset V \oplus V$ is the diagonal for~$\xi = \id$ and the antidiagonal for~$\xi = -\id$. 
Then we have 
\begin{equation*}
\label{classical-join-intersection}
\bJ(X_1, X_2) \cap \bP(W) \cong X_1 \cap X_2. 
\end{equation*}
If, more generally, we have morphisms $X_1 \to \bP({V})$ and $X_2 \to \bP({V})$ 
instead of embeddings, then 
\begin{equation*}
\label{classical-join-fiber-product}
\tJ(X_1, X_2) \times_{\bP(V \oplus V)} \bP(W) \cong X_1 \times_{\bP(V)} X_2. 
\end{equation*}

Categorifying this, we show in Proposition~\ref{proposition-cJT} that if $X_1 \to \bP({V})$ and $X_2 \to \bP({V})$ are 
Lefschetz varieties, then 
\begin{equation*}
\cJ(X_1, X_2)_{\bP(W)} \simeq \Perf{\left( X_1 \times_{\bP(V)} X_2 \right)}
\end{equation*}
where the left side is the base change of $\cJ(X_1, X_2)$ along the embedding $\bP(W) \to \bP(V \oplus V)$. 
Here and below, all fiber products of schemes are taken in the derived sense, 
but we note that in the $\Tor$-independent case this agrees with the usual fiber product of schemes (see \S\ref{subsection-conventions}). 

The orthogonal subspace to the diagonal $V \subset V \oplus V$ is the antidiagonal 
$V^{\svee} \subset V^{\svee} \oplus V^{\svee}$. 
Thus combining Theorem~\ref{main-theorem-intro} with the main theorem of HPD, 
we obtain the following result, that we formulate here loosely; 
see Theorem~\ref{theorem-nonlinear-HPD} for the precise statement. 

\begin{theorem}
\label{intro-nonlinear-HPD}
Let $X_1 \to \bP({V})$ and $X_2 \to \bP({V})$ be moderate Lefschetz varieties, with HPD varieties $X^\hpd_1 \to \bP({\vV})$ and $X^\hpd_2 \to \bP({\vV})$.  
Then there are induced semiorthogonal decompositions of 
\begin{equation}
\label{A1A2fiberproduct}
\Perf{ ( X_1 \times_{\bP(V)} X_2 )}
\qquad \text{and} \qquad 
\Perf{ ( X^\hpd_1 \times_{\bP(\vV)} X^\hpd_2 )}
\end{equation}
which have a distinguished component in common. 
\end{theorem}

Using~\cite{bznp} or (if the fiber products in~\eqref{A1A2fiberproduct} are assumed $\Tor$-independent) \cite{kuznetsov-base-change}, 
Theorem~\ref{intro-nonlinear-HPD} also implies an analogous result at the level 
of bounded derived categories of coherent sheaves in place of perfect complexes; 
see Remark~\ref{remark-Perf-to-Db} for details.

If $X_2 = \bP(L) \subset \bP(V)$ for a vector subspace 
$0 \subsetneq L \subsetneq V$, then the HPD variety is the 
orthogonal space $X^\hpd_2 = \bP(L^{\perp}) \subset \bP(\vV)$ (see Example~\ref{example:HPD-linear}). 
Hence 
\begin{equation*}
X_1 \times_{\bP(V)} X_2 = X_1 \times_{\bP(V)} \bP(L)
\qquad\text{and}\qquad 
X^\hpd_1 \times_{\bP(\vV)} X^\hpd_2 = X^\hpd_1 \times_{\bP(\vV)} \bP(L^\perp)
\end{equation*}
are mutually orthogonal linear sections of $X_1$ and $X^\hpd_1$.
The result of Theorem~\ref{intro-nonlinear-HPD} then reduces to the main theorem of HPD (Theorem~\ref{theorem-HPD}).
Accordingly, Theorem~\ref{intro-nonlinear-HPD} should be thought of as a \emph{nonlinear} version of the main theorem of HPD.

\begin{remark}
Jiang, Leung, and Xie~\cite{categorical-plucker} established a version of 
Theorem~\ref{intro-nonlinear-HPD} using a different argument which does not involve joins. 
We note that our result is more general in that~$X_k$ and~$X_k^{\hpd}$ are allowed to be 
noncommutative, and we do not require any transversality hypotheses 
(at the expense of considering derived fiber products in the non-transverse case). 
Moreover, our proof places Theorem~\ref{intro-nonlinear-HPD} in a larger conceptual framework, 
from which it follows as a corollary of the much more general  
Theorem~\ref{main-theorem-intro}. 
This framework also leads to many results 
beyond those in~\cite{categorical-plucker},
including an iterated version of Theorem~\ref{intro-nonlinear-HPD} 
for the fiber product of any number of Lefschetz varieties (Theorem~\ref{theorem-iterated-nonlinear-HPD}), 
a description of the derived categories of Enriques surfaces (Theorem~\ref{theorem-enriques}), 
a quadratic HPD theorem with many applications (see the discussion in~\S\ref{intro-applications}), 
and new derived equivalences of Calabi--Yau threefolds (see~\cite{inoue}). 
\end{remark}

\subsection{Applications} 
\label{intro-applications} 
Theorems~\ref{main-theorem-intro} and~\ref{intro-nonlinear-HPD} have many applications. 
We provide a few of them in~\S\ref{section:applications} of the paper, to show how the theory works.

First, we consider the Grassmannians $\Gr(2,5) \subset \bP^9$ and $\OGrp(5,10) \subset \bP^{15}$, 
which have the special property of being (homologically) projectively self-dual. 
Given two copies of a Grassmannian of either type, we obtain a pair of 
varieties by forming their intersection and the intersection of their projective duals.  
We show in~\S\ref{subsection:intersection-gr25} that the $\Gr(2,5)$ case  
gives a new pair of derived equivalent Calabi--Yau threefolds, 
and the $\OGrp(5,10)$ case gives a new pair of derived equivalent Calabi--Yau fivefolds. 
These equivalences are a key ingredient in the recent papers~\cite{GPK3,jj-torelli,double-spinor}, 
which show these pairs of varieties lead to the first known counterexamples to 
the birational Torelli problem for Calabi--Yau varieties. 

As another example, we use Theorem~\ref{main-theorem-intro} to understand the derived 
category of a general Enriques surface $\Sigma$. 
Namely, we prove 
the orthogonal subcategory 
$\langle \cO_{\Sigma} \rangle^{\perp} \subset \Perf(\Sigma)$ 
to the structure sheaf $\cO_\Sigma$ 
embeds into the twisted derived category of a stacky projective plane as the orthogonal to an 
exceptional object (Theorem~\ref{theorem-enriques}). 
This result can be regarded as an algebraization of the logarithmic transform, which creates 
an Enriques surface from a rational elliptic surface with two marked fibers. 

For further applications, in a sequel~\cite{categorical-cones} to this paper 
we show that our results yield a very useful ``quadratic'' HPD theorem. 
Using our nonlinear HPD theorem (Theorem~\ref{intro-nonlinear-HPD}), this boils 
down to describing HPD for a quadric $Q \to \bP(V)$. 
For smooth quadrics we do this in~\cite{kuznetsov-perry-HPD-quadrics}. 
For singular quadrics we develop a theory of \emph{categorical cones},
which are categorical joins with a projective space and thus provide categorical resolutions of classical cones.
In particular, categorical cones over smooth quadrics provide categorical resolutions of singular quadrics, 
and a description of their HPD reduces to the smooth case. 

In \cite{categorical-cones} we use the resulting quadratic HPD theorem to prove  
the duality conjecture for Gushel--Mukai varieties as stated in~\cite{kuznetsov2016perry}; 
as a special case, this gives a new, conceptual proof of the main result of~\cite{kuznetsov2016perry}. 
Besides providing a close connection between the birational geometry, Hodge theory, and 
derived categories of Gushel--Mukai varieties, this gives families of noncommutative 
deformations of the derived categories of K3 surfaces, analogous to the family of 
noncommutative K3 surfaces associated to cubic fourfolds \cite{kuznetsov2010derived}. 
In \cite{categorical-cones}, we also give other applications of the quadratic HPD theorem, 
including examples of noncommutative Calabi--Yau threefolds which are not equivalent to the derived 
category of a variety, but which admit singular degenerations with crepant resolutions by 
the derived category of a genuine Calabi--Yau threefold. 

Recently, Inoue \cite{inoue} applied our results to construct several new examples of derived equivalent 
but non-birational Calabi--Yau threefolds. 
These examples are smoothings of the classical join of suitable elliptic curves. 

\subsection{Homological projective geometry} 
\label{section-homological-projective-geometry} 

Our results suggest the existence of a robust theory of \emph{homological projective geometry}, 
of which homological projective duality, categorical joins, and categorical cones 
are the first instances. 
In this theory, Lefschetz categories over~$\bP(V)$ should play the role of 
projective varieties embedded in $\bP(V)$. 
An interesting feature of the operations of homological projective geometry known so far
is that they preserve smoothness of the objects involved, 
whereas in classical projective geometry this is far from true. 
In fact, this principle of ``homological smoothness'' guided our constructions.

The vision of homological projective geometry is alluring  
because the known results are so powerful, 
and yet they correspond to a small sector of the vast theory of 
classical projective geometry. 
For instance, it would be very interesting to categorify secant varieties, 
and prove an HPD statement for them. 
The ideas of this paper should be useful for making progress in this area. 
As an illustration, in Appendix~\ref{section-future} we discuss projected categorical joins --- 
which in particular give an approximation to the sought for theory of categorical secant varieties --- 
and show that under HPD they correspond to fiber products. 

\subsection{Noncommutative algebraic geometry framework}
\label{ss:intro-cat-framework}

To finish the Introduction, we explain the framework of noncommutative algebraic geometry adopted in this paper. 
Since the categorical join is defined only as a Lefschetz category and not a variety, 
such a framework is absolutely necessary for us --- without one we could not even formulate our main results. 
On the other hand, our approach is sufficiently flexible to work within any framework 
that has appropriate notions of \emph{$T$-linear categories} and \emph{$T$-linear functors} over an arbitrary scheme~$T$ (not just $T = \bP(V)$)
that satisfy some natural properties.

Specifically, a $T$-linear category should carry an action of the monoidal category~$\Perf(T)$, 
and include as examples derived categories of schemes over~$T$, 
and more generally admissible subcategories preserved under the action of~$\Perf(T)$. 
Moreover, there should be a notion of tensor products of $T$-linear categories, which 
in the case of derived categories of schemes is compatible with taking (derived) fiber products. 
Let us briefly describe two alternatives to defining such a class of $T$-linear categories: 

\begin{enumerate}
\item 
\label{NCAG-min}
\emph{Down-to-earth approach}:
Define a $T$-linear category to be an admissible subcategory $\cA$ of~$\Perf(X)$ or $\Db(X)$ 
(the bounded derived category of coherent sheaves)
which is preserved by the $\Perf(T)$-action, where $X$ is a proper scheme over $T$,  
and define a $T$-linear functor between $\cA \subset \Db(X)$ and $\cB \subset \Db(Y)$ 
to be a functor induced by a Fourier--Mukai functor between $\Db(X)$ and $\Db(Y)$ 
with kernel schematically supported on the fiber product $X \times_T Y$. 
In this setting, a base change operation (along morphisms satisfying a transversality assumption) with 
the necessary compatibilities is developed in~\cite{kuznetsov-base-change}.
A version of HPD in this context is described in~\cite{kuznetsov-hab}. 

\item 
\label{NCAG-max}
\emph{Higher approach}:
Define a $T$-linear category to be a small idempotent-complete stable $\infty$-category 
equipped with a $\Perf(T)$-module structure, 
and define a $T$-linear functor to be an exact functor commuting with the $\Perf(T)$-modules structures.
Relying on Lurie's foundational work \cite{lurie-HA}, this approach is developed in \cite{NCHPD} 
and used to give a version of HPD in this context. 
\end{enumerate}

The advantage of~\eqref{NCAG-min} is that it avoids the use of higher 
category theory and derived algebraic geometry. 
The advantages of~\eqref{NCAG-max} are that it includes~\eqref{NCAG-min} 
as a special case, allows us to prove more general results (e.g. over general 
base schemes and without transversality hypotheses), and there is a complete 
reference~\cite{NCHPD} for the results we need in this setting. 
In particular, \cite{NCHPD} proves a version of HPD which allows linear categories 
as both inputs and outputs. 

To fix ideas, in this paper we adopt approach~\eqref{NCAG-max}.  
However, the reader who prefers~\eqref{NCAG-min} (or any other appropriate framework)
should have no trouble translating everything to that setting. 
In fact, we encourage the reader who is not already familiar with noncommutative 
algebraic geometry to assume all noncommutative schemes are ``commutative'', 
i.e. of the form $\Perf(X)$; for intuition, we have explained throughout the paper 
what our constructions amount to in this situation. 

To recapitulate, from now on we use the following definition.

\begin{definition}
\label{definition-linear-category}
Let $T$ be a scheme. 
A \emph{$T$-linear category} is a small 
idempotent-complete stable $\infty$-category equipped with a $\Perf(T)$-module structure, 
and a \emph{$T$-linear functor} between $T$-linear categories 
is an exact functor of $\Perf(T)$-modules. 
\end{definition}

In Appendix \ref{section:linear-cats} we summarize the key facts about $T$-linear categories used in this paper. 

\subsection{Conventions} 
\label{subsection-conventions}

All schemes are assumed to be quasi-compact and separated. 
Instead of working over a ground field, we work relative to a fixed base scheme $S$ throughout the paper. 
Namely, all schemes will be defined over~$S$ and all categories will be linear over $S$. 
The only time we make extra assumptions on $S$ is in our discussion of applications 
in~\S\ref{section:applications}, where for simplicity 
we assume $S$ is the spectrum of a field.

A vector bundle $V$ on a scheme $S$ is a finite locally free $\cO_S$-module of constant rank.
Given such a $V$, its projectivization is 
\begin{equation*}
\bP(V) = \Proj(\Sym^\bullet(V^{\svee})) \to S
\end{equation*}
with $\cO_{\bP(V)}(1)$ normalized so that its pushfoward to $S$ is $V^{\svee}$. 
Note that we suppress $S$ by writing $\bP(V)$ instead of $\bP_S(V)$. 
A subbundle $W \subset V$ is an inclusion of vector bundles whose cokernel is a vector bundle. 
Given such a $W \subset V$, the orthogonal subbundle is defined by 
\begin{equation}
\label{eq:perp-bundle}
W^{\perp} = \ker(V^{\svee} \to W^{\svee}). 
\end{equation}

We often commit the following convenient abuse of notation: given a line bundle~$\cL$ or a divisor class~$D$ on a scheme~$T$, 
we denote still by~$\cL$ or~$D$ its pullback to any variety mapping to~$T$. 
Similarly, if $X \to T$ is a morphism and $V$ is a vector bundle on~$T$, we sometimes 
write~\mbox{$V \otimes \cO_X$} for the pullback of~$V$ to~$X$.

Given morphisms of schemes $X \to T$ and $Y \to T$, the symbol 
$X \times_T Y$ denotes their \emph{derived} fiber product. 
We refer to \cite{TV, lurie-SAG, gaitsgory-DAG} for treatments of derived algebraic geometry or \cite{toen-survey} for a survey. 
The existence of fiber products of derived schemes is proved in \cite[\S1.3.3]{TV}; 
explicitly, if $X = \Spec(A)$, $Y=\Spec(B)$, $T = \Spec(C)$ are affine, 
then the fiber product is computed by the spectrum of the derived tensor product $A \otimes_C B$ 
(viewed as a simplicial commutative ring), and in general $X \times_T Y$ is glued from such local affine pictures. 
The derived fiber product agrees with the usual fiber product of schemes whenever the 
morphisms~\mbox{$X \to T$} and~\mbox{$Y \to T$} are $\Tor$-independent over~$T$. 
To lighten the notation, we write fiber products over our fixed base~$S$ as absolute fiber products, i.e. we write 
\begin{equation*}
X \times Y = X \times_S Y. 
\end{equation*}

If $X$ is a scheme over $T$, we denote by $\Perf(X)$ its category of perfect complexes and 
by~$\Db(X)$ its bounded derived category of coherent sheaves, which we consider as $T$-linear 
categories in the sense of Definition~\ref{definition-linear-category}. 
The $\Perf(T)$-module structure on these categories is given by the (derived) tensor product 
with the (derived) pullback of objects from $\Perf(T)$. 

For a triangulated subcategory $\cA \subset \cT$ we denote by~${}^\perp\cA$ and~$\cA^\perp$ 
the left and right orthogonals to~$\cA$ in~$\cT$. 

We always consider derived functors (pullbacks, pushforwards, tensor products, etc.),  
but write them with underived notation. 
For example, for a morphism of schemes~$f \colon X \to Y$ 
we write $f^* \colon \Perf(Y) \to \Perf(X)$ for the derived pullback functor, and similarly for the functors~$f_*$ and~$\otimes$. 
We always work with functors defined between categories of perfect complexes. 
Note that in general, $f_*$ may not preserve perfect complexes, but it does if 
$f \colon X \to Y $ is a perfect (i.e. pseudo-coherent of finite $\Tor$-dimension) proper morphism \cite[Example 2.2(a)]{lipman}. 
This assumption will be satisfied in all of the cases where we use $f_*$ in the paper. 

The functor $f_*$ is right adjoint to $f^*$. 
Sometimes, we need other adjoint functors as well.
Provided they exist, we denote by $f^!$ the right adjoint of $f_* \colon \Perf(X) \to \Perf(Y)$ and by 
$f_!$ the left adjoint of $f^* \colon \Perf(Y) \to \Perf(X)$, so that $(f_!,f^*,f_*,f^!)$ is an adjoint sequence.
For instance, if $f \colon X \to Y$ is a perfect proper morphism and 
the relative dualizing complex $\omega_f$ is a shift of a line bundle on~$X$ (e.g. if $f$ is Gorenstein), 
then $f^!$ and $f_!$ exist and are given by 
\begin{equation}
\label{eq:shriek-adjoints}
f^!(-) \simeq f^*(-) \otimes \omega_f 
\qquad \text{and} \qquad 
f_!(-) \simeq f_*(- \otimes \omega_f). 
\end{equation}
Indeed, (due to our standing quasi-compactness and separatedness assumptions) for any morphism~$f$ 
the functor  $f_* \colon \QCoh(X) \to \QCoh(Y)$ between unbounded derived categories of quasi-coherent sheaves 
admits a right adjoint $f^! \colon \QCoh(Y) \to \QCoh(X)$, 
the relative dualizing complex of $f$ is by definition $\omega_f = f^!(\cO_Y)$, 
and if~$f$ is a perfect proper morphism then the above formula for~$f^!$ holds by~\cite[Proposition 2.1]{lipman}; 
if further~$\omega_f$ is a shift of a line bundle, then the formula for~$f_!$ follows from the one for~$f^!$. 
Hence if~$f$ is a perfect proper morphism and~$\omega_f$ is a shift of a line bundle, 
it follows that all of these functors and adjunctions restrict to categories of perfect complexes. 
In all of the cases where we need~$f^!$ and~$f_!$ in the paper, the following stronger assumptions will be satisfied.

\begin{remark}
\label{remark:good-morphism}
Suppose $f \colon X \to Y$ is a morphism between schemes which 
are smooth, projective, and of constant relative dimension over $S$. 
Then $f$ is perfect, projective, and has a relative dualizing complex, which is a shift of a line bundle: 
\begin{equation}
\label{eq:intro-omega-f}
\omega_{f} = \omega_{X/S} \otimes f^*(\omega_{Y/S})^{\svee} . 
\end{equation}
In particular, for such an $f$, all of the functors $f_!$, $f^*$, $f_*$, $f^!$ are defined 
and adjoint between categories of perfect complexes 
and the isomorphisms~\eqref{eq:shriek-adjoints} hold.
\end{remark} 

Given a $T$-linear category $\cC$, we write $C \otimes F$ 
for the action of an object $F \in \Perf(T)$ on an object $C \in \cC$. 
Given $T$-linear categories $\cC$ and $\cD$, we denote by $\cC \otimes_{\Perf(T)} \cD$ their 
$T$-linear tensor product, see~Appendix \ref{section:linear-cats}.
Parallel to our convention for fiber products of schemes, if~\mbox{$T = S$} is our fixed base scheme, 
we simplify notation by writing 
\begin{equation*}
\cC \otimes \cD = \cC \otimes_{\Perf(S)} \cD. 
\end{equation*}
If $\cC$ is a $T$-linear category and $T' \to T$ is a morphism of schemes, we write 
\begin{equation*}
\cC_{T'} = \cC \otimes_{\Perf(T)} \Perf(T')
\end{equation*}
for the $T'$-linear category obtained by \emph{base change}. 
By abuse of notation, if $\cC$ is a $T$-linear category and $\psi \colon \cD_1 \to \cD_2$ 
is a $T$-linear functor, 
then we frequently write $\psi$ for the induced functor 
\begin{equation*}
\psi \colon \cC \otimes_{\Perf(T)} \cD_1 \to \cC \otimes_{\Perf(T)} \cD_2. 
\end{equation*} 
Finally, if $\phi \colon \cC \to \cD$ is a $T$-linear functor, we write $\phi^*, \phi^! \colon \cD \to \cC$ 
for its left and right adjoints if they are defined, in which case they are automatically $T$-linear (\cite[Lemma 2.11]{NCHPD}).

\subsection{Organization of the paper} 
In \S\ref{section-HPD} we gather preliminaries on HPD, including 
a useful new characterization of the HPD category on which the proof of our main theorem relies.

In \S\ref{section-categorical-joins} we define the categorical join 
of two Lefschetz categories, show that it is equipped with a canonical 
Lefschetz structure (Theorem~\ref{theorem-ld-cJ}), and study its behavior under base change. 

In \S\ref{section-joins-HPD} we prove our main theorem, stated above 
as Theorem~\ref{main-theorem-intro}. 

In \S\ref{section-nonlinear-HPD} we prove the nonlinear HPD theorem, 
stated above as Theorem \ref{intro-nonlinear-HPD}. 

In \S\ref{section:applications} we discuss the applications  
of the previous two theorems mentioned in \S\ref{intro-applications}. 

In Appendix~\ref{section:linear-cats} we collect some useful results on linear categories. 

In Appendix~\ref{section-future} we show how our methods give a categorification of linear projections of joins.

\subsection{Acknowledgements} 
We are grateful to Roland Abuaf and Johan de Jong for useful conversations,
and to Sergey Galkin for the suggestion to consider the Enriques surface example of~\S\ref{subsection-Enriques}. 
We would also like to thank the referee for their comments.  


\section{Preliminaries on HPD}
\label{section-HPD} 

In this section, we discuss preliminary material on HPD that will be needed 
in the rest of the paper. 
In \S\ref{subsection-lef-cats} we review the notion of a Lefschetz category, 
and in \S\ref{subsection-HPD} we recall the definition of the HPD category 
and state the main theorem of HPD. 
In \S\ref{subsection-HPD-characterization} we prove a useful 
characterization of the HPD category in terms of the projection functor 
from the universal hyperplane section. 

Recall that we work relative to a general base scheme $S$. 
In particular, we consider HPD over a projective bundle $\bP(V)$, where $V$ is a vector bundle on $S$. 
This is convenient because it includes various relative versions of HPD (cf. \cite[Theorem~6.27 and Remark~6.28]{kuznetsov-hpd}) into the general framework.
We denote by $N$ the rank of $V$ and by 
$H$ the relative hyperplane class on the projective bundle $\bP(V)$ 
such that~$\cO(H) = \cO_{\bP(V)}(1)$. 

\subsection{Lefschetz categories} 
\label{subsection-lef-cats}

The fundamental objects of HPD are Lefschetz categories. 
We summarize the basic definitions following \cite[\S6]{NCHPD}, 
starting with the notion of a Lefschetz center. 

\begin{definition}
\label{definition:lefschetz-category}
Let $T$ be a scheme over $S$ with a line bundle $\cL$.
Let $\cA$ be a $T$-linear category. 
An admissible
{$S$-linear} subcategory $\cA_0 \subset \cA$ is called a \emph{Lefschetz center} of $\cA$ with respect to~$\cL$ 
if the subcategories $\cA_i \subset \cA$, $i \in \bZ$, determined by 
\begin{align}
\label{Ai-igeq0}
\cA_i & = \cA_{i-1} \cap {}^\perp(\cA_0 \otimes \cL^{-i})\hphantom{{}^\perp}, \quad 
i \ge 1 \\ 
\label{Ai-ileq0}
\cA_i & = \cA_{i+1} \cap \hphantom{{}^\perp}(\cA_0 \otimes \cL^{-i})^\perp, \quad 
i \le -1 
\end{align} 
are right admissible in $\cA$ for $i \geq 1$, left admissible in $\cA$ for $i \leq -1$,  
vanish for all $i$ of sufficiently large absolute value, say for $|i| \geq m$, and 
provide {$S$-linear} semiorthogonal decompositions 
\begin{align}
\label{eq:right-decomposition}
\cA & = \langle \cA_0, \cA_1 \otimes \cL, \dots, \cA_{m-1} \otimes \cL^{m-1} \rangle, \\ 
\label{eq:left-decomposition}
\cA & = \langle \cA_{1-m} \otimes \cL^{1-m}, \dots, \cA_{-1} \otimes \cL^{-1}, \cA_0 \rangle. 
\end{align} 
The categories $\cA_i$, $i \in \bZ$, are called the \emph{Lefschetz components} of the Lefschetz center $\cA_0 \subset \cA$. 
The semiorthogonal decompositions~\eqref{eq:right-decomposition} 
and~\eqref{eq:left-decomposition} are called \emph{the right Lefschetz decomposition} and 
\emph{the left Lefschetz decomposition} of $\cA$. 
The minimal $m$ above is called the \emph{length} of the Lefschetz decompositions. 
\end{definition}

The Lefschetz components form two (different in general) chains of subcategories
\begin{equation}
\label{eq:lefschetz-chain}
0 \subset \cA_{1-m} \subset \dots \subset \cA_{-1} \subset \cA_0 \supset \cA_1 \supset \dots \supset \cA_{m-1} \supset 0.
\end{equation} 
Note that the assumption of right or left admissibility of $\cA_i$ in $\cA$ 
is equivalent to the assumption of right or left admissibility in $\cA_0$.

\begin{remark}
By \cite[Lemma 6.3]{NCHPD}, if the subcategories $\cA_i \subset \cA$ are admissible 
for all $i \geq 0$ or all $i \leq 0$, then the length $m$ defined above satisfies 
\begin{equation*}
m = \min \set{ i \geq 0 \st \cA_{i} = 0 } = \min \set{ i \geq 0 \st \cA_{-i} = 0 } . 
\end{equation*}
This holds true, e.g., if~$\cA$ is smooth and proper over~$S$ \cite[Lemma 4.15]{NCHPD}. 
\end{remark}

\begin{remark}
\label{remark-alpha-ld}
If $\cA$ is a $T$-linear category equipped with a $T$-linear 
autoequivalence $\alpha \colon \cA \to \cA$, there is a more general notion of 
a Lefschetz center of $\cA$ with respect to $\alpha$, see \cite[\S6.1]{NCHPD}. 
The case where $\alpha$ is the autoequivalence $- \otimes \cL$ for a line bundle $\cL$ 
recovers the above definitions. 
This notion is also useful for other choices of $\alpha$, see~\cite[\S2]{kuznetsov2018residual}.
\end{remark}

The following shows that giving a Lefschetz center is equivalent to giving 
Lefschetz decompositions with suitably admissible components. 
This is useful in practice for constructing Lefschetz centers.

\begin{lemma}
\label{lemma-lef-center-from-decomp}
Let $T$ be a scheme over $S$ with a line bundle $\cL$. 
Let $\cA$ be a $T$-linear category with $S$-linear semiorthogonal decompositions 
\begin{alignat}{2}
\label{lcfd-1} \cA & = \langle \cA_0, \cA_1 \otimes \cL, \dots, \cA_{m-1} \otimes \cL^{m-1} \rangle, 
 &  &  \text{where }   {\cA_0 \supset \cA_1 \supset \dots \supset \cA_{m-1}} , \\ 
\label{lcfd-2}  \cA & = \langle \cA_{1-m} \otimes \cL^{1-m}, \dots, \cA_{-1} \otimes \cL^{-1}, \cA_0 \rangle,
  \quad & & \text{where } {\cA_{1-m} \subset \dots \subset \cA_{-1} \dots \subset \cA_0}.
\end{alignat} 
Then the categories $\cA_i$, $|i| < m$, satisfy \eqref{Ai-igeq0} and \eqref{Ai-ileq0}, and 
the categories defined by \eqref{Ai-igeq0} and \eqref{Ai-ileq0} for $|i| \geq m$ vanish. 
Hence if $\cA_i \subset \cA$ is right admissible for $i \geq 0$ and left admissible for $i \leq 0$, 
then $\cA_0 \subset \cA$ is a Lefschetz center. 
\end{lemma}

\begin{proof}
Note that $\cA_0$ is left admissible by the first semiorthogonal decomposition and right admissible by the second.
The rest follows from~\cite[Lemma 6.3]{NCHPD}. 
\end{proof}

\begin{remark}
\label{remark:lc-sp}
If $\cA$ as in Lemma~\ref{lemma-lef-center-from-decomp} is smooth and proper over $S$, 
then in order for a subcategory $\cA_0 \subset \cA$ to be a Lefschetz center, 
it is enough to give only one of the semiorthogonal decompositions~\eqref{lcfd-1} or \eqref{lcfd-2}. 
This follows from \cite[Lemmas 4.15 and 6.3]{NCHPD}.
\end{remark}

Recall the chains~\eqref{eq:lefschetz-chain} of Lefschetz components of~$\cA$.
For $i \ge 1$ the $i$-th \emph{right primitive component}~$\fa_i$ of a Lefschetz center is 
defined as the right orthogonal to $\cA_{i+1}$ in $\cA_i$, i.e.  
\begin{equation*}
\fa_i = \cA_{i+1}^\perp \cap \cA_{i},
\end{equation*}
so that 
\begin{equation}
\label{eq:ca-fa-ca-plus}
\cA_i = \llangle \fa_i, \cA_{i+1} \rrangle =  \llangle \fa_i, \fa_{i+1}, \dots, \fa_{m-1} \rrangle. 
\end{equation}
Similarly, for $i \le {-1}$ the $i$-th \emph{left primitive component} $\fa_i$ of a Lefschetz center is 
the left orthogonal to $\cA_{i-1}$ in $\cA_i$, i.e. 
\begin{equation*}
\fa_i = {}^\perp\cA_{i-1} \cap \cA_{i}, 
\end{equation*}
so that 
\begin{equation}
\label{eq:ca-fa-ca-minus}
\cA_i = \llangle \cA_{i-1}, \fa_{i} \rrangle = \llangle \fa_{1-m}, \dots, \fa_{i-1}, \fa_{i} \rrangle. 
\end{equation}
For $i = 0$, we have both right and left primitive components, defined by  
\begin{equation*}
\fa_{+0} = \cA_1^{\perp} \cap \cA_0 \quad \text{and} \quad 
\fa_{-0} = {}^\perp\cA_{-1} \cap \cA_0.
\end{equation*}
These are related by the formula $\fa_{-0} = \fa_{+0} \otimes \cL$, see \cite[Remark 6.4]{NCHPD}. 

To simplify formulas, we sometimes abusively write $\fa_0$ to mean either $\fa_{-0}$ or $\fa_{+0}$, when 
it is clear from context which is intended. 
So for instance the formulas~\eqref{eq:ca-fa-ca-plus} and~\eqref{eq:ca-fa-ca-minus} make sense for $i = 0$, and  
the right Lefschetz decomposition of $\cA$ in terms of primitive categories can be written as
\begin{equation}
\label{eq:ca-primitive-plus}
\begin{aligned}
\cA & = \llangle \fa_0, \dots ,\fa_{m-1}, \fa_1 \otimes \cL, \dots, \fa_{m-1} \otimes \cL, \dots, \fa_{m-1} \otimes \cL^{m-1} \rrangle \\
& = \llangle \fa_i \otimes \cL^t \rrangle_{0 \le t \le i \le m-1}, 
\end{aligned}
\end{equation}
while the left Lefschetz decomposition can be written as
\begin{equation}
\label{eq:ca-primitive-minus}
\begin{aligned}
\cA & = \llangle \fa_{1-m} \otimes \cL^{1-m}, \dots, \fa_{1-m} \otimes \cL^{-1}, \dots, \fa_{-1} \otimes \cL^{-1}, \fa_{m-1}, \dots, \fa_0 \rrangle  \\
& = \llangle \fa_i \otimes \cL^t \rrangle_{1-m \leq i \leq t \leq 0}. 
\end{aligned}
\end{equation}

\begin{definition} 
\label{definition-lc} 
A \emph{Lefschetz category} $\cA$ over $\bP(V)$ is a 
$\bP(V)$-linear category equipped with a Lefschetz center $\cA_0 \subset \cA$ with respect to $\cO(H)$. 
The \emph{length} of $\cA$ is the length of its Lefschetz decompositions, and 
is denoted by $\length(\cA)$. 

Given Lefschetz categories $\cA$ and $\cB$ over $\bP(V)$, an \emph{equivalence of Lefschetz categories} or a \emph{Lefschetz equivalence}  
is a $\bP(V)$-linear equivalence $\cA \simeq \cB$ which induces an $S$-linear 
equivalence $\cA_0 \simeq \cB_0$ of centers.  
\end{definition}

In this paper, we will be concerned with proving equivalences of Lefschetz categories. 
For this, the following criterion will be useful. 

\begin{lemma}
\label{lemma-equivalence-lef-cat}
Let $\phi \colon \cA \to \cB$ be a $\bP(V)$-linear functor between Lefschetz categories $\cA$ and $\cB$ 
over~$\bP(V)$. 
Assume: 
\begin{enumerate}
\item $\phi$ induces an equivalence $\cA_0 \simeq \cB_0$.
\item $\phi$ admits a left adjoint $\phi^* \colon \cB \to \cA$. 
\item $\phi^*(\cB_0) \subset \cA_0$.
\end{enumerate} 
Then $\phi$ is an equivalence of Lefschetz categories. 
\end{lemma}

\begin{remark}
A similar criterion is true if we replace the left adjoint $\phi^*$ with the right adjoint~$\phi^!$. 
\end{remark} 

\begin{proof}
First we show $\phi$ is fully faithful. 
Consider the counit morphism 
\begin{equation*}
\phi^* \circ \phi \to \id_{\cA} .
\end{equation*}
This is a morphism in the category of $\bP(V)$-linear functors 
$\Fun_{\Perf(\bP(V))}(\cA, \cA)$ (note that~$\phi^*$ is a $\bP(V)$-linear 
functor \cite[Lemma 2.11]{NCHPD}). 
Let $\psi \colon \cA \to \cA$ be the $\bP(V)$-linear functor given by 
the cone of this morphism. 
Then the claim that $\phi$ is fully faithful is equivalent to $\psi$ being 
the zero functor. 
Our assumptions imply $\psi$ vanishes on the subcategory $\cA_0 \subset \cA$. 
By $\bP(V)$-linearity it follows that $\psi$ vanishes on 
$\cA_0(iH)$, and hence on $\cA_i(iH) \subset \cA_0(iH)$, for all $i$. 
But then the (right or left) Lefschetz decomposition of $\cA$ implies  
$\psi$ is the zero functor. 

Since $\phi$ is fully faithful, its image is a $\bP(V)$-linear triangulated subcategory of $\cB$. 
By assumption this image contains $\cB_0$, and hence by $\bP(V)$-linearity it contains $\cB_i(iH) \subset \cB_0(iH)$ for all $i$. 
But then the (right or left) Lefschetz decomposition of $\cB$ implies $\phi$ is essentially surjective. 
\end{proof}

For HPD we will need to consider Lefschetz categories that satisfy 
certain ``strongness'' and ``moderateness'' conditions, defined below. 

\begin{definition}
\label{definition:strong}
A Lefschetz category $\cA$ is called \emph{right strong} if all of its right primitive components 
$\fa_{+0}, \fa_i$, $i \geq 1$, are admissible in $\cA$, 
\emph{left strong} if all of its left primitive components~$\fa_{-0}, \fa_{i}$, $i \leq -1$, are admissible in $\cA$, 
and \emph{strong} if all of its primitive components are admissible. 
\end{definition} 

\begin{remark}
\label{remark:smooth-strong}
If $\cA$ is smooth and proper over $S$, then any Lefschetz structure on $\cA$ is automatically strong, 
see \cite[Remark 6.7]{NCHPD}. 
\end{remark}

By \cite[Corollary 6.19(1)]{NCHPD}, the length of a Lefschetz category $\cA$ over $\bP(V)$ satisfies 
\begin{equation}
\label{length-leq-rank}
\length(\cA) \leq \rank(V). 
\end{equation}

\begin{definition}
\label{def:moderateness}
A Lefschetz category $\cA$ over $\bP(V)$ is called \emph{moderate} if its length satisfies the strict inequality
\begin{equation*}
\length(\cA) < \rank(V) . 
\end{equation*} 
\end{definition}

\begin{remark}
\label{remark-moderate} 
Moderateness of a Lefschetz category $\cA$ over $\bP(V)$ is a very mild condition. 
Indeed, we can always embed $V$ into a vector bundle $V'$ of larger rank, 
e.g. $V' = V \oplus \cO$, and then $\cA$ is a moderate Lefschetz category 
over $\bP(V')$. 
Moreover, essentially all Lefschetz categories that arise in practice are moderate; 
we do not know any interesting immoderate examples. 
\end{remark}

There are many examples of interesting Lefschetz categories, 
some of which are listed in~\S\ref{section:applications}; see also ~\cite{kuznetsov2014semiorthogonal} for a survey. 
Here we recall one simple example, which is just the relative version of Example~\ref{example-linear-HPD}.

\begin{example} 
\label{example-projective-bundle-lc}
Let $W \subset V$ be a subbundle of rank $m$. 
The morphism $\bP(W) \to \bP(V)$ induces a $\bP(V)$-linear structure on $\Perf(\bP(W))$. 
Note that pullback along the projection $\bP(W) \to S$ gives an embedding $\Perf(S) \subset \Perf(\bP(W))$. 
The category $\Perf(\bP(W))$ is a strong Lefschetz category over $\bP(V)$ with center 
$\Perf(S)$;  
the corresponding right and left Lefschetz decompositions are 
given by Orlov's projective bundle formulas: 
\begin{align*} 
\Perf(\bP(W)) & = \llangle \Perf(S), \Perf(S)(H), \dots, \Perf(S)((m-1)H)  \rrangle , \\
\Perf(\bP(W)) & = \llangle \Perf(S)((1-m)H), \dots, \Perf(S)(-H), \Perf(S) \rrangle . 
\end{align*} 
We call this the \emph{standard Lefschetz structure} on~$\bP(W)$. 
Note that the length of $\Perf(\bP(W))$ is $m$, so it is a moderate Lefschetz category as long as~$W \neq V$. 
\end{example}

The key property of a Lefschetz category is that its Lefschetz decomposition behaves well 
under passage to linear sections. 
Recall that $\cA_{\bP(L)} = \cA \otimes_{\Perf(\bP(V))} \Perf(\bP(L))$ 
denotes the base change of a $\bP(V)$-linear category $\cA$ along a morphism $\bP(L) \to \bP(V)$.

\begin{lemma}
\label{lemma-linear-section-lc} 
Let $\cA$ be a Lefschetz category over $\bP(V)$ of length $m$. 
Let~$L \subset V$ be a subbundle of corank $s$.  
Then the functor
\begin{equation*}
\cA \to \cA_{\bP(L)} 
\end{equation*}
induced by pullback along $\bP(L) \to \bP(V)$ is fully faithful on the 
Lefschetz components $\cA_i \subset \cA$ for $|i| \geq s$. 
Moreover, denoting their images by the same symbols, 
there are semiorthogonal decompositions 
\begin{align*}
\label{eq:lefschetz-restricted} \cA_{\bP(L)} & = \llangle \cK_L(\cA), \cA_s(H), \dots, \cA_{m-1}((m-s)H) \rrangle ,  \\ 
\cA_{\bP(L)} & = \llangle \cA_{1-m}((s-m)H), \dots,\cA_{-s}(-H), \cK'_{L}(\cA) \rrangle. 
\end{align*}
\end{lemma} 

\begin{proof}
This is a special case of \cite[Lemmas 6.20 and 6.22(3)]{NCHPD}. 
\end{proof}

\begin{remark}
The analogy between Lemma~\ref{lemma-linear-section-lc} and the Lefschetz 
hyperplane theorem is the source of the terminology ``Lefschetz category''.  
The main theorem of HPD (Theorem \ref{theorem-HPD} below) 
describes the categories $\cK_L(\cA)$ and $\cK'_{L}(\cA)$ in terms of the HPD category of~$\cA$. 
\end{remark}

\subsection{The HPD category} 
\label{subsection-HPD}

Let 
\begin{equation*}
\delta \colon \bH(\bP(V)) \to \bP(V) \stimes \bP(\vV). 
\end{equation*} 
be the natural incidence divisor.
We think of $\bH(\bP(V))$ as the universal hyperplane in $\bP(V)$.
  
If $X$ is a scheme with a morphism $X \to \bP(V)$, 
then the universal hyperplane section of $X$ is defined by 
\begin{equation*}
\bH(X) = X \times_{\bP(V)} \bH(\bP(V)). 
\end{equation*} 
This definition extends directly to linear categories as follows.  
 
\begin{definition}
Let $\cA$ be a $\bP(V)$-linear category. 
The \emph{universal hyperplane section} of $\cA$ is defined by 
\begin{equation*}
\bH(\cA) = \cA \otimes_{\Perf(\bP(V))} \Perf(\bH(\bP(V))). 
\end{equation*} 
\end{definition} 

The above definition is compatible with the geometric one in the following sense: 
if $X$ is a scheme over $\bP(V)$, then by Theorem~\ref{theorem-bzfn} there is an equivalence
$\bH(\Perf(X)) \simeq \Perf(\bH(X))$. 
We sometimes use the more elaborate notation 
\begin{equation*}
\bH(X/\bP(V)) = \bH(X) 
\qquad\text{and}\qquad 
\bH(\cA/\bP(V)) = \bH(\cA) 
\end{equation*}
to emphasize the universal hyperplane section is being taken with respect to $\bP(V)$.

The natural embedding $\delta$ includes into the following diagram of morphisms
\begin{equation}
\label{eq:h-diagram}
\vcenter{\xymatrix@C=6em{
& \bH(\bP(V)) \ar[dl]_\pi \ar[d]^\delta \ar[dr]^h 
\\
\bP(V) &
\bP(V) \times \bP(\vV) \ar[l]^-{\pr_1} \ar[r]_-{\pr_2} & 
\bP(\vV).
}}
\end{equation}
Here we deviate slightly from the notation of \cite{NCHPD}, where the morphisms $\pi$, $\delta$, 
and $h$ are instead denoted $p$, $\iota$, and $f$. 
All schemes in the diagram are smooth and projective over $S$, hence Remark~\ref{remark:good-morphism} applies to all morphisms.
For a $\bP(V)$-linear category $\cA$, 
it follows from Theorem~\ref{theorem-bzfn} that there are canonical identifications 
\begin{equation*} 
\cA \otimes_{\Perf(\bP(V))} \Perf(\bP(V) \times \bP(\vV)) \simeq 
\cA \sotimes \Perf(\bP(\vV)), 
\quad
\cA \otimes_{\Perf(\bP(V))} \Perf(\bP(V)) \simeq \cA, 
\end{equation*} 
by which we will regard the functors induced by morphisms in~\eqref{eq:h-diagram} as functors 
\begin{equation*}
\delta_* \colon \bH(\cA) \to \cA \sotimes \Perf(\bP(\vV)),
\quad 
\pi_* \colon \bH(\cA) \to \cA,
\end{equation*}
and so on.

The next definition differs from the original  
in~\cite{kuznetsov-hpd}, but is equivalent to it, as Lemma~\ref{lemma:hpd-sod} below shows. 
The advantage of this definition is that it is more symmetric 
(with respect to the left and the right Lefschetz decompositions of $\cA$). 

\begin{definition} 
\label{definition-HPD-category}
Let $\cA$ be a Lefschetz category over $\bP(V)$.
Then the \emph{HPD category} $\cAd$ of $\cA$ is the full $\bP(\vV)$-linear subcategory of 
$\bH(\cA)$ defined by 
\begin{equation} 
\label{eq:hpd-category}
\cAd = \set{ C \in \bH(\cA) \st \delta_*(C) \in \cA_0 \sotimes \Perf(\bP(\vV)) }.
\end{equation}
We sometimes use the notation 
\begin{equation*}
(\cA/\bP(V))^{\hpd}  = \cA^{\hpd}
\end{equation*}
to emphasize the dependence on the $\bP(V)$-linear structure. 
\end{definition}

\begin{remark}
The HPD category $\cAd$ depends on the choice of the Lefschetz center $\cA_0 \subset \cA$, 
although this is suppressed in the notation. 
For instance, for the ``stupid'' Lefschetz center~$\cA_0 = \cA$ we have $\cAd = \bH(\cA)$. 
\end{remark}

A less trivial example of HPD is the following relative version of Example~\ref{example:HPD-linear}.

\begin{example}
\label{ex:categorical-linear-hpd}
Consider the Lefschetz category $\Perf(\bP(W))$ of Example~\ref{example-projective-bundle-lc}, 
and assume~$0 \subsetneq W \subsetneq V$.
Then by \cite[Corollary~8.3]{kuznetsov-hpd} 
there is a Lefschetz equivalence
\begin{equation*}
\Perf(\bP(W))^\hpd \simeq \Perf(\bP(W^\perp)).
\end{equation*}
\end{example}

\begin{remark}
\label{remark-left-HPD} 
In fact, following \cite[\S7.1]{NCHPD}, the category $\cAd$ should more precisely be 
called the \emph{right HPD category} of $\cA$. 
Indeed, there is also a \emph{left HPD category} $\dcA$, which is defined 
by replacing the right adjoint $\delta_*$ to $\delta^*$ with the left adjoint 
$\delta_!$ in \eqref{eq:hpd-category}. 
As shown in~\cite[Lemma~7.2]{NCHPD}, there is a $\bP(\vV)$-linear 
equivalence $\cAd \simeq \dcA$. 
Under mild hypotheses these categories are endowed with natural Lefschetz 
structures, see \cite[\S7.2]{NCHPD}; 
for $\cAd$ this is part of Theorem~\ref{theorem-HPD} below. 
Under stronger hypotheses we can show that there is 
a Lefschetz equivalence $\cAd \simeq \dcA$ \cite[Proposition 7.12]{NCHPD}, 
but in general we do not know if one exists. 
In this paper, we will deal almost exclusively with $\cAd$, and 
therefore simply refer to it as the HPD category. 
All of our results can be translated directly to the ``left HPD'' setting. 
\end{remark}

\begin{remark}
\label{remark-HPD-smooth-proper}
If $\cA$ is a Lefschetz category over $\bP(V)$ which is smooth and proper over~$S$, 
then the HPD category $\cAd$ is also smooth and proper over $S$ \cite[Lemma 7.18]{NCHPD}. 
This is an instance of the ``homological smoothness principle'' of homological projective geometry, 
see~\S\ref{section-homological-projective-geometry}.
\end{remark}

Sometimes it is convenient to describe the HPD category 
$\cA^{\hpd}$ in terms of the right or left Lefschetz decompositions of $\cA$ as follows.

\begin{lemma}
\label{lemma:hpd-sod}
Let $\cA$ be a Lefschetz category over $\bP(V)$ of length $m$. 
Then there are $\bP(\vV)$-linear semiorthogonal decompositions 
\begin{align} 
\label{HC-sod}
\hspace{-.64em}
\bH(\cA) & = \llangle \cAd, 
\delta^*(\cA_1(H) \sotimes \Perf(\bP(\vV))), 
\dots, 
\delta^*(\cA_{m-1}((m-1)H) \sotimes \Perf(\bP(\vV))) \rrangle , \\ 
\label{HC-sod-left}
\hspace{-.64em}
\bH(\cA) & = \llangle 
\delta^!(\cA_{1-m}((1-m)H) \sotimes \Perf(\bP(\vV))), 
\dots, 
\delta^!(\cA_{-1}(-H) \sotimes \Perf(\bP(\vV))),
\cAd
\rrangle, 
\end{align}
where the functors $\delta^*, \delta^! \colon \cA \sotimes \Perf(\bP(\vV)) \to \bH(\cA)$ are fully faithful 
on the categories to which they are applied. 
In particular, $\cAd$ is an admissible subcategory in $\bH(\cA)$ 
and its inclusion functor \mbox{$\gamma \colon \cAd \to \bH(\cA)$} has both left and right adjoints 
$\gamma^*,\gamma^! \colon \bH(\cA) \to \cAd$. 
\end{lemma}

\begin{proof}
This holds by \cite[Definition 7.1 and Lemma 7.2]{NCHPD}. 
Note that admissibility is by definition the existence of adjoint functors to the inclusion. 
\end{proof} 

By~\cite[Lemma 7.3]{NCHPD}, if $\cA$ is a moderate Lefschetz category the composition 
\begin{equation*}
\cA \xrightarrow{\, \pi^* \,} \bH(\cA) \xrightarrow{\, \gamma^* \,} \cAd 
\end{equation*}
is fully faithful on the center $\cA_0 \subset \cA$; in this case, we define
\begin{equation}
\label{Ad0}
\cAd_0 = \gamma^*\pi^*(\cA_0). 
\end{equation} 
For later use, we note the following.

\begin{lemma}
\label{lemma-A0-Ad0-equivalence} 
For a moderate Lefschetz category $\cA$ over $\bP(V)$, 
the functors $\pi_* \circ \gamma \colon \cAd \to \cA$ 
and $\gamma^* \circ \pi^* \colon \cA \to \cAd$ induce mutually inverse 
equivalences between 
$\cAd_0 \subset \cAd$ and $\cA_0 \subset \cA$. 
\end{lemma} 

\begin{proof}
Since $\gamma^* \circ \pi^*$ is fully faithful on $\cA_0$ with image $\cAd_0$, 
this follows from the fact that the image of the right adjoint $\pi_* \circ \gamma$ is $\cA_0$ by \cite[Lemma~7.11]{NCHPD}. 
\end{proof}

The main theorem of HPD, recalled below, shows in particular that $\cAd_0 \subset \cAd$ 
is a Lefschetz center under certain hypotheses. 
This theorem was originally proved in \cite{kuznetsov-hpd} in the ``commutative'' case. 
We need the following ``noncommutative'' version from~\cite[Theorem~8.7]{NCHPD}.  

Let $H'$ denote the relative hyperplane class on $\bP(V^{\svee})$ 
such that~$\cO(H') = \cO_{\bP(\vV)}(1)$. 
Recall the definition~\eqref{eq:perp-bundle} of the orthogonal of a subbundle.

\begin{theorem}
\label{theorem-HPD} 
Let $\cA$ be a right strong, moderate Lefschetz category over $\bP(V)$. 
Then: 
\begin{enumerate}
\item 
\label{Cd-ld}
$\cAd$ is a left strong, moderate Lefschetz category over $\bP(\vV)$ with center $\cAd_0 \subset \cAd$ and 
length given by 
\begin{equation*}
\length(\cAd) = \rank(V) - \# \{ \, i \geq 0 \mid \cA_i = \cA_0 \, \}. 
\end{equation*}

\item \label{HPD-linear}  
Let $L \subset V$ be a subbundle and let $L^{\perp} \subset \vV$ be 
its orthogonal. 
Set 
\begin{equation*}
r = \rank(L), \quad s = \rank(L^{\perp}), \quad m = \length(\cA),  \quad n = \length(\cA^\hpd). 
\end{equation*} 
Then there are semiorthogonal decompositions 
\begin{align*}
\cA_{\bP(L)} 
& = \llangle \cK_L(\cA), 
\cA_s(H) , 
\dots, 
\cA_{m-1}((m-s)H)  \rrangle , 
\\
\cAd_{\bP(L^{\perp})} 
& = 
\llangle 
\cAd_{1-n}((r-n)H') , \dots, \cAd_{-r}(-H') , 
\cK'_{L^{\perp}}(\cAd) 
\rrangle , 
\end{align*}
and an $S$-linear equivalence $\cK_L(\cA) \simeq \cK'_{L^{\perp}}(\cAd)$. 

\end{enumerate}
\end{theorem} 

\begin{remark}
The Lefschetz components $\cAd_j \subset \cAd$ can be expressed explicitly in terms of the 
right primitive components of $\cA$, see \cite[\S7.2]{NCHPD}. 
\end{remark}

\begin{remark}
\label{remark-HPD-duality}
In the setup of Theorem~\ref{theorem-HPD}, 
there are equivalences ${^\hpd}(\cAd) \simeq \cA$ and $({^\hpd}\cA)^{\hpd} \simeq \cA$ 
of Lefschetz categories over $\bP(V)$, 
where ${^\hpd}(-)$ denotes the left HPD operation, 
see Remark~\ref{remark-left-HPD} and~\cite[Theorem 8.9]{NCHPD}.
This property justifies HPD being called a ``duality''.
\end{remark}

\subsection{Characterization of the HPD category} 
\label{subsection-HPD-characterization} 
Given $\cA$ as in Theorem~\ref{theorem-HPD}, we will need later 
a characterization of the Lefschetz category $\cAd$ in terms of the 
functor   
\begin{equation*}
\pi_* \colon \bH(\cA) \to \cA. 
\end{equation*}
For this, we must characterize $\cAd$ and its Lefschetz center $\cAd_0 \subset \cAd$ 
in terms of $\pi_*$. We handle the first in Lemma~\ref{lemma-characterization-Cd}, and 
the second in Proposition~\ref{proposition-characterization-Ad0}. 
Recall that given a $T$-linear category $\cC$, we write $C \otimes F$ 
for the action of an object $F \in \Perf(T)$ on an object $C \in \cC$. 

\begin{lemma}
\label{lemma-characterization-Cd} 
Let $\cA$ be a Lefschetz category over $\bP(V)$. 
Then $\cAd$ is the full $\bP(\vV)$-linear subcategory of $\bH(\cA)$ given by
\begin{equation*}
\cAd = \set{ C \in \bH(\cA) \st \pi_*(C \otimes h^*F) \in \cA_0 \text{ for all } F \in \Perf(\bP(\vV))} . 
\end{equation*}
\end{lemma} 

\begin{proof}
Consider the diagram~\eqref{eq:h-diagram} and its base change from $\bP(V)$ to $\cA$.
By~\cite[Lemma 3.18]{NCHPD} the defining property~\eqref{eq:hpd-category} of $\cAd$ holds for $C \in \bH(\cA)$
if and only if 
\begin{equation*}
\pr_{1*}(\delta_*(C) \otimes \pr_2^*(F)) \in \cA_0  \text{ for all } F \in \Perf(\bP(\vV)). 
\end{equation*}
But $\pr_1 \circ \delta = \pi$ and $\pr_2 \circ \delta = h$, hence the result follows from projection formula 
\begin{equation*}
\pr_{1*}(\delta_*(C) \otimes \pr_2^*(F)) \simeq
\pi_*(C \otimes h^*(F)).\qedhere
\end{equation*}
\end{proof}

The following related result will also be needed later. 

\begin{lemma} 
\label{lemma-bH-sod}
Let $\cA$ be a $\bP(V)$-linear category with a $\bP(V)$-linear semiorthogonal 
decomposition $\cA = \llangle \cA', \cA'' \rrangle$. 
Then there is a $\bP(\vV)$-linear semiorthogonal decomposition 
\begin{equation*}
\bH(\cA) = \llangle \bH(\cA') , \bH(\cA'') \rrangle, 
\end{equation*}
where $\bH(\cA')$ can be described as the full subcategory of $\bH(\cA)$ given by 
\begin{equation*}
\bH(\cA') = \set{C \in \bH(\cA) \st \pi_*(C \otimes h^*F) \in \cA' \text{ for all } F \in \Perf(\bP(\vV))} , 
\end{equation*} 
and $\bH(\cA'')$ is given analogously. 
\end{lemma}

\begin{proof}
The claimed semiorthogonal decomposition of $\bH(\cA)$ 
holds by Lemma~\ref{lemma-sod-tensor}. 
By \cite[Lemma 3.18]{NCHPD}, for $C \in \bH(\cA)$ we have 
$C \in \bH(\cA')$ if and only if 
\begin{equation}
\label{eq:condition}
\pi_*(C \otimes G) \in \cA' \text{ for all } G \in \Perf(\bH(\bP(V))). 
\end{equation} 
Since $\delta \colon \bH(\bP(V)) \to \bP(V) \stimes \bP(\vV)$ is a closed 
embedding, $\Perf(\bH(\bP(V)))$ is thickly generated by objects in the image 
of $\delta^*$. 
Hence by Lemma~\ref{lemma-generators-box-tensor}, the category $\Perf(\bH(\bP(V)))$ is thickly generated 
by the objects $\delta^*(E \boxtimes F)$ 
for $E \in \Perf(\bP(V)), F \in \Perf(\bP(\vV))$. 
It follows that~\eqref{eq:condition} is equivalent to 
\begin{equation*}
\pi_*(C \otimes  \delta^*(E \boxtimes F)) \in \cA' \text{ for all } 
E \in \Perf(\bP(V)), F \in \Perf(\bP(\vV)). 
\end{equation*}
Note that $\delta^*(E \boxtimes F) \simeq \pi^*(E) \otimes h^*(F)$, hence 
\begin{equation*}
\pi_*(C \otimes  \delta^*(E \boxtimes F)) \simeq \pi_*(C \otimes h^*F) \otimes E. 
\end{equation*}
Since $\cA'$ is $\bP(V)$-linear, the above condition is thus equivalent to 
\begin{equation*}
\pi_*(C \otimes h^*F) \in \cA' \text{ for all } F \in \Perf(\bP(\vV)). \qedhere
\end{equation*}
\end{proof}

To characterize $\cAd_0 \subset \cAd$ we need to introduce some notation. 
Consider the tautological inclusion 
\begin{equation*}
\cO_{\bP(V)}(-H) \to V \otimes \cO_{\bP(V)}
\end{equation*}
on $\bP(V)$, 
and the tautological surjection 
\begin{equation*}
V \otimes \cO_{\bP(\vV)} \to \cO_{\bP(\vV)}(H')
\end{equation*}
on $\bP(\vV)$. 
By the definition of $\bH(\bP(V))$, the composition 
\begin{equation*}
\cO_{\bH(\bP(V))}(-H) \to V \otimes \cO_{\bH(\bP(V))} \to \cO_{\bH(\bP(V))}(H') 
\end{equation*} 
of the pullbacks of these morphisms to $\bH(\bP(V))$ vanishes, and hence can 
be considered as a complex concentrated in degrees $[-1,1]$. 
By construction, this complex has cohomology concentrated in degree $0$, 
i.e. it is a \emph{monad};  
we define $\cM$ as the degree $0$ cohomology sheaf, 
\begin{equation*}
\cM \simeq \left\{ \cO_{\bH(\bP(V))}(-H) \to V \otimes \cO_{\bH(\bP(V))} \to \cO_{\bH(\bP(V))}(H') \right\},
\end{equation*}
which is a vector bundle of rank $N-2$ on $\bH(\bP(V))$,
where recall that~$N$ is the rank of~$V$.

\begin{lemma}
\label{lemma-HC-sod-Pbundle} 
Let $\cA$ be a $\bP(V)$-linear category. 
Then there is a semiorthogonal decomposition 
\begin{equation*}
\bH(\cA) = \llangle 
\pi^*(\cA)(-(N-2)H'), \dots, \pi^*(\cA)(-H'), \pi^*(\cA) 
\rrangle  . 
\end{equation*}
Moreover, for $0 \leq t \leq N-2$ the projection functor onto the $-tH'$ component 
\textup(regarded as a functor to $\cA$\textup) is given by 
\begin{equation*}
\eta_t \colon  \bH(\cA) \to \cA,
\quad
C \mapsto \pi_*(C \otimes \wedge^t \cM[t]). 
\end{equation*}
\end{lemma}

\begin{proof}
Define  
\begin{equation*}
\cK = \ker(V^{\svee} \otimes \cO_{\bP(V)} \to \cO_{\bP(V)}(H)). 
\end{equation*}
Then it is easy to see there is an isomorphism $\bH(\bP(V)) \cong \bP_{\bP(V)}(\cK)$, 
under which $H'$ corresponds to the tautological $\cO(1)$ line bundle 
and $\cM$ corresponds to $\Omega_{{\bP_{\bP(V)}(\cK)}/\bP(V)}(H')$. 
Hence we have the standard projective bundle semiorthogonal 
decomposition 
\begin{multline*}
\Perf(\bH(\bP(V))) =  \\
\llangle 
\pi^*(\Perf(\bP(V)))(-(N-2)H'), \dots, \pi^*(\Perf(\bP(V)))(-H'), \pi^*(\Perf(\bP(V))) 
\rrangle , 
\end{multline*}
whose projection functors for $0 \leq t \leq N-2$ are given by 
\begin{equation*}
\eta_t \colon \Perf(\bH(\bP(V))) \to \Perf(\bP(V)),
\qquad
F \mapsto \pi_*(F \otimes \wedge^t \cM[t])
\end{equation*}
(this is a relative version of the Beilinson spectral sequence, see the proof of~\cite[Theorem~2.6]{orlov1992projective}).
Now by Lemma~\ref{lemma-sod-tensor} the result follows by base change. 
\end{proof}

\begin{proposition}
\label{proposition-characterization-Ad0}
Let $\cA$ be a moderate Lefschetz category over $\bP(V)$.
Then $\cAd_0$ is the full subcategory of $\cAd$ given by 
\begin{equation*}
\cAd_0 = \set{ 
C \in \cAd \st 
\pi_*(\gamma(C) \otimes \wedge^t \cM) \in {}^\perp\cA_0 \text{ for all } t \geq 1
} , 
\end{equation*} 
where ${}^\perp\cA_0$ is the left orthogonal to the center $\cA_0 \subset \cA$. 
\end{proposition}

\begin{proof}
First note that since $\cM$ is a vector bundle of rank $N - 2$, for 
$C \in \cAd$ the condition  
\begin{equation*}
\pi_*(\gamma(C) \otimes \wedge^t \cM) \in {}^\perp\cA_0
\end{equation*} 
holds for all $t \geq 1$ if and only if it holds for $1 \leq t \leq N-2$. 
This condition is in turn equivalent to 
\begin{equation}
\label{eq:cone-condition}
\operatorname{Cone}(\pi^*\pi_*\gamma(C) \to \gamma(C)) \in 
\Big\langle \pi^*{\left( \cA_i(iH) \right)} \otimes \cO(-tH') \Big\rangle_{1 \leq t \leq N-2, \, 1 \leq i \leq m-1}.
\end{equation} 
Indeed, this follows from the form of the projection functors for the 
semiorthogonal decomposition of Lemma~\ref{lemma-HC-sod-Pbundle}, together with  
the equality  
\begin{equation*}
{}^\perp\cA_0 = \langle \cA_1(H), \dots, \cA_{m-1}((m-1)H) \rangle
\end{equation*}
where $\cA_i$ are the components of the right Lefschetz decomposition~\eqref{eq:right-decomposition} of $\cA$. 
It remains to show that \eqref{eq:cone-condition} is equivalent to $C \in \cAd_0$. 

Suppose~\eqref{eq:cone-condition} holds for $C \in \cAd$. 
Then $\gamma^*$ kills the left side of~\eqref{eq:cone-condition},
since by~\eqref{HC-sod} all components in the right side 
are contained in~${}^{\perp}(\cAd)$. 
Hence $\gamma^* \pi^* \pi_* \gamma(C) \simeq \gamma^* \gamma(C) \simeq C$. 
But~$\pi_* \gamma(C) \in \cA_0$ by Lemma~\ref{lemma-characterization-Cd}, 
so we conclude $C \in \cAd_0$ by the definition~\eqref{Ad0} of $\cAd_0$. 

Conversely, assume $C \in \cAd$ lies in $\cAd_0$, i.e. $C = \gamma^* \pi^*(D)$ 
for some $D \in \cA_0$. 
By Lemma~\ref{lemma-A0-Ad0-equivalence} we have 
$\pi_* \gamma  \gamma^* \pi^*(D) \simeq D$. 
Under this isomorphism, the morphism $\pi^* \pi_* \gamma(C) \to \gamma(C)$
is identified with the canonical morphism 
\begin{equation*}
\pi^*(D) \to \gamma \gamma^* \pi^*(D) , 
\end{equation*}
whose cone is nothing but $\rR_{\cAd}(\pi^*(D))[1]$, where $\rR_{\cAd}$ is the right mutation functor 
through the subcategory~$\cAd \subset \bH(\cA)$. 
But by~\cite[Lemma 7.8]{NCHPD} (or \cite[Lemma~5.6]{kuznetsov-hpd} in the commutative case) 
the object $\rR_{\cAd}(\pi^*(D))$ lies in the subcategory 
\begin{equation*}
\Big\langle \pi^*{\left( \cA_i(iH) \right)} \otimes \cO(-tH') \Big\rangle_{1 \leq t \leq m-1, \, 1 \leq i \leq m-t} \subset \bH(\cA) .
\end{equation*}
Note that $m-1 \leq N-2$ since $\cA$ is a moderate Lefschetz category, so 
this subcategory is contained in the right side of~\eqref{eq:cone-condition}, and hence~\eqref{eq:cone-condition} holds. 
This completes the proof.  
\end{proof}


\section{Categorical joins} 
\label{section-categorical-joins} 

In this section, we introduce the categorical join of two Lefschetz categories, 
which in the commutative case was briefly described in~\S\ref{subsection:intro-cat-joins}.   
First in~\S\ref{subsection-resolved-joins} we define the resolved join of two categories linear 
over projective bundles, by analogy with the canonical resolution of singularities 
of the classical join of two projective schemes. 
In~\S\ref{subsection-categorical-joins} we define the categorical join of two Lefschetz categories
as a subcategory of the resolved join, and prove some basic 
properties of this construction. 
In~\S\ref{subsection-bc-cat-joins} we study base changes of categorical joins, 
and in particular show that categorical and resolved joins agree away from 
the ``exceptional locus'' of the resolved join. 
Finally, in~\S\ref{subsection-ld-cJ} we construct a canonical Lefschetz 
structure on the categorical join of two Lefschetz categories. 

We fix nonzero vector bundles $V_1$ and $V_2$ on $S$, and write 
$H_1, H_2$, and $H$ for the relative hyperplane classes on $\bP(V_1), \bP(V_2)$, 
and $\bP(V_1 \oplus V_2)$. 

\subsection{Resolved joins} 
\label{subsection-resolved-joins}
Let $X_1 \to \bP(V_1)$ and $X_2 \to \bP(V_2)$ be morphisms of schemes. 
The \emph{resolved join} of $X_1$ and $X_2$ is defined as the $\bP^1$-bundle 
\begin{equation*}
\tJ(X_1, X_2) = \bP_{X_1 \stimes X_2}(\cO(-H_1) \oplus \cO(-H_2)). 
\end{equation*}
The canonical embedding of vector bundles 
\begin{equation*}
\cO(-H_1) \oplus \cO(-H_2) \hookrightarrow (V_1 \otimes \cO) \oplus (V_2 \otimes \cO) = 
(V_1 \oplus V_2) \otimes \cO 
\end{equation*} 
induces a morphism 
\begin{equation*}
\tJ(X_1, X_2) \to \bP(V_1 \oplus V_2). 
\end{equation*}
Recall from \S\ref{subsection:intro-cat-joins} that if $X_1 \to \bP(V_1)$ and 
$X_2 \to \bP(V_2)$ are embeddings, this morphism factors birationally through 
the classical join $\bJ(X_1, X_2) \subset \bP(V_1 \oplus V_2)$, and provides a 
resolution of singularities if $X_1$ and $X_2$ are smooth. 

Note that there is an isomorphism 
\begin{equation}
\label{tJ-base-change}
\tJ(X_1, X_2) \cong (X_1 \stimes X_2) \times_{(\bP(V_1) \stimes \bP(V_2))} \tJ(\bP(V_1), \bP(V_2)). 
\end{equation} 
Motivated by this, we call $\tJ(\bP(V_1), \bP(V_2))$ the \emph{universal resolved join}. 
Denote by 
\begin{equation*}
p \colon \tJ(\bP(V_1), \bP(V_2)) \to \bP(V_1) \stimes \bP(V_2)  
\end{equation*}
the canonical projection morphism, and by 
\begin{equation*}
f \colon \tJ(\bP(V_1), \bP(V_2)) \to \bP(V_1 \oplus V_2) 
\end{equation*}
the canonical morphism introduced above. 
Define 
\begin{align*}
\bE_1  & = \bP_{\bP(V_1) \stimes \bP(V_2)}(\cO(-H_1)) \cong \bP(V_1) \stimes \bP(V_2), \\  
\bE_2  & = \bP_{\bP(V_1) \stimes \bP(V_2)}(\cO(-H_2)) \cong \bP(V_1) \stimes \bP(V_2). 
\end{align*}
These are disjoint divisors in $\tJ(\bP(V_1), \bP(V_2))$, whose embeddings we denote by 
\begin{equation*}
\eps_1  \colon \bE_1 \to \tJ(\bP(V_1), \bP(V_2))    \quad \text{and} \quad  \eps_2 \colon \bE_2 \to \tJ(\bP(V_1), \bP(V_2)) . 
\end{equation*}
We have a commutative diagram
\begin{equation}
\label{diagram-tJ-projective-bundle} 
\vcenter{\xymatrix{
\bE_1 \ar[r]^-{\eps_1} \bijarbottom[dr] & \tJ(\bP(V_1), \bP(V_2)) \ar[d]_-{p} & \bijartop[dl]  \ar[l]_-{\eps_2} \bE_2  \\ 
& \bP(V_1) \stimes \bP(V_2). & 
}}
\end{equation}

The next result follows easily from the definitions.  

\begin{lemma}
\label{lemma-tJ-divisors}
The following hold: 
\begin{enumerate}
\item \label{f-blowup}
The morphism $f \colon \tJ(\bP(V_1), \bP(V_2)) \to \bP(V_1 \oplus V_2)$ 
is the blowup of $\bP(V_1 \oplus V_2)$ in the disjoint union $\bP(V_1) \sqcup \bP(V_2)$, 
with exceptional divisor $\bE_1 \sqcup \bE_2$.  

\item
The $\cO(1)$ line bundle for the $\bP^1$-bundle 
$p \colon \tJ(\bP(V_1), \bP(V_2)) \to \bP(V_1) \stimes \bP(V_2)$ is $\cO(H)$. 

\item \label{Ek-Hk-H}
We have the following equalities of divisors modulo linear equivalence: 
\begin{align*}
\bE_1 = H - H_2,  & \quad H|_{\bE_1} = H_1, \\ 
\bE_2 = H - H_1,  & \quad H|_{\bE_2} = H_2.
\end{align*} 

\item  \label{omega-p}
The relative dualizing complex of the morphism $p$ is given by 
\begin{equation*}
\omega_p  = \cO(H_1 + H_2 - 2H)[1].
\end{equation*}
\end{enumerate} 
\end{lemma}

Part~\eqref{f-blowup} of the lemma can be summarized by the blowup diagram  
\begin{equation}
\label{diagram-tJ-blowup}
\vcenter{\xymatrix{
\bE_1 \ar[r]^-{\eps_1} \ar[d] & \tJ(\bP(V_1), \bP(V_2)) \ar[d]_f & \bE_2 \ar[l]_-{\eps_2} \ar[d] \\
\bP(V_1) \ar[r]^-{} & \bP(V_1 \oplus V_2) & \bP(V_2) \ar[l]_-{}
}}
\end{equation} 
All schemes in the diagram are smooth and projective over $S$, hence Remark~\ref{remark:good-morphism} applies to all morphisms. 

Following~\eqref{tJ-base-change} we define the resolved join of categories linear over $\bP(V_1)$ and $\bP(V_2)$ 
by base change from the universal resolved join. 

\begin{definition}
\label{definition-tJ}
Let $\cA^1$ be a $\bP(V_1)$-linear category and $\cA^2$ a $\bP(V_2)$-linear category. 
The \emph{resolved join} of $\cA^1$ and $\cA^2$ is the category 
\begin{equation*}
\tJ(\cA^1, \cA^2) = 
\left( \cA^1 \sotimes \cA^2 \right) \otimes_{\Perf(\bP(V_1) \stimes \bP(V_2))} \Perf(\tJ(\bP(V_1), \bP(V_2))). 
\end{equation*}
Further, for $k=1,2$, we define 
\begin{equation*}
\bE_k(\cA^1, \cA^2) = \left( \cA^1 \sotimes \cA^2 \right) \otimes_{\Perf(\bP(V_1) \stimes \bP(V_2))} \Perf(\bE_k) .
\end{equation*}
\end{definition}

\begin{remark} 
\label{remark-Ek}
The isomorphism $\bE_k \cong  \bP(V_1) \times \bP(V_2)$ 
induces a canonical equivalence 
\begin{equation*}
\bE_k(\cA^1, \cA^2) \simeq \cA^1 \sotimes \cA^2.
\end{equation*} 
We identify these categories via this equivalence; 
in particular, below we will regard subcategories of the right side as subcategories of the left. 
Furthermore, using this identification the morphisms $\eps_k$ from~\eqref{diagram-tJ-projective-bundle} or~\eqref{diagram-tJ-blowup} 
induce functors between $\cA^1 \sotimes \cA^2$ and $\tJ(\cA^1,\cA^2)$.
\end{remark}

\begin{remark}
\label{remark-tJ-spaces}
If $X_1 \to \bP(V_1)$ and $X_2 \to \bP(V_2)$ are morphisms of schemes, then 
by the isomorphism~\eqref{tJ-base-change} and Theorem~\ref{theorem-bzfn} the resolved join satisfies 
\begin{equation*}
\tJ(\Perf(X_1), \Perf(X_2)) \simeq \Perf(\tJ(X_1, X_2)). 
\end{equation*} 
\end{remark}

Below we gather some elementary lemmas about resolved joins. 

Let $\gamma_k \colon \cA^k \to \cB^k$ be $\bP(V_k)$-linear functors.
Then we have a $\bP(V_1) \times \bP(V_2)$-linear functor 
\begin{equation*}
\gamma_1 \otimes \gamma_2 \colon \cA^1 \otimes \cA^2 \to \cB^1 \otimes \cB^2,
\end{equation*}
and by base change along the morphism $p$ we obtain a $\tJ(\bP(V_1),\bP(V_2))$-linear functor
\begin{equation}
\label{eq:tj-gg}
\tJ(\gamma_1,\gamma_2) \colon \tJ(\cA^1,\cA^2) \to \tJ(\cB^1,\cB^2).
\end{equation}

\begin{lemma}
\label{lemma:tj-functoriality}
Let $\gamma_k \colon \cA^k \to \cB^k$ be $\bP(V_k)$-linear functors.
There are commutative diagrams
\begin{equation*} 
\vcenter{\xymatrix@C=5em{
\tJ(\cA^1,\cA^2) \ar[r]^-{\tJ(\gamma_1,\gamma_2)} \ar[d]_{p_*} &
\tJ(\cB^1,\cB^2) \ar[d]_{p_*}
\\
\cA^1 \otimes \cA^2 \ar[r]^-{\gamma_1 \otimes \gamma_2} &
\cB^1 \otimes \cB^2 
}}
\quad\text{and}\quad 
\vcenter{\xymatrix@C=5em{
\tJ(\cA^1,\cA^2) \ar[r]^-{\tJ(\gamma_1,\gamma_2)} &
\tJ(\cB^1,\cB^2) 
\\
\cA^1 \otimes \cA^2 \ar[r]^-{\gamma_1 \otimes \gamma_2} \ar[u]_{p^*} &
\cB^1 \otimes \cB^2 \ar[u]_{p^*}
}}
\end{equation*}
Moreover, if $\gamma_1$ and $\gamma_2$ both admit left or right adjoints, then so does $\tJ(\gamma_1,\gamma_2)$. 
If further $\gamma_1$ and $\gamma_2$ are fully faithful or equivalences, then so is $\tJ(\gamma_1,\gamma_2)$. 
\end{lemma}

\begin{proof}
The formalism of base change for linear categories gives the claimed commutative diagram. 
The rest follows from Lemma~\ref{lemma-adjoints}. 
\end{proof}

\begin{lemma}
\label{lemma-tJ-sod}
Let $\cA^1$ be a $\bP(V_1)$-linear category and $\cA^2$ a $\bP(V_2)$-linear category. 
Then for any $\bP(V_1)$-linear semiorthogonal decomposition $\cA^1 = \llangle \cA', \cA'' \rrangle$, 
there is a ${\tJ(\bP(V_1), \bP(V_2))}$-linear semiorthogonal decomposition 
\begin{equation*}
\tJ(\cA^1, \cA^2) = \llangle \tJ(\cA', \cA^2), \tJ(\cA'', \cA^2) \rrangle. 
\end{equation*} 
A semiorthogonal decomposition of $\cA^2$ induces an analogous decomposition of 
$\tJ(\cA^1, \cA^2)$. 
\end{lemma}

\begin{proof}
Follows from the definition of the resolved join and Lemma~\ref{lemma-sod-tensor}. 
\end{proof}

\begin{lemma}
\label{lemma-tJ-P1-sod}
Let $\cA^1$ be a $\bP(V_1)$-linear category and $\cA^2$ a $\bP(V_2)$-linear category. 
Then the functor 
\begin{equation*}
p^* \colon \cA^1 \sotimes \cA^2 \to \tJ(\cA^1, \cA^2) 
\end{equation*}
is fully faithful, and there is a semiorthogonal decomposition with admissible components
\begin{equation*}
\tJ(\cA^1, \cA^2) = \llangle 
p^*{\left(\cA^1 \sotimes \cA^2 \right)}, 
p^*{\left(\cA^1 \sotimes \cA^2 \right)}(H)
 \rrangle . 
\end{equation*}
\end{lemma}

\begin{proof}
By virtue of the $\bP^1$-bundle structure $p \colon \tJ(\bP(V_1), \bP(V_2)) \to \bP(V_1) \stimes \bP(V_2)$, 
we have a semiorthogonal decomposition 
\begin{equation*}
\Perf(\tJ(\bP(V_1), \bP(V_2))) = \llangle p^* \Perf (\bP(V_1) \stimes \bP(V_2) ), 
p^* \Perf (\bP(V_1) \stimes \bP(V_2) )(H) 
 \rrangle .
\end{equation*} 
Now by Lemmas~\ref{lemma-sod-tensor}, \ref{lemma-admissible-tensor}, and~\ref{lemma-adjoints}, the result follows by base change. 
\end{proof} 

\begin{lemma}
\label{lemma-tJ-smooth-proper}
Let $\cA^1$ and $\cA^2$ be categories linear over $\bP(V_1)$ and $\bP(V_2)$ 
which are smooth and proper over $S$. 
Then the resolved join $\tJ(\cA^1, \cA^2)$ is smooth and proper over $S$. 
\end{lemma}

\begin{proof}
By \cite[Lemma~4.8]{NCHPD} combined with (the proof of) \cite[Chapter I.1, Corollary~9.5.4]{gaitsgory-DAG}, 
the category $\cA^1 \otimes \cA^2$ is smooth and proper over $S$. 
Moreover, $\tJ(\cA^1, \cA^2)$ is obtained from $\cA^1 \otimes \cA^2$ by base change along 
the smooth and proper morphism $\tJ(\bP(V_1), \bP(V_2)) \to \bP(V_1) \stimes \bP(V_2)$. 
Hence the result follows from \cite[Lemma 4.11]{NCHPD}. 
\end{proof}

\subsection{Categorical joins} 
\label{subsection-categorical-joins}

We define the categorical join of Lefschetz categories over $\bP(V_1)$ and $\bP(V_2)$ 
as a certain subcategory of the resolved join. 

\begin{definition}
\label{definition-cat-join} 
Let $\cA^1$ and $\cA^2$ be Lefschetz categories over $\bP(V_1)$ and $\bP(V_2)$ with Lefschetz centers~$\cA^1_0$ and~$\cA^2_0$. 
The \emph{categorical join} $\cJ(\cA^1,\cA^2)$ of $\cA^1$ and $\cA^2$ is defined by
\begin{equation*}
\cJ(\cA^1,\cA^2) = \left\{ C \in \tJ(\cA^1, \cA^2) \ \left|\ 
\begin{aligned}
\eps_1^*(C) &\in \cA^1 \sotimes \cA^2_0 \subset \bE_1(\cA^1, \cA^2) , \\
\eps_2^*(C) &\in \cA^1_0 \sotimes \cA^2 \subset \bE_2(\cA^1, \cA^2)  
\end{aligned}
\right.\right\} , 
\end{equation*}
where we have used the identifications of Remark~\ref{remark-Ek}. 
\end{definition} 

\begin{remark}
The categorical join depends on the choice of Lefschetz centers for~$\cA^1$ and~$\cA^2$, although this is suppressed in the notation. 
For instance, for the ``stupid'' Lefschetz centers~\mbox{$\cA^1_0 = \cA^1$} and~\mbox{$\cA^2_0 = \cA^2$}, 
the condition in the definition is void, so \mbox{$\cJ(\cA^1, \cA^2) = \tJ(\cA^1, \cA^2)$}. 
\end{remark}

To show that $\cJ(\cA^1,\cA^2)$ is an admissible subcategory of $\tJ(\cA^1,\cA^2)$ 
and to describe its orthogonal category, we need the following noncommutative version 
of \cite[Proposition~4.1]{kuznetsov2008lefschetz}, whose proof translates directly to our setting. 
Recall that for a morphism $\eps \colon E \to Y$ we denote by~$\eps_!$ 
the left adjoint of the pullback functor $\eps^* \colon \Perf(Y) \to \Perf(E)$, see~\eqref{eq:shriek-adjoints}.

\begin{proposition} 
\label{proposition-ld-resolution} 
Let $Y$ be a scheme over a base scheme $T$. 
Let $\eps \colon E \to Y$ 
be the embedding of a Cartier divisor in $Y$
with conormal bundle $\cL = \cO_E(-E)$. 
Let~$\cA$ be a $T$-linear category and set 
\begin{equation*}
\cA_Y = \cA \otimes_{\Perf(T)} \Perf(Y) 
\qquad\text{and}\qquad 
\cA_E = \cA \otimes_{\Perf(T)} \Perf(E). 
\end{equation*}
Assume $\cA_E$ is a Lefschetz category with respect to $\cL$ 
with Lefschetz center $\cA_{E,0}$ and Lefschetz components $\cA_{E,i}$, $i \in \bZ$.
Set $m = \length(\cA_E)$.
Then:  
\begin{enumerate}
\item The full subcategory of $\cA_Y$ defined by  
\begin{equation*}
\cB = \set{ C \in \cA_Y \st \eps^*(C) \in \cA_{E,0} } 
\end{equation*}
is admissible.
\item
\label{ld-resolution-right} 
The functor 
$\eps_! \colon \cA_E \to \cA_Y$
is fully faithful on the subcategories $\cA_{E,i} \otimes \cL^i$ for $i \geq 1$, and 
there is a semiorthogonal decomposition 
\begin{equation*}
\cA_Y
= \langle \cB, \eps_!(\cA_{E,1} \otimes \cL), \, \eps_!(\cA_{E,2} \otimes \cL^2), \dots, \eps_!(\cA_{E,m-1} \otimes \cL^{m-1}) \rangle.
\end{equation*} 

\item 
\label{ld-resolution-left} 
The functor 
$\eps_* \colon \cA_E \to \cA_Y$
is fully faithful on the subcategories $\cA_{E,i} \otimes \cL^i$ for~\mbox{$i \leq -1$}, and 
there is a semiorthogonal decomposition 
\begin{equation*}
\cA_Y
= \langle \eps_*(\cA_{E,1-m} \otimes \cL^{1-m}), \, \dots, \, \eps_*(\cA_{E,-2} \otimes \cL^{-2}), \eps_*(\cA_{E,-1} \otimes \cL^{-1}) , \cB \rangle. 
\end{equation*} 
\end{enumerate} 
\end{proposition} 

In the next lemma we apply the above proposition to the resolved join.

\begin{lemma}
\label{lemma-tJ}
For $k=1,2$, let $\cA^k$ be a Lefschetz category over $\bP(V_k)$ of length $m_k$. 
Then the categorical join $\cJ(\cA^1, \cA^2)$ is an admissible $\bP(V_1 \oplus V_2)$-linear 
subcategory of $\tJ(\cA^1, \cA^2)$, and there are $\bP(V_1 \oplus V_2)$-linear 
semiorthogonal decompositions 

\begin{align}
\label{sod-tJ}
\tJ(\cA^1, \cA^2) = \Big \langle  & \notag
\cJ(\cA^1, \cA^2),  \\
& \notag \eps_{1!}{\left(\cA^1 \sotimes \cA^2_1(H_2) \right)},  \dots, 
\eps_{1!}{\left(\cA^1 \sotimes \cA^2_{m_2-1}((m_2-1)H_2)\right)}, \\
& \eps_{2!}{\left( \cA^1_{1}(H_1) \sotimes \cA^2 \right)}, \dots, 
\eps_{2!}{\left( \cA^1_{m_1-1}((m_1-1)H_1) \sotimes \cA^2 \right)}  \Big\rangle, \\ 
\label{sod-tJ-other}
\tJ(\cA^1, \cA^2) = \Big\langle  & \notag  
\eps_{1*}{\left( \cA^1 \sotimes \cA^2_{{1-m_2}}((1-m_2)H_2) \right)},  \dots, 
\eps_{1*}{\left(\cA^1 \sotimes \cA^2_{{-1}}(-H_2)\right)}, \\ 
& \notag 
\eps_{2*}{\left( \cA^1_{{1-m_1}}((1-m_1)H_1) \sotimes \cA^2  \right)}, \dots, 
\eps_{2*}{\left( \cA^1_{{-1}}(-H_1) \sotimes \cA^2  \right)},  \\ 
& \cJ(\cA^1, \cA^2)  \Big\rangle . 
\end{align} 
\end{lemma}

\begin{proof}
We apply Proposition~\ref{proposition-ld-resolution} in the following setup:
\begin{equation*}
T = \bP(V_1) \times \bP(V_2),
\quad
Y = \tJ(\bP(V_1),\bP(V_2)),
\quad
E = \bE_1 \sqcup \bE_2,
\quad\text{and}\quad
\cA = \cA^1 \otimes \cA^2.
\end{equation*}
Then $\cA_Y = \tJ(\cA^1,\cA^2)$ and $\cA_E = \bE_1(\cA^1,\cA^2) \oplus \bE_2(\cA^1,\cA^2)$. 
We claim that 
\begin{equation*}
\cA_{E,0} = (\cA^1 \otimes \cA^2_0) \oplus (\cA^1_0 \otimes \cA^2)
\end{equation*}
is a Lefschetz center of $\cA_E$ with respect to $\cL = \cO_E(-E)$, with Lefschetz components 
\begin{equation*}
\cA_{E,i} = (\cA^1 \otimes \cA^2_i) \oplus (\cA^1_i \otimes \cA^2).  
\end{equation*} 
Indeed, 
by Lemma~\ref{lemma-tJ-divisors} we have 
\begin{equation*}
\cL\vert_{\bE_1} = \cO_{\bE_1}(-\bE_1) = \cO_{\bE_1}(H_2-H_1) 
\quad \text{and} \quad \cL\vert_{\bE_2} = \cO_{\bE_2}(-\bE_2) = \cO_{\bE_2}(H_1-H_2), 
\end{equation*} 
from which the claim follows easily. 

In the above setup, the category $\cB$ of Proposition~\ref{proposition-ld-resolution} coincides with 
the definition of the categorical join $\cJ(\cA^1,\cA^2)$. 
Hence the proposition shows $\cJ(\cA^1,\cA^2)$ is an admissible subcategory of $\tJ(\cA^1, \cA^2)$, 
and gives the semiorthogonal decompositions~\eqref{sod-tJ} and~\eqref{sod-tJ-other}.  

It remains to show the categorical join and the decompositions are $\bP(V_1 \oplus V_2)$-linear.
Since the categorical join is the orthogonal of the other components in the decompositions,
it is enough to check that every other component is~$\bP(V_1 \oplus V_2)$-linear.
By diagram~\eqref{diagram-tJ-blowup} the morphism $\bE_1 \to \bP(V_1 \oplus V_2)$ factors through 
the projection $\bE_1 \cong \bP(V_1) \times \bP(V_2) \to \bP(V_1)$. 
Thus, since the subcategory $\cA^1 \otimes \cA^2_i(iH_2) \subset \bE_1(\cA^1,\cA^2)$ is $\bP(V_1)$-linear (because $\cA^1$ is), 
it is also $\bP(V_1 \oplus V_2)$-linear.
Since $\eps_1$ is a morphism over $\bP(V_1 \oplus V_2)$, it follows that $\eps_{1!}(\cA^1 \otimes \cA^2_i(iH_2))$ is also 
$\bP(V_1 \oplus V_2)$-linear for any $i \geq 1$.
The same argument works for the other components in~\eqref{sod-tJ} and~\eqref{sod-tJ-other}, which finishes the proof. 
\end{proof}

\begin{remark}
The last two rows in~\eqref{sod-tJ} and the first two rows in~\eqref{sod-tJ-other} are completely orthogonal since $\bE_1$ and $\bE_2$ are disjoint.
\end{remark}

Categorical joins preserve smoothness and properness: 

\begin{lemma}
\label{lemma-cJ-smooth-proper}
Let $\cA^1$ and $\cA^2$ be Lefschetz categories over $\bP(V_1)$ and $\bP(V_2)$ 
which are smooth and proper over $S$. 
Then the categorical join~$\cJ(\cA^1, \cA^2)$ is smooth and proper over $S$. 
\end{lemma}

\begin{proof}
Follows from Lemma~\ref{lemma-tJ-smooth-proper}, Lemma~\ref{lemma-tJ}, and \cite[Lemma 4.15]{NCHPD}. 
\end{proof}

\begin{example}
\label{example:cj-pp}
As an example, we consider the categorical join of two projective bundles. 
Let $W_1 \subset V_1$ and $W_2 \subset V_2$ be subbundles, 
so that $\bP(W_1) \subset \bP(V_1)$ and $\bP(W_2) \subset \bP(V_2)$. 
The classical join of these spaces is given by $\bJ(\bP(W_1), \bP(W_2)) = \bP(W_1 \oplus W_2)$.
Consider the Lefschetz structures of $\bP(W_1)$ and $\bP(W_2)$ defined in Example~\ref{example-projective-bundle-lc}.
Then the pullback functor 
\begin{equation*}
f^* \colon \Perf(\bP(W_1 \oplus W_2)) \to \Perf(\tJ(\bP(W_1), \bP(W_2))) 
\end{equation*}
induces a $\bP(W_1 \oplus W_2)$-linear equivalence 
\begin{equation*}
\Perf(\bP(W_1 \oplus W_2)) \simeq \cJ(\bP(W_1), \bP(W_2)) . 
\end{equation*}
Indeed, this follows easily from Lemma~\ref{lemma-tJ-divisors}, Orlov's blowup formula, and the definitions. 
Moreover, Theorem~\ref{theorem-join-lef-cat} below equips $\cJ(\bP(W_1), \bP(W_2))$ with a canonical Lefschetz structure. 
It is easy to check that the above equivalence is a Lefschetz equivalence. 
\end{example}

\subsection{Base change of categorical joins} 
\label{subsection-bc-cat-joins} 
Let $T \to \bP(V_1 \oplus V_2)$ be a morphism of schemes. 
The base change of diagram~\eqref{diagram-tJ-blowup} along 
this morphism gives a diagram 
\begin{equation}
\vcenter{\xymatrix{
\bE_{1T} \ar[r] \ar[d] & \tJ(\bP(V_1), \bP(V_2))_T \ar[d] & \bE_{2T} \ar[l] \ar[d] \\
\bP(V_1)_{T} \ar[r]^-{} & T & \bP(V_2)_{T} \ar[l]_-{}
}}
\end{equation} 
with cartesian squares. 
Note that the isomorphisms $\bE_k \cong \bP(V_1) \stimes \bP(V_2)$, $k=1,2$, 
induce isomorphisms 
\begin{equation*}
\bE_{1T} \cong \bP(V_1)_T \stimes \bP(V_2), \quad 
\bE_{2T} \cong \bP(V_1) \stimes \bP(V_2)_T. 
\end{equation*}
If $\cA^k$ is $\bP(V_k)$-linear for $k=1,2$, then by the 
definition of $\bE_k(\cA^1, \cA^2)$ we have 
\begin{equation*}
\bE_k(\cA^1, \cA^2)_T  
\simeq 
\left( \cA^1 \sotimes \cA^2 \right) \otimes_{\Perf(\bP(V_1) \stimes \bP(V_2))} \Perf(\bE_{kT}).
\end{equation*}
Hence by the above isomorphisms and Lemma~\ref{lemma-fiber-product-categories}, 
we have equivalences 
\begin{align}
\label{E1T} 
\bE_1(\cA^1, \cA^2)_T & \simeq  \cA^1_{\bP(V_1)_T} \sotimes \cA^{2} , \\
\label{E2T} 
\bE_2(\cA^1, \cA^2)_T & \simeq  \cA^{1} \sotimes \cA^2_{\bP(V_2)_T} . 
\end{align}
We identify these categories via these equivalences. 

\begin{lemma}
\label{lemma-tJT}
For $k=1,2$, let $\cA^k$ be a Lefschetz category over $\bP(V_k)$ of length~$m_k$. 
Let \mbox{$T \to \bP(V_1 \oplus V_2)$} be a morphism of schemes. 
Then there is a $T$-linear semiorthogonal decomposition 
\begin{equation*}
\begin{aligned}
\tJ(\cA^1, \cA^2)_T = 
\Big \langle  & \cJ(\cA^1, \cA^2)_T,  \\
& \eps_{1!}{\left(\cA^1_{\bP(V_1)_T} \sotimes \cA^2_1(H_2) \right)},  \dots, 
\eps_{1!}{\left(\cA^1_{\bP(V_1)_T} \sotimes \cA^2_{m_2-1}((m_2-1)H_2)\right)}, \\
& \eps_{2!}{\left( \cA^1_{1}(H_1) \sotimes  \cA^2_{\bP(V_2)_T} \right)}, \dots, 
\eps_{2!}{\left( \cA^1_{m_1-1}((m_1-1)H_1) \sotimes  \cA^2_{\bP(V_2)_T} \right)}  \Big\rangle.
\end{aligned} 
\end{equation*} 
\end{lemma}

\begin{proof}
This is the base change of~\eqref{sod-tJ} with the identifications~\eqref{E1T} and~\eqref{E2T} taken into account.
\end{proof}

For our next result, we need the notion of a linear category being supported over a closed subset. 
The ``support'' of a $T$-linear category $\cC$ 
should be thought of as the locus of points in~$T$ over which $\cC$ 
is nonzero. 
Instead of fully developing this notion, we make the following ad hoc definition 
which is sufficient for our purposes: given a closed subset $Z \subset T$, 
we say~$\cC$ is \emph{supported over $Z$} if $\cC_U \simeq 0$, where $U = T \setminus Z$. 
In particular, note that if $X \to T$ is a morphism of schemes and $Z \subset T$ is a closed subset containing 
the image of $X$, then $\Perf(X)$ is supported over $Z$.

\begin{proposition}
\label{proposition-cJT} 
For $k=1,2$, let $\cA^k$ be a Lefschetz category over $\bP(V_k)$.  
Assume $\cA^k$ is supported over a closed subset $Z_k \subset \bP(V_k)$. 
Assume $T \to \bP(V_1 \oplus V_2)$ is a morphism which factors through the 
complement of $Z_1 \sqcup Z_2$ in~$\bP(V_1 \oplus V_2)$. 
Then there is a $T$-linear equivalence 
\begin{equation}
\label{cJ-tJ-bc}
\cJ(\cA^1, \cA^2)_T \simeq \tJ(\cA^1, \cA^2)_T.  
\end{equation}
If further $T \to \bP(V_1 \oplus V_2)$ 
factors through the complement of $\bP(V_1) \sqcup \bP(V_2)$ in~$\bP(V_1 \oplus V_2)$, 
then there is an equivalence 
\begin{equation}
\label{tJ-A1A2-bc}
\tJ(\cA^1, \cA^2)_T \simeq \cA^1_T \otimes_{\Perf(T)} \cA^2_T, 
\end{equation}
where the factors in the tensor product are the base changes of $\cA^1$ and $\cA^2$ 
along the morphisms~$T \to \bP(V_1)$ and $T \to \bP(V_2)$ obtained by composing $T \to \bP(V_1 \oplus V_2)$
with the linear projections of~$\bP(V_1 \oplus V_2)$ from $\bP(V_2)$ and $\bP(V_1)$, respectively. 
\end{proposition}

\begin{proof}
The assumption on $T \to \bP(V_1 \oplus V_2)$ implies that $\bP(V_k)_T \to \bP(V_k)$ factors 
through the open subset~$\bP(V_k) \setminus Z_k$. 
Thus we have $\cA^k_{\bP(V_k)_T} \simeq 0$ and the equivalence 
\eqref{cJ-tJ-bc} follows from Lemma~\ref{lemma-tJT}.   

By the definition of $\tJ(\cA^1, \cA^2)$ we have an equivalence 
\begin{equation*}
\tJ(\cA^1, \cA^2)_T \simeq \left( \cA^1 \sotimes \cA^2 \right) \otimes_{\Perf(\bP(V_1) \stimes \bP(V_2))} \Perf(\tJ(\bP(V_1), \bP(V_2))_T). 
\end{equation*}
By Lemma~\ref{lemma-tJ-divisors}\eqref{f-blowup} the morphism $f \colon \tJ(\bP(V_1), \bP(V_2)) \to \bP(V_1 \oplus V_2)$ is an isomorphism 
over the complement of $\bP(V_1) \sqcup \bP(V_2)$. 
Hence if $T \to \bP(V_1 \oplus V_2)$ factors through this complement, we have 
have an isomorphism $\tJ(\bP(V_1), \bP(V_2))_T \cong T$. 
Combining this isomorphism and the above equivalence 
we obtain
\begin{equation*}
\tJ(\cA^1,\cA^2)_T \simeq (\cA^1 \sotimes \cA^2) \otimes_{\Perf(\bP(V_1) \times \bP(V_2))} \Perf(T),
\end{equation*}
and applying Corollary~\ref{corollary-base-change-along-diagonal} we deduce~\eqref{tJ-A1A2-bc}. 
\end{proof}

\begin{remark} 
\label{remark-cJ-vs-classical} 
Let $X_1 \subset \bP(V_1)$ and $X_2 \subset \bP(V_2)$ be closed subschemes.
Then the morphism $\tJ(X_1, X_2) \to \bJ(X_1, X_2)$ to the classical join 
is an isomorphism over~$U = \bP(V_1 \oplus V_2) \setminus (X_1 \sqcup X_2)$, 
so the pullback functor~$\Perf(\bJ(X_1, X_2)) \to \Perf(\tJ(X_1, X_2))$ becomes an equivalence after base 
change to $U$. 
Hence if~$\Perf(X_1)$ and~$\Perf(X_2)$ are equipped with Lefschetz structures, 
by Proposition~\ref{proposition-cJT} we have an equivalence~$\cJ(X_1, X_2)_U \simeq \Perf(\bJ(X_1, X_2)_U)$. 
Heuristically, this says that $\cJ(X_1, X_2)$ is birational to the classical join $\bJ(X_1, X_2)$ over $\bP(V_1 \oplus V_2)$. 
On the other hand, if~$X_1$ and~$X_2$ are smooth then so is~$\cJ(X_1, X_2)$ by Lemma~\ref{lemma-cJ-smooth-proper}, 
so in this case~$\cJ(X_1, X_2)$ can be thought of as a resolution of singularities of the classical join~$\bJ(X_1, X_2)$.

There are other notions of noncommutative resolutions of singularities in the literature, 
in particular categorical resolutions in the sense 
of~\cite{kuznetsov2008lefschetz,kuznetsov-lunts} and noncommutative resolutions in the 
sense of Van den Bergh~\cite{vdb-flops, vdb-crepant}. 
Using the results of~\cite{kuznetsov2008lefschetz} it can be shown 
that under certain hypotheses, $\cJ(X_1, X_2)$ is also a resolution of~$\bJ(X_1, X_2)$ in these senses. 
For instance, working over an algebraically closed field, 
$\cJ(X_1, X_2)$ is a categorical resolution if~$X_1$ and~$X_2$ are smooth  and~$\bJ(X_1, X_2)$ has rational singularities. 
\end{remark} 

\subsection{The Lefschetz structure of a categorical join} 
\label{subsection-ld-cJ}
Our next goal is to show that given Lefschetz categories 
over $\bP(V_1)$ and $\bP(V_2)$, their categorical join has a canonical Lefschetz structure.
Recall that $p \colon \tJ(\bP(V_1), \bP(V_2)) \to \bP(V_1) \stimes \bP(V_2)$ denotes the projection. 

\begin{lemma}
\label{lemma-J0}
Let $\cA^1$ and $\cA^2$ be Lefschetz categories over $\bP(V_1)$ and $\bP(V_2)$ with Lefschetz centers~$\cA^1_0$ and~$\cA^2_0$. 
Then the image of the subcategory 
\begin{equation*}
\cA^1_0 \sotimes \cA^2_0 \subset \cA^1 \sotimes \cA^2  
\end{equation*}
under the functor $p^* \colon \cA^1 \sotimes \cA^2 \to \tJ(\cA^1, \cA^2)$ 
is contained in the categorical join $\cJ(\cA^1, \cA^2)$ as an admissible 
subcategory. 
\end{lemma}

\begin{proof} 
By Lemma~\ref{lemma-tJ-P1-sod} the functor  $p^* \colon \cA^1 \sotimes \cA^2 \to \tJ(\cA^1, \cA^2)$ is fully faithful with admissible image. 
By Lemma~\ref{lemma-admissible-tensor} the subcategory $\cA^1_0 \sotimes \cA^2_0 \subset \cA^1 \sotimes \cA^2$ is admissible, 
so its image under~$p^*$ is admissible. 
Finally, since $p \circ \eps_k$ is the identity, 
it follows from Definition~\ref{definition-cat-join} that this image is contained in the categorical join $\cJ(\cA^1, \cA^2)$. 
\end{proof} 

\begin{definition}
\label{definition-lef-center-join}
For Lefschetz categories $\cA^1$ and $\cA^2$ over $\bP(V_1)$ and $\bP(V_2)$, 
we define 
\begin{equation}\label{eq:cat-join-center}
\cJ(\cA^1, \cA^2)_0 = 
p^* {\left(\cA^1_0 \sotimes \cA^2_0 \right)} \subset \cJ(\cA_1,\cA_2). 
\end{equation} 
\end{definition} 

Note that the containment $\cJ(\cA^1, \cA^2)_0 \subset \cJ(\cA^1, \cA^2)$ holds by Lemma~\ref{lemma-J0}. 

\begin{theorem} 
\label{theorem-join-lef-cat}
Let $\cA^1$ and $\cA^2$ be Lefschetz categories over $\bP(V_1)$ and $\bP(V_2)$ with Lefschetz centers~$\cA^1_0$ and~$\cA^2_0$. 
Then the categorical join $\cJ(\cA^1, \cA^2)$ has the structure of a Lefschetz category over~$\bP(V_1 \oplus V_2)$ 
with center~$\cJ(\cA^1, \cA^2)_0$ given by~\eqref{eq:cat-join-center}, 
and Lefschetz components given by~\eqref{Ji} and~\eqref{Ji-left} below. 

If $\cA^1$ and $\cA^2$ are both either right or left strong, then so is~$\cJ(\cA^1, \cA^2)$. 
Moreover, we have 
\begin{equation*}
\length(\cJ(\cA^1, \cA^2)) = \length(\cA^1) + \length(\cA^2), 
\end{equation*}
and $\cJ(\cA^1, \cA^2)$ is moderate if and only if one of $\cA^1$ or $\cA^2$ is moderate. 
\end{theorem}

The proof of Theorem~\ref{theorem-join-lef-cat} takes the rest of this section. 
We let $\cA^1$ and $\cA^2$ be as in the theorem. 
Further, we let 
\begin{equation*}
m_1 = \length(\cA^1), \quad m_2 = \length(\cA^2), \quad m = m_1 + m_2, 
\end{equation*}
and set 
\begin{equation*}
\cJ_0 =  \cJ(\cA^1, \cA^2)_0. 
\end{equation*}
Note that by Lemma~\ref{lemma-tJ}, the categorical join $\cJ(\cA^1, \cA^2)$ is naturally  
$\bP(V_1 \oplus V_2)$-linear. 
To prove the theorem, we will explicitly construct 
the required Lefschetz decompositions of~$\cJ(\cA^1, \cA^2)$ and apply Lemma~\ref{lemma-lef-center-from-decomp}.

For $k = 1,2$, 
let $\fa^k_i$, $0 \neq i \in \bZ$, and $\fa^{k}_{+0}$, $\fa^{k}_{-0}$, be the primitive components of the 
Lefschetz category $\cA^k$ as defined in \S\ref{subsection-lef-cats}. 
We define 
\begin{equation}
\label{fji} 
\fj_i 
= 
\bigoplus_{\substack{i_1 + i_2 = i - 1 \\ i_1, i_2 \geq 0}} p^*{\left(\fa^1_{i_1} \sotimes \fa^2_{i_2} \right)} , \quad i \geq 0, 
\end{equation} 
where in the formula $\fa^k_0$ denotes $\fa^{k}_{+0}$ for $k=1,2$. 
Note that $\fj_0 = 0$ while $\fj_1 = {p^*(\fa^1_{+0} \sotimes \fa^2_{+0})}$.

Similarly, we define
\begin{equation}
\label{fji-left} 
\fj_i  = 
\bigoplus_{\substack{i_1 + i_2 = i + 1 \\ i_1, i_2 \leq 0}} p^*{\left(\fa^1_{i_1} \sotimes \fa^2_{i_2} \right)} , \quad i \leq 0 , 
\end{equation} 
where in the formula $\fa^k_0$ denotes $\fa^{k}_{-0}$ for $k=1,2$. 
Note that $\fj_0 = 0$ with this definition, which is consistent with~\eqref{fji} for $i = 0$. 

\begin{lemma}
\label{lemma:cj0-sod}
We have semiorthogonal decompositions
\begin{equation*}
\cJ_0  = \llangle \fj_0, \fj_1, \dots, \fj_{m-1} \rrangle
\qquad\text{and}\qquad
\cJ_0 = \llangle \fj_{1-m}, \dots, \fj_{-1}, \fj_{0} \rrangle. 
\end{equation*}
\end{lemma}
\begin{proof}
Applying Lemma~\ref{lemma-sod-tensor} to $\cA^1_0 = \langle \fa^1_{0},\fa^1_1,\dots,\fa^1_{m_1-1} \rangle$ and 
$\cA^2_0 = \langle \fa^2_{0},\fa^2_1,\dots,\fa^2_{m_2-1} \rangle$, 
we obtain a semiorthogonal decomposition
\begin{equation*}
\cA^1_0 \otimes \cA^2_0 = \llangle \fa^1_{i_1} \otimes \fa^2_{i_2} \rrangle_{0 \le i_1 \le m_1-1,\, 0 \le i_2 \le m_2-1}
\end{equation*}
with components $\fa^1_{i_1} \otimes \fa^2_{i_2}$ and $\fa^1_{j_1} \otimes \fa^2_{j_2}$ semiorthogonal if $i_1 < j_1$ or $i_2 < j_2$.
Since $p^*$ defines an equivalence between $\cA^1_0 \otimes \cA^2_0$ and $\cJ_0$, 
we obtain a semiorthogonal decomposition
\begin{equation*}
\cJ_0 = \llangle p^*(\fa^1_{i_1} \otimes \fa^2_{i_2}) \rrangle_{0 \le i_1 \le m_1-1, \, 0 \le i_2 \le m_2-1}
\end{equation*}
with the same semiorthogonalities between the components.
It follows that the summands in the right hand side of~\eqref{fji} are completely orthogonal 
as subcategories of $\cJ_0$, so that we indeed have inclusions $\fj_i \subset \cJ_0$ for $i \geq 0$, 
which give the first claimed semiorthogonal decomposition of $\cJ_0$. 
This is illustrated in Figure~\ref{figure:J0} below. 
\begin{figure}[h]
\centering
\includegraphics[clip, trim=150 548 170 125]{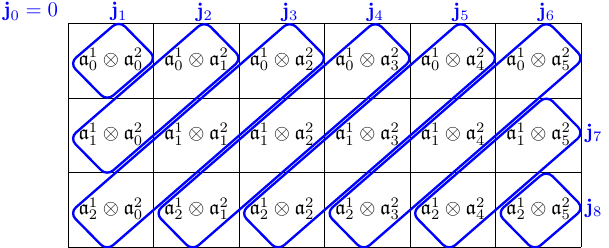}
\caption{The semiorthogonal decomposition of $\cJ_0$ into the components $\bj_i$ for $m_1 = 3$ and $m_2 = 6$. 
For simplicity, $p^*$ is omitted from $p^*(\fa^1_{i_1} \otimes \fa^2_{i_2})$. }  
\label{figure:J0}
\end{figure}

The second claimed semiorthogonal decomposition of $\cJ_0$ is constructed analogously.
\end{proof}

We define two descending chains of subcategories of $\cJ_0$ by 
\begin{align}
\label{Ji}
\cJ_i  & = \langle \fj_i, \fj_{i+1}, \dots, \fj_{m-1} \rangle , \quad 0 \leq i \leq m-1, \\ 
\label{Ji-left}
\cJ_i & = \llangle \fj_{1-m}, \dots, \fj_{i-1}, \fj_{i} \rrangle, \quad 1- m \leq i \leq 0.
\end{align}
Note that $\fj_0 = 0$ implies $\cJ_{-1} = \cJ_0 = \cJ_1$.

\begin{lemma}
\label{lemma-Ji-admissible}
The subcategories $\cJ_i \subset \cJ(\cA^1, \cA^2)$ are 
right admissible for $i \geq 0$ and left admissible for $i \leq 0$. 
Further, if $\cA^1$ and $\cA^2$ are both right strong \textup(or left strong\textup), 
then the subcategory $\fj_{i} \subset \cJ(\cA^1, \cA^2)$ is admissible for $i \geq 0$ \textup(or $i \leq 0$\textup). 
\end{lemma}

\begin{proof}
For $i = 0$ we know $\cJ_0 \subset \cJ(\cA^1, \cA^2)$ is admissible by Lemma~\ref{lemma-J0}. 
If $i > 0$, then by definition we have a semiorthogonal decomposition 
\begin{equation*}
\cJ_0 = \llangle \fj_0, \dots, \fj_{i-1}, \cJ_i \rrangle. 
\end{equation*} 
It follows that $\cJ_i$ is right admissible in $\cJ_0$, and hence also in $\cJ(\cA^1, \cA^2)$ 
as $\cJ_0$ is. 
Similarly, $\cJ_i$ is left admissible in $\cJ(\cA^1, \cA^2)$ for $i < 0$. 

Finally, assume $\cA^1$ and $\cA^2$ are both right strong. 
Then Lemma~\ref{lemma-admissible-tensor} implies the subcategory~$\fa^1_{i_1} \sotimes \fa^2_{i_2}  \subset \cA^1_0 \sotimes \cA^2_0$ 
is admissible for $i_1, i_2 \geq 0$, and hence by Lemma~\ref{lemma-J0} 
the summands defining $\fj_i$, $i \geq 0$, in \eqref{fji} are admissible in 
$\cJ(\cA^1, \cA^2)$. 
So by \cite[Lemma 3.10]{NCHPD} we conclude that $\fj_i$, $i \geq 0$, is admissible 
in $\cJ(\cA^1, \cA^2)$. 
A similar argument applies if $\cA^1$ and $\cA^2$ are left strong. 
\end{proof}

The following alternative expressions for the categories $\cJ_i$ are sometimes useful, 
and can be proved by unwinding the definitions. 

\begin{lemma}
\label{lemma-Ji-alternate} 
For $i \neq 0$ we have $\cJ_i = p^*(\barcJ_i)$, where $\barcJ_i \subset \cA^1 \otimes \cA^2$ 
is the subcategory defined by 
\allowdisplaybreaks
\begin{align*}
\barcJ_i & = 
\begin{cases}
\llangle \cA^1_{i-1} \sotimes \fa^2_0, \cA^1_{i-2} \sotimes \fa^2_1, \dots, \cA^1_{1} \sotimes \fa^2_{i-2}, \cA^1_0 \sotimes \cA^2_{i-1} \rrangle
& \quad \text{if $1 \le i \le m_2$},\\
\llangle \cA^1_{i-1} \sotimes \fa^2_0, \cA^1_{i-2} \sotimes \fa^2_1, \dots, \cA^1_{i-m_2} \sotimes \fa^2_{m_2-1} \rrangle
& \quad \text{if $m_2 < i \leq m-1$},
\end{cases}
\\ 
& = 
\begin{cases}
\llangle \fa^1_0 \sotimes \cA^2_{i-1}, {\fa^1_1} \sotimes \cA^2_{i-2}, \dots, \fa^1_{i-2} \sotimes \cA^2_{1}, \cA^1_{i-1} \sotimes \cA^2_0 \rrangle
& \quad \text{if $1 \le i \le m_1$},\\
\llangle \fa^1_0 \sotimes \cA^2_{i-1}, {\fa^1_1} \sotimes \cA^2_{i-2}, \dots, \fa^1_{m_1-1} \sotimes \cA^2_{i-m_1} \rrangle
& \quad \text{if $m_1 < i \leq m-1$},
\end{cases} 
\\
& = 
\begin{cases}
\llangle \cA^1_0 \sotimes \cA^2_{i+1}, 
\cA^1_{-1} \sotimes \fa^2_{i+2}, \dots, 
\cA^1_{i+2} \sotimes \fa^2_{-1}, 
\cA^1_{i+1} \sotimes \fa^2_0 \rrangle 
& \text{if $-m_2 \le i \le -1$},\\ 
\llangle \cA^1_{i+m_2} \sotimes \fa^2_{1-m_2}, \dots, 
\cA^1_{i+2} \sotimes \fa^2_{-1}, 
\cA^1_{i+1} \sotimes \fa^2_0 \rrangle 
& \text{if $1-m \leq i < -m_2$},
\end{cases}
\\ 
& = 
\begin{cases}
\llangle  \cA^1_{i+1} \sotimes \cA^2_0,
\fa^1_{i+2} \sotimes \cA^2_{-1} , \dots, 
\fa^1_{-1} \sotimes \cA^2_{i+2}, 
\fa^1_0 \sotimes \cA^2_{i+1}  \rrangle 
& \text{if $-m_1 \le i \le -1$},\\ 
\llangle  \fa^1_{1-m_1} \sotimes \cA^2_{i+m_1} , \dots, 
\fa^1_{-1} \sotimes \cA^2_{i+2}, 
\fa^1_0 \sotimes \cA^2_{i+1}  \rrangle 
& \text{if $1-m \leq i < -m_1$}, 
\end{cases}
\end{align*} 
where for $k =1,2$, the symbol $\fa^k_0$ denotes $\fa^{k}_{+0}$ 
in the first two equalities and $\fa^{k}_{-0}$ in the last two equalities. 
\end{lemma}

\begin{figure}[h]
    \centering
 \includegraphics[clip, trim=175 535 175 135]{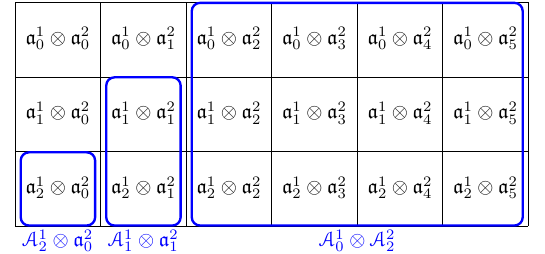}
\caption{The first semiorthogonal decomposition of $\barcJ_i$ from Lemma \ref{lemma-Ji-alternate}  for~$i = 3$, $m_1 = 3$, and $m_2 = 6$.}     
\label{figure:J3}
\end{figure}

To prove Theorem~\ref{theorem-join-lef-cat}, we will show that we have semiorthogonal decompositions 
\begin{align}
\label{sod-cJ}
\cJ(\cA^1, \cA^2) & = \llangle \cJ_0 , \cJ_1(H), \dots, \cJ_{m-1}((m-1)H) \rrangle, \\ 
\label{sod-cJ-left}
\cJ(\cA^1, \cA^2) & = \llangle \cJ_{1-m}((1-m)H), \dots, \cJ_{-1}(-H), \cJ_0 \rrangle, 
\end{align} 
and then apply Lemma~\ref{lemma-lef-center-from-decomp}.
We focus on proving~\eqref{sod-cJ} below; an analogous argument proves~\eqref{sod-cJ-left}. 

\begin{lemma}
\label{lemma-fji}
The sequence of subcategories 
\begin{equation}
\label{A-sequence} 
\cJ_0 , \, \cJ_1(H), \, \dots, \, \cJ_{m-1}((m-1)H) , 
\end{equation}
in~\eqref{sod-cJ} is semiorthogonal.
\end{lemma}

\begin{proof}
By the definitions~\eqref{Ji} and \eqref{fji}, 
it is enough to check that for any integer $t$ such that $1 \leq t \leq j_1+j_2+1$, 
the subcategories  
\begin{equation*}
p^*{\left(\fa^1_{i_1} \sotimes \fa^2_{i_2} \right)}, \quad 
p^*{\left(\fa^1_{j_1} \sotimes \fa^2_{j_2} \right)}(tH), 
\end{equation*}
of $\cJ(\cA^1, \cA^2)$ are semiorthogonal. 
Let $C_1 \in \fa^1_{i_1}$, $C_2 \in \fa^2_{i_2}$ and $D_1 \in \fa^1_{j_1}$, $D_2 \in \fa^2_{j_2}$. 
Since $\fa^1_{i_1} \sotimes \fa^2_{i_2}$ and $\fa^1_{j_1} \sotimes \fa^2_{j_2}$ are thickly generated 
by objects of the form $C_1 \boxtimes C_2$ and $D_1 \boxtimes D_2$ respectively,
we must show that
\begin{equation*}
p^*(D_1 \boxtimes D_2)(tH) \in {}^\perp\left( p^*(C_1 \boxtimes C_2) \right). 
\end{equation*}
Recall that $p_!$ denotes the left adjoint functor of $p^*$, see~\eqref{eq:shriek-adjoints}.
By adjunction and the projection formula, this is equivalent to
\begin{equation*}
(D_1 \boxtimes D_2) \otimes p_!(\cO_{{\tJ(\bP(V_1),\bP(V_2))}}(tH)) \in {}^\perp \left( C_1 \boxtimes C_2 \right). 
\end{equation*}
Using the formula of Lemma~\ref{lemma-tJ-divisors}\eqref{omega-p} for the dualizing complex of $p$,
for $t \geq 1$ we obtain 
\begin{equation}
\label{eq:p-shriek}
p_{!}(\cO_{{\tJ(\bP(V_1),\bP(V_2))}}(tH)) \simeq   
\bigoplus_{\substack{t_1 + t_2 = t \\ t_1, t_2 \geq 1}} \cO_{\bP(V_1) \times \bP(V_2)}(t_1H_1 + t_2H_2)[1].
\end{equation}
So, it is enough to show that
\begin{equation*}
D_1(t_1H_1) \boxtimes D_2(t_2H_2) \in {}^\perp \left( C_1 \boxtimes C_2 \right)
\end{equation*}
for all $t_1, t_2 \geq 1$ such that $t_1 + t_2 = t$. 
The left side is contained in $\fa^1_{j_1}(t_1H_1) \otimes \fa^2_{j_2}(t_2H_2)$, while $C_1 \boxtimes C_2 \in \fa^1_{i_1} \otimes \fa^2_{i_2}$, 
so it suffices to show the pair $(\fa^1_{j_1}(t_1H_1) \otimes \fa^2_{j_2}(t_2H_2), \fa^1_{i_1} \otimes \fa^2_{i_2})$ is semiorthogonal.
By Lemma~\ref{lemma-sod-tensor} and~\eqref{eq:ca-primitive-plus}, 
this holds for~$1 \leq t_1 \leq j_1$ because of the semiorthogonality of the first factors, 
and for $1 \leq t_2 \leq j_2$ because of the semiorthogonality of the second factors. 
Since the assumption $t_1 + t_2 = t \leq j_1 + j_2 + 1$ implies either 
$t_1 \leq j_1$ or~$t_2 \leq j_2$, this finishes the proof.
\end{proof}

To show that the categories in~\eqref{A-sequence} generate $\cJ(\cA^1, \cA^2)$, 
we consider the idempotent-complete triangulated subcategory~$\cP$ of the resolved join $\tJ(\cA^1, \cA^2)$ 
generated by the following subcategories: 
\begin{align}
\label{P1} p^*( \fa^1_{i_1} \sotimes \fa^2_{i_2})(tH), & \quad 
0 \leq i_1 \leq m_1 - 1, ~ 0 \leq i_2 \leq m_2 - 1, ~
0 \leq t \leq i_1 + i_2 + 1 , \\ 
\label{P2}\eps_{1*}(\cA^1 \sotimes \fa^2_{i_2}(s_2H_2)), & \quad 
0 \leq i_2 \leq m_2 - 1, ~ 
0 \leq s_2 \leq i_2 - 1, \\ 
\label{P3} \eps_{2*}(\fa^1_{i_1}(s_1H_{1}) \sotimes \cA^2 ), & \quad 
0 \leq i_1 \leq m_1 - 1, ~ 
0 \leq s_1 \leq i_1 - 1. 
\end{align}
It follows from the definitions~\eqref{fji} and~\eqref{Ji} 
that~\eqref{A-sequence} and~\eqref{P1} generate the same subcategory of the categorical join~$\cJ(\cA^1, \cA^2)$. 
Further, by~\eqref{eq:shriek-adjoints} and Lemma~\ref{lemma-tJ-divisors}\eqref{Ek-Hk-H} we have
\begin{equation*}
\eps_{1!}(C) \simeq \eps_{1*}(C(H_1 - H_2)[-1]) , 
\end{equation*}
so it follows from $\bP(V_1)$-linearity of $\cA^1$ and the definitions 
that~\eqref{P2} generates the second line of the semiorthogonal decomposition~\eqref{sod-tJ} of $\tJ(\cA^1, \cA^2)$.
Similarly, \eqref{P3} generates the third line of~\eqref{sod-tJ}.
Hence to establish~\eqref{sod-cJ}, it suffices to show $\cP = \tJ(\cA^1, \cA^2)$. 
For this, we will need the following lemma.  

\begin{lemma}
\label{lemma-induction}
For all integers $i_1, i_2, s_1, s_2, t$, such that 
\begin{equation*}
0 \leq s_1 \leq i_1 \leq m_1 - 1 , ~  
0 \leq s_2 \leq i_2 \leq m_2 - 1, ~ 
0 \leq t \leq i_1 + i_2 - (s_1 + s_2) + 1 , 
\end{equation*}
the subcategory 
\begin{equation*}
p^*{\left(\fa^1_{i_1} \sotimes \fa^2_{i_2} \right)}(s_1H_1 + s_2H_2 + tH) \subset \tJ(\cA^1, \cA^2) 
\end{equation*}
is contained in $\cP$. 
\end{lemma}

\begin{proof}
We argue by induction on $s = s_1 + s_2$. 
The base case $s_1 = s_2 = 0$ holds by~\eqref{P1} 
in the definition of $\cP$. 
Now assume $s > 0$ and the result holds for $s-1$. 
Either $s_1 > 0$ or~$s_2 > 0$. 
Assume $s_1 > 0$. 
Consider the exact sequence
\begin{equation}
\label{eq:e2-sequence}
0 \to \cO(-\bE_2) \to \cO \to \eps_{2*}\cO_{\bE_2} \to 0
\end{equation}
on $\tJ(\bP(V_1),\bP(V_2))$. 
By Lemma~\ref{lemma-tJ-divisors}\eqref{Ek-Hk-H} we 
have $-\bE_2 = H_1-H$ and $H|_{\bE_2} = H_2$, so twisting 
this sequence by $\cO((s_1 - 1)H_1 + s_2H_2 + (t + 1)H)$ gives 
\begin{multline*}
0 \to \cO(s_1H_1 + s_2H_2 + tH) \to \cO((s_1 - 1)H_1 + s_2H_2 + (t + 1)H) 
\\
\to \eps_{2*}\cO_{\bE_2}((s_1 - 1)H_1 + (s_2 + t + 1)H_2) \to 0.
\end{multline*}
For $C_1 \in \fa^1_{i_1}$, $C_2 \in \fa^2_{i_2}$, tensoring 
this sequence with $p^*(C_1 \boxtimes C_2)$ gives an exact  
triangle 
\begin{multline*}
p^*(C_1 \boxtimes C_2)(s_1H_1 + s_2H_2 + tH) \to p^*(C_1 \boxtimes C_2)((s_1 - 1)H_1 + s_2H_2 + (t + 1)H) 
\\
\to \eps_{2*}(C_1((s_1 - 1)H_1) \boxtimes C_2 ((s_2 + t + 1)H_2) ),
\end{multline*}
where we have used the projection formula and diagram~\eqref{diagram-tJ-projective-bundle} to rewrite the third term. 
The second term of this triangle is in $\cP$ by the induction hypothesis, 
and the third term is in~$\cP$ by~\eqref{P3} since $s_1 - 1 \le i_1 - 1$ by the assumption of the lemma.
Hence the first term is also in $\cP$. 
By Lemma~\ref{lemma-generators-box-tensor} 
the objects $p^*(C_1 \boxtimes C_2)$ for $C_1 \in \fa^1_{i_1}$, $C_2 \in \fa^2_{i_2}$, 
thickly generate~$p^*{\left(\fa^1_{i_1} \sotimes \fa^2_{i_2}\right)}$, 
so we deduce the required containment 
\begin{equation*}
p^*{\left(\fa^1_{i_1} \sotimes \fa^2_{i_2}\right)}(s_1H_1 + s_2H_2 + tH) \subset \cP. 
\end{equation*}
The case $s_2 > 0$ follows by the same argument (with $\bE_2$ replaced by $\bE_1$ and~\eqref{P2} used instead of~\eqref{P3}). 
This completes the induction. 
\end{proof}

\begin{proof}[{Proof of Theorem~\textup{\ref{theorem-join-lef-cat}}}]
Let us show $\cP = \tJ(\cA^1, \cA^2)$, which as observed above will complete 
the proof of the semiorthogonal decomposition~\eqref{sod-cJ}. 
By Lemma~\ref{lemma-tJ-P1-sod} we have a semiorthogonal decomposition 
\begin{equation}
\label{tJ-P1}
\tJ(\cA^1, \cA^2) = \llangle 
p^*{\left(\cA^1 \sotimes \cA^2 \right)}, 
p^*{\left(\cA^1 \sotimes \cA^2 \right)}(H) 
 \rrangle , 
\end{equation}
so it suffices to show $\cP$ contains both components of this decomposition. 
But tensoring the decompositions~\eqref{eq:ca-primitive-plus} for $\cA^1$ and $\cA^2$ and using Lemma~\ref{lemma-sod-tensor}, 
we see that $\cA^1 \sotimes \cA^2$ 
is generated by the categories 
\begin{equation*}
(\fa^1_{i_1} \sotimes \fa^2_{i_2})(s_1H_1 + s_2H_2) , \quad 
0 \leq s_1 \leq i_1 \leq m_1 - 1 , ~
0 \leq s_2 \leq i_2 \leq m_2 - 1. 
\end{equation*}
Hence taking $t =0$ in Lemma~\ref{lemma-induction} shows $\cP$ contains 
$p^*{\left(\cA^1 \sotimes \cA^2 \right)}$, and taking $t = 1$ shows 
$\cP$ contains $p^*{\left(\cA^1 \sotimes \cA^2 \right)}(H)$, 
as required. 

The semiorthogonal decomposition~\eqref{sod-cJ-left} holds by a similar argument.
Thus by Lemma~\ref{lemma-lef-center-from-decomp} and 
Lemma~\ref{lemma-Ji-admissible} we deduce that $\cJ_0 \subset \cJ(\cA_1,\cA_2)$
is a Lefschetz center with $\cJ_i$, $i \in \bZ$, the corresponding Lefschetz components. 
The strongness claims follow from the definitions and Lemma~\ref{lemma-Ji-admissible}, 
and the claims about the length and moderateness of $\cJ(\cA^1, \cA^2)$ follow 
from the definitions and~\eqref{length-leq-rank}.
\end{proof}


\section{HPD for categorical joins} 
\label{section-joins-HPD} 

In this section we prove our main theorem, 
the general form of Theorem~\ref{main-theorem-intro} for Lefschetz categories, 
which says that (under suitable hypotheses) 
the formation of categorical joins commutes with~HPD. 

\begin{theorem}
\label{theorem-joins-HPD}
Let $\cA^1$ and $\cA^2$ be right strong, moderate Lefschetz categories over $\bP(V_1)$ and~$\bP(V_2)$.  
Then there is an equivalence 
\begin{equation*}
\cJ(\cA^1, \cA^2)^{\hpd} \simeq \cJ((\cA^1)^{\hpd}, (\cA^2)^{\hpd}) 
\end{equation*}
of Lefschetz categories over $\bP(\vV_1 \oplus \vV_2)$. 
\end{theorem}

\begin{remark}
\label{remark-joins-HPD}
The Lefschetz structures on the categories $\cJ(\cA^1, \cA^2)^{\hpd}$ and $\cJ((\cA^1)^{\hpd}, (\cA^2)^{\hpd})$ 
in Theorem~\ref{theorem-joins-HPD} are the ones obtained by combining 
Theorem~\ref{theorem-HPD}\eqref{Cd-ld} and Theorem~\ref{theorem-join-lef-cat}. 
\end{remark}

The key object in the proof of Theorem~\ref{theorem-joins-HPD} is a certain fiber product of resolved joins, 
which we call a double resolved join. 
We discuss this construction in~\S\ref{subsection-double-resolved-joins}, 
and then use it in~\S\ref{subsection-joins-HPD-functor} to define a functor 
$\gamma_{\tJ} \colon \tJ((\cA^1)^{\hpd}, (\cA^2)^{\hpd}) \to \bH(\tJ(\cA^1,\cA^2))$. 
In~\S\ref{subsection-HPD-functor-relations} we prove various properties of~$\gamma_{\tJ}$, 
which we use in~\S\ref{subsection-joins-HPD-proof} to show~$\gamma_{\tJ}$ induces the equivalence of Theorem~\ref{theorem-joins-HPD}.

\subsection{Double resolved joins} 
\label{subsection-double-resolved-joins}
For $k=1,2$, let $V_k$ be a vector bundle on $S$ and 
denote by $H_k$ and $H'_k$ the relative hyperplane 
classes on $\bP(V_k)$ and $\bP(\vV_k)$. 

In this section sometimes we will consider pairs of $\bP(V_k)$-linear categories, 
so we can form their resolved join over $\bP(V_1 \oplus V_2)$,
and sometimes we will consider pairs of $\bP(\vV_k)$-linear categories, 
so we can form their resolved join over $\bP(\vV_1 \oplus \vV_2)$.
Moreover, sometimes we will consider pairs of $\bP(V_k) \times \bP(\vV_k)$-linear categories,
so that we can form both types of joins for them.
To distinguish notationally between the two types of joins we will write
\begin{align*}
\tJ(Y_1, Y_2) &= \bP_{Y_1 \stimes Y_2}(\cO(-H_1) \oplus \cO(-H_2)),\\
\tJ(\cB^1, \cB^2) &= (\cB^1 \otimes \cB^2) \otimes_{\Perf(\bP(V_1) \stimes \bP(V_2))} \Perf(\tJ(\bP(V_1),\bP(V_2)))
\intertext{if $Y_k$ are schemes over $\bP(V_k)$ and $\cB^k$ are $\bP(V_k)$-linear categories, and}
\tJv(Y_1, Y_2) &= \bP_{Y_1 \stimes Y_2}(\cO(-H'_1) \oplus \cO(-H'_2)),\\
\tJv(\cB^1, \cB^2) &= (\cB^1 \otimes \cB^2) \otimes_{\Perf(\bP(\vV_1) \stimes \bP(\vV_2))} \Perf(\tJv(\bP(\vV_1),\bP(\vV_2)))
\end{align*}
if $Y_k$ are schemes over $\bP(\vV_k)$ and $\cB^k$ are $\bP(\vV_k)$-linear categories.
We will also use this convention for schemes over $\bP(V_k) \times \bP(\vV_k)$ and for $\bP(V_k) \times \bP(\vV_k)$-linear categories.
Note, however, that we do not extend this convention to categorical joins, or to resolved joins of functors. 

Let $Y_1$ and $Y_2$ be $S$-schemes equipped with morphisms 
\begin{equation*}
Y_1 \to \bP(V_1) \times \bP(\vV_1),  \qquad   Y_2 \to \bP(V_2) \times \bP(\vV_2).
\end{equation*}
We define the \emph{double resolved join} of $Y_1$ and $Y_2$ 
as the fiber product 
\begin{equation*}
\tJJ(Y_1, Y_2) = \tJ(Y_1, Y_2) \times_{(Y_1 \times Y_2)} \tJv(Y_1, Y_2). 
\end{equation*}
In particular, we can consider the universal double resolved join with its natural projection
\begin{equation}
\label{tJJ-bc}
\tJJ(\bP(V_1) \times \bP(\vV_1), \bP(V_2) \times \bP(\vV_2))
\to (\bP(V_1) \times \bP(\vV_1)) \times (\bP(V_2) \times \bP(\vV_2)).
\end{equation}
This projection is a $(\bP^1 \times \bP^1)$-bundle.
Now, given for $k=1,2$, a category~$\cB^k$ which has a~$\bP(V_k) \times \bP(\vV_k)$-linear structure, 
the \emph{double resolved join} $\tJJ(\cB^1, \cB^2)$ 
of~$\cB^1$ and~$\cB^2$ is defined as 
\begin{equation*}
(\cB^1 \sotimes \cB^2) 
\otimes_{\Perf((\bP(V_1) \times \bP(\vV_1)) \times (\bP(V_2) \times \bP(\vV_2)))}
\Perf(\tJJ(\bP(V_1) \times \bP(\vV_1), \bP(V_2) \times \bP(\vV_2))),
\end{equation*}
that is the base change of $\cB^1 \sotimes \cB^2$ along~\eqref{tJJ-bc}. 

For us, the key case of a double resolved join is when $Y_1$ 
and $Y_2$ are the universal spaces of hyperplanes in $\bP(V_1)$ 
and $\bP(V_2)$, which we denote by 
\begin{equation*}
\bH_1 = \bH(\bP(V_1)) \quad \text{and} \quad \bH_2 = \bH(\bP(V_2)).  
\end{equation*}
Note that for $k=1,2$, the space $\bH_k$ indeed comes with a natural map to $\bP(V_k) \times \bP(\vV_k)$, 
hence we can form the double resolved join of~$\bH_1$ and~$\bH_2$. 
The following commutative diagram summarizes the spaces involved 
and names the relevant morphisms: 
\begin{equation}
\label{tJJH1H2-diagram}
\vcenter{
\xymatrix@C=.2em{
& & \tJJ(\bH_1, \bH_2) \ar[dl]_{\tilde{p}} \ar[dr]^{\tilde{q}} 
\\
& \tJv(\bH_1,\bH_2) \ar[dl]_{\wtilde{h_1 \times h_2}} \ar[dr]^{q_{\bH}} & &
\tJ(\bH_1, \bH_2) \ar[dl]_{p_{\bH}} \ar[dr]^{\wtilde{\pi_1 \times \pi_2}}
\\
\tJv(\bP(\vV_1), \bP(\vV_2)) \ar[dr]^{q} \ar[d]_{g} & & 
\bH_1 \times \bH_2 \ar[dl]_{h_1 \times h_2} \ar[dr]^{\pi_1 \times \pi_2} & & 
\tJ(\bP(V_1), \bP(V_2)) \ar[dl]_{p} \ar[d]^{f}  
\\ 
\bP(\vV_1 \oplus \vV_2) & 
\bP(\vV_1) \times \bP(\vV_2) & & 
\bP(V_1) \times \bP(V_2) & 
\bP(V_1 \oplus V_2) 
}}
\end{equation}
All of the squares in this diagram are cartesian. 

Since $\tJ(\bP(V_1), \bP(V_2))$ maps to $\bP(V_1 \oplus V_2)$, we can form 
the corresponding universal hyperplane section, which sits as a divisor in the product
\begin{equation}
\label{bHtJ}
\bH(\tJ(\bP(V_1), \bP(V_2))) \subset \tJ(\bP(V_1), \bP(V_2)) \times \bP(\vV_1 \oplus \vV_2). 
\end{equation}

\begin{lemma}
\label{lemma-tJJ-functor-spaces}
We have a diagram 
\begin{equation}
\label{tJJ-functor-spaces}
\vcenter{
\xymatrix{
& \tJJ(\bH_1, \bH_2) \ar[dl]_{\tilde{p}} \ar[dr]^{\alpha}
\\
\tJv(\bH_1,\bH_2) 
&&
\bH(\tJ(\bP(V_1),\bP(V_2))) 
}
}
\end{equation}
of schemes over $\bP(\vV_1 \oplus \vV_2)$, where 
all schemes appearing are smooth and projective over $S$. 
\end{lemma}

\begin{proof}
The morphism~$\tilde{p}$ is constructed in diagram~\eqref{tJJH1H2-diagram}.
Furthermore, the same diagram gives morphisms
\begin{align*}
\tJJ(\bH_1, \bH_2) &\xrightarrow{\hbox to 6em {\hfil$\scriptstyle\wtilde{\pi_1 \times \pi_2} \circ \tilde{q}$\hfil} } \tJ(\bP(V_1),\bP(V_2)),
\\
\tJJ(\bH_1, \bH_2) &\xrightarrow{\hbox to 6em {\hfil$\scriptstyle g \circ \wtilde{h_1 \times h_2} \circ \tilde{p}$\hfil} } \bP(\vV_1 \oplus \vV_2).
\end{align*}
It is easy to see that their product factors through the embedding~\eqref{bHtJ} 
via a morphism which we denote~$\alpha \colon \tJJ(\bH_1, \bH_2) \to \bH(\tJ(\bP(V_1), \bP(V_2)))$.  
This gives~\eqref{tJJ-functor-spaces}. The second claim of the lemma is evident. 
\end{proof}

\subsection{The HPD functor for categorical joins} 
\label{subsection-joins-HPD-functor}
Let $\cA^1$ and $\cA^2$ be Lefschetz categories over~$\bP(V_1)$ and $\bP(V_2)$. 
For $k = 1,2$, we denote by 
\begin{equation*}
\gamma_k \colon (\cA^k)^{\hpd} \to \bH(\cA^k) 
\end{equation*}
the canonical $\bP(\vV_k)$-linear inclusion functor. 
Then~\eqref{eq:tj-gg} defines a $\bP(\vV_1 \oplus \vV_2)$-linear functor 
\begin{equation*}
\tJ(\gamma_1, \gamma_2) \colon \tJv((\cA^1)^{\hpd}, (\cA^2)^{\hpd}) \to \tJv(\bH(\cA^1),\bH(\cA^2)). 
\end{equation*} 
By base changing the $\bP(V_1) \times \bP(V_2)$-linear category $\cA^1 \otimes \cA^2$ along 
diagram~\eqref{tJJ-functor-spaces}, we obtain a diagram of functors 
\begin{equation*}
\xymatrix@C=1.5em{
&&& \tJJ(\bH(\cA^1), \bH(\cA^2)) \ar[dr]^{\alpha_*}
\\
\tJv((\cA^1)^{\hpd}, (\cA^2)^{\hpd}) \ar[rr]^-{\tJ(\gamma_1, \gamma_2)} &&
\tJv(\bH(\cA^1),\bH(\cA^2)) \ar[ur]^{\tilde{p}^*} && \bH(\tJ(\cA^1,\cA^2)).
}
\end{equation*}
Since the diagram~\eqref{tJJ-functor-spaces} is over $\bP(\vV_1 \oplus \vV_2)$,
all of the above functors are~$\bP(\vV_1 \oplus \vV_2)$-linear.

By composing the functors in the diagram, we obtain a $\bP(\vV_1 \oplus \vV_2)$-linear functor 
\begin{equation}
\label{gammatJ}
\gamma_{\tJ} = \alpha_* \circ \tp^* \circ \tJ(\gamma_1, \gamma_2) 
\colon \tJv((\cA^1)^{\hpd}, (\cA^2)^{\hpd}) \to \bH(\tJ(\cA^1,\cA^2)) .  
\end{equation}
Our goal is to show that $\gamma_{\tJ}$ induces the desired equivalence $\cJ((\cA^1)^{\hpd}, (\cA^2)^{\hpd}) \simeq \cJ(\cA^1, \cA^2)^{\hpd}$ 
when~$\cA^1$ and $\cA^2$ satisfy the assumptions of Theorem~\ref{theorem-joins-HPD}. 

The following observation will be needed later.

\begin{lemma}
\label{lemma:gtj-adjoints}
The functor~$\gamma_{\tJ}$ has both left and right adjoints.
\end{lemma}
\begin{proof}
The functors $\gamma_k$ have both left and right adjoints by Lemma~\ref{lemma:hpd-sod}.
Therefore, $\tJ(\gamma_1, \gamma_2)$ has both left and right adjoints by Lemma~\ref{lemma:tj-functoriality}.
On the other hand, the functors $\alpha_*$ and $\tp^*$ have both left and right adjoints 
by Lemma~\ref{lemma-tJJ-functor-spaces} and Remark~\ref{remark:good-morphism}.
\end{proof}

\begin{remark}
\label{remark:fm-join}
The functor $\gamma_{\tJ}$ can be described in terms of Fourier--Mukai kernels. 
For simplicity, in this remark we restrict ourselves to the commutative case
as in~\cite{kuznetsov-hpd}, but 
the same description works in general using the formalism of Fourier--Mukai kernels from \cite[\S5]{NCHPD}. 
Namely, assume for $k=1,2$ we are given a smooth and proper $\bP(V_k)$-scheme $X_k$ with an admissible $\bP(V_k)$-linear 
subcategory $\cA^k \subset \Perf(X_k)$, a smooth and proper $\bP(\vV_k)$-scheme~$Y_k$ with an admissible 
$\bP(\vV_k)$-linear subcategory~$\cB^k \subset \Perf(Y_k)$, and a $\bP(\vV_k)$-linear Fourier--Mukai functor 
\begin{equation*}
\Phi_{\cE_k} \colon \Perf(Y_k) \to \Perf(\bH(X_k)) ,  \quad  \cE_k \in \Perf \left(\bH(X_k) \times_{\bP(\vV_k)} Y_k \right) , 
\end{equation*}
such that $\Phi_{\cE_k}$ is a left splitting functor in the sense of \cite[Definition 3.1]{kuznetsov-hpd}, 
$(\cA^k)^{\hpd}$ is the image of $\Phi_{\cE_k}$, and $\cB_k$ is the image of $\Phi_{\cE_k}^*$. 
Note that by \cite[Theorem 3.3]{kuznetsov-hpd} the functors $\Phi_{\cE_k}$ and $\Phi_{\cE_k}^*$ then induce 
mutually inverse equivalences $\cB^k \simeq (\cA^k)^{\hpd}$.  

In this situation, we claim there is a $\bP(\vV_1 \oplus \vV_2)$-linear Fourier--Mukai functor  
\begin{equation*}
\Phi_{\cE} \colon 
\Perf(\tJv(Y_1, Y_2)) \to \Perf(\bH(\tJ(X_1, X_2))) ,  
\ 
\cE \in \Perf\left( \bH(\tJ(X_1, X_2)) \times_{\bP(\vV_1 \oplus \vV_2)} \tJv(Y_1, Y_2)  \right)
\end{equation*} 
which restricts to the functor 
\begin{equation*}
\tJv(\cB^1, \cB^2) \simeq \tJv((\cA^1)^{\hpd}, (\cA^2)^{\hpd}) \xrightarrow{\, \gamma_{\tJ} \,} \bH(\tJ(\cA^1,\cA^2)) . 
\end{equation*} 
For this, consider the commutative diagram 
\begin{equation*}
\xymatrix{
\tJJ(\bH_1, \bH_2) \ar[r]^{u} \ar[d]_{v} & \bH_1 \times \bH_2 \ar[d] \\ 
\bH(\tJ(\bP(V_1), \bP(V_2))) \times_{\bP(\vV_1 \oplus \vV_2)} \tJv(\bP(\vV_1), \bP(\vV_2)) \ar[r] & 
\bP(V_1) \times \bP(V_2) \times \bP(\vV_1) \times \bP(\vV_2) 
}
\end{equation*}
where $u = q_{\bH} \circ \tilde{p} = p_{\bH} \circ \tilde{q}$ and~$v$ is induced 
by the morphisms~$\wtilde{h_1 \times h_2} \circ \tilde{p}$ in~\eqref{tJJH1H2-diagram} and~$\alpha$ in~\eqref{tJJ-functor-spaces}.
Base changing from $\bP(V_1), \bP(V_2), \bP(\vV_1), \bP(\vV_2)$, to 
$X_1, X_2, Y_1, Y_2$, gives a commutative diagram 
\begin{equation*}
\xymatrix{
\tJJ(\bH(X_1), \bH(X_2)) \ar[r]^-{u} \ar[d]_{v} & \left( \bH(X_1) \times_{\bP(\vV_1)} Y_1 \right) \times \left( \bH(X_2) \times_{\bP(\vV_2)} Y_2 \right) \ar[d] \\ 
\bH(\tJ(X_1, X_2)) \times_{\bP(\vV_1 \oplus \vV_2)} \tJv(Y_1, Y_2) \ar[r] & 
X_1 \times X_2 \times Y_1 \times Y_2 
}
\end{equation*}
where we abusively still denote the top and left maps by $u$ and $v$. 
It is straightforward to verify that 
the object 
\begin{equation}
\label{eq:kernel-join}
\cE = v_*u^*(\cE_1 \boxtimes \cE_2)  \in 
\Perf \left( \bH(\tJ(X_1, X_2)) \times_{\bP(\vV_1 \oplus \vV_2)} \tJv(Y_1, Y_2)  \right)
\end{equation}
is the desired Fourier--Mukai kernel. 
\end{remark}

\subsection{Relations between the HPD functors}
\label{subsection-HPD-functor-relations}

Let $H$ and $H'$ denote the relative hyperplane classes on $\bP(V_1 \oplus V_2)$ 
and $\bP(\vV_1 \oplus \vV_2)$. 
As in \S\ref{subsection-HPD-characterization}, 
let $\cM$ be the cohomology sheaf of the monad 
\begin{equation}
\label{cR-monad} 
\cM \simeq \{ \cO(-H) \to (V_1 \oplus V_2) \otimes \cO \to \cO(H') \}
\end{equation}
on $\bH(\bP(V_1 \oplus V_2))$. 
Similarly, for $k=1,2$, we let $\cM_k$ be the cohomology sheaf of the monad
\begin{equation}
\label{cRk-monad} 
\cM_k \simeq \{ \cO(-H_k) \to V_k \otimes \cO \to \cO(H_k') \} 
\end{equation}
on $\bH_k$. 
Pushforward along the morphisms 
\begin{align*}
\pi_{\tJ} & \colon \bH(\tJ(\bP(V_1), \bP(V_2))) \to \tJ(\bP(V_1), \bP(V_2)) \\ 
\pi_{k} & \colon \bH_k \to \bP(V_k), \qquad k = 1,2, 
\end{align*}
induces functors 
\begin{align*}
\pi_{\tJ*} & \colon \bH(\tJ(\cA^1, \cA^2)) \to \tJ(\cA^1, \cA^2) , \\ 
\pi_{k*} & \colon \bH(\cA^k) \to \cA^k, \qquad k =1,2. 
\end{align*}
For $t \geq 0$, we aim to relate the composition 
\begin{equation}
\label{pi-PhiE}
\pi_{\tJ*} \circ (- \otimes \wedge^t \cM) \circ \gamma_\tJ \colon 
\tJv((\cA^1)^{\hpd}, (\cA^2)^{\hpd})) \to \tJ(\cA^1, \cA^2) 
\end{equation}
to the analogous compositions 
\begin{equation}
\label{pik-PhiEk}
\pi_{k*} \circ (- \otimes \wedge^t \cM_k) \circ \gamma_k \colon (\cA^k)^{\hpd} \to \cA^k, \qquad k=1,2. 
\end{equation}
Combined with the results of~\S\ref{subsection-HPD-characterization}, 
this will be the key ingredient in our proof that $\gamma_{\tJ}$ induces 
the equivalence of Theorem~\ref{theorem-joins-HPD}. 
The following result handles the case $t = 0$. 
We use the notation from diagram~\eqref{tJJH1H2-diagram}.

\begin{proposition}
\label{proposition-pi-gammatJ-1}
There is an isomorphism
\begin{equation*}
\pi_{\tJ*} \circ \gamma_\tJ \simeq p^* \circ ((\pi_{1*} \circ \gamma_1) \otimes (\pi_{2*} \circ \gamma_2)) \circ q_*  
\end{equation*}
of functors $\tJv((\cA^1)^{\hpd}, (\cA^2)^{\hpd}) \to \tJ(\cA^1, \cA^2)$. 
\end{proposition}

\begin{proof}
Consider the commutative diagram
\begin{equation}
\label{eq:big-diagram}
\vcenter{\xymatrix{
& \tJJ(\bH_1, \bH_2) \ar[dl]_{\tilde{p}} \ar[dr]^{\alpha} \ar[d]_{\tilde{q}}
\\
\tJv(\bH_1,\bH_2) \ar[d]_{q_\bH} \ar@{}[r]|-\bigstar &
\tJ(\bH_1,\bH_2) \ar[dl]_{p_\bH} \ar[dr]^{\widetilde{\pi_1\times \pi_2}} \ar@{}[dd]|-\bigstar &
\bH(\tJ(\bP(V_1),\bP(V_2))) \ar[d]^{\pi_\tJ}
\\
\bH_1 \times \bH_2 \ar[dr]^{\pi_1 \times \pi_2} &&
\tJ(\bP(V_1),\bP(V_2)) \ar[dl]_p
\\
& \bP(V_1) \times \bP(V_2)
}}
\end{equation}
where the squares marked by $\star$ are cartesian and $\Tor$-independent since $p$, being a $\bP^1$-bundle, is flat.
Therefore, we have a chain of isomorphisms
\begin{align*}
\pi_{\tJ*} \circ \alpha_* \circ \tilde{p}^* 
&\simeq (\widetilde{\pi_1\times \pi_2})_* \circ \tilde{q}_* \circ \tilde{p}^* \\
&\simeq (\widetilde{\pi_1\times \pi_2})_* \circ p_\bH^* \circ q_{\bH*} \\
&\simeq p^* \circ ({\pi_1\times \pi_2})_* \circ q_{\bH*}  
\end{align*}
of functors $\Perf(\tJv(\bH_1, \bH_2)) \to \Perf(\tJ(\bP(V_1), \bP(V_2)))$. 
After base change from $\bP(V_1)$ to $\cA^1$ and from $\bP(V_2)$ to $\cA^2$ and composition with the functor $\tJ(\gamma_1,\gamma_2)$, 
we obtain an isomorphism
\begin{equation*}
\pi_{\tJ*} \circ \alpha_* \circ \tilde{p}^* \circ \tJ(\gamma_1,\gamma_2) \simeq
p^* \circ ({\pi_1\times \pi_2})_* \circ q_{\bH*} \circ \tJ(\gamma_1,\gamma_2)
\end{equation*}
of functors $\tJv((\cA^1)^{\hpd},(\cA^2)^{\hpd}) \to \tJ(\cA^1,\cA^2)$.
The left hand side is $\pi_{\tJ*} \circ \gamma_\tJ$. 
On the other hand, from the commutative diagram of Lemma~\ref{lemma:tj-functoriality} we obtain an isomorphism
\begin{equation*}
q_{\bH*} \circ \tJ(\gamma_1,\gamma_2) \simeq (\gamma_1 \sotimes \gamma_2) \circ q_*,
\end{equation*} 
that allows us to rewrite the right hand side as
\begin{equation*}
p^* \circ ({\pi_1\times \pi_2})_* \circ (\gamma_1 \sotimes \gamma_2) \circ q_*,
\end{equation*}
which is equivalent to the right hand side in the statement of the proposition.
\end{proof}

To relate the functor~\eqref{pi-PhiE} to the functors~\eqref{pik-PhiEk} 
for arbitrary $t \geq 0$, we will need the following lemma. 
By definition $\tJJ(\bH_1, \bH_2)$ admits projections to $\bH_1$ 
and $\bH_2$, and also maps to $\bH(\bP(V_1 \oplus V_2))$ via 
the composition 
\begin{equation*}
\tJJ(\bH_1, \bH_2) \xrightarrow{\ \alpha\ } \bH(\bJ(\bP(V_1), \bP(V_2))) 
\xrightarrow{\ \bH(f)\ } \bH(\bP(V_1 \oplus V_2)). 
\end{equation*}
We denote by $\tM, \tM_{1}, \tM_2$, the pullbacks to $\tJJ(\bH_1, \bH_2)$ 
of the sheaves $\cM, \cM_1, \cM_2$, defined by~\eqref{cR-monad} and~\eqref{cRk-monad} above. 
Note that by Lemma~\ref{lemma-tJ-divisors}\eqref{omega-p} we have the formula
\begin{equation}
\label{omegap} 
\omegatp
= \cO(H_1 + H_2 - 2H)[1] 
\end{equation} 
for the relative dualizing complex of the morphism $\tp \colon \tJJ(\bH_1, \bH_2) \to \tJv(\bH_1, \bH_2)$.

As in Remark~\ref{remark:fm-join} we write 
\begin{equation*}
u = q_\bH \circ \tilde{p} = p_\bH \circ \tilde{q} \colon \tJJ(\bH_1, \bH_2) \to \bH_1 \times \bH_2
\end{equation*}
for the canonical $\bP^1 \times \bP^1$-bundle, see~\eqref{eq:big-diagram}. 
For $k=1,2$, we write 
\begin{equation*}
\pr_k \colon \bH_1 \times \bH_2 \to \bH_k
\end{equation*}
for the projection. 

\begin{lemma}
\label{lemma-wedge-tR}
There is an isomorphism of sheaves on~$\bH_1 \times \bH_2$
\begin{equation*}
u_*(\omegatp \otimes \wedge^t \tM) \simeq \wedge^t(\pr_1^*\cM_1 \oplus \pr_2^*\cM_2) . 
\end{equation*}
\end{lemma} 

\begin{proof}
The pullback to $\tJJ(\bH_1, \bH_2)$ of the monad~\eqref{cR-monad} and 
the direct sum of the pullbacks to $\tJJ(\bH_1, \bH_2)$ 
of the monads~\eqref{cRk-monad} fit 
into the following bicomplex on $\tJJ(\bH_1, \bH_2)$: 
\begin{equation*}
\xymatrix{
& \cO(H'_1) \oplus \cO(H'_2) \ar[r] &
\cO(H')
\\
\cO(-H) \ar[r] &
(V_1 \oplus V_2) \otimes \cO \ar[r] \ar[u] &
\cO(H') \ar@{=}[u] 
\\
\cO(-H) \ar[r] \ar@{=}[u] &
\cO(-H_1) \oplus \cO(-H_2) \ar[u] 
}
\end{equation*}
The nontrivial cohomology sheaves of its rows with respect to the horizontal differential 
are given by \mbox{$\cO(H-H_1-H_2)$}, $\tM$, $\cO(H'_1+H'_2-H')$ in degrees $(0,-1), (0,0), (0,1)$. 
The only nontrivial cohomology sheaf with respect to the vertical differential 
is $\tM_1 \oplus \tM_2$ in degree~$(0,0)$. 
It follows that there is a filtration of $\tM$ whose associated graded is 
\begin{equation*}
\cO(H-H_1-H_2) \oplus (\tM_1 \oplus \tM_2) \oplus \cO(H'_1+H'_2-H'). 
\end{equation*}
Hence $\wedge^t \tM$ has a filtration whose associated graded is 
\begin{align*}
& \left( \wedge^{t-1}(\tM_1 \oplus \tM_2) \otimes \cO(H-H_1-H_2) \right) \\  
\oplus & \left( \wedge^t(\tM_1 \oplus \tM_2)  \right) \\ 
\oplus & \left( \wedge^{t-2}(\tM_1 \oplus \tM_2) \otimes \cO(H-H_1-H_2+H'_1+H'_2-H') \right) \\ 
\oplus & \left( \wedge^{t-1}(\tM_1 \oplus \tM_2) \otimes \cO(H'_1+H'_2-H') \right). 
\end{align*} 
Since $u \colon \tJJ(\bH_1, \bH_2) \to \bH_1 \times \bH_2$ is a $\bP^1 \times \bP^1$-bundle 
(whose relative hyperplane classes are~$H$ and $H'$), we have
\begin{equation*}
u_*(\cO(aH + a'H')) = 0 \quad \text{if either $a = -1$ or $a' = -1$.} 
\end{equation*}
Moreover, since $u = q_{\bH} \circ \tp$ where $q_{\bH}$ and $\tp$ are 
$\bP^1$-bundles, we have 
\begin{equation*}
u_*(\omegatp)  \simeq \cO. 
\end{equation*}
Hence by the formula~\eqref{omegap} for $\omegatp$ and the above description of the 
associated graded of the filtration of $\wedge^t \tM$, we find 
\begin{equation*}
u_*(\omegatp \otimes \wedge^t \tM) \simeq \wedge^t(\pr_1^*\cM_1 \oplus \pr_2^*\cM_2)  
\end{equation*}
as desired. 
\end{proof} 

Recall that $p_!$ denotes the left adjoint of $p^* \colon  \Perf(\bP(V_1) \stimes \bP(V_2)) \to \Perf(\tJ(\bP(V_1), \bP(V_2)))$. 
Below $\cM$ is considered as a vector bundle on~$\bH(\bJ(\bP(V_1), \bP(V_2)))$ 
(pulled back along~$\bH(f)$ from~$\bH(\bP(V_1 \oplus V_2))$). 

\begin{proposition}
\label{proposition-pi-gammatJ-2} 
There is an isomorphism 
\begin{equation*}
p_! \circ \pi_{\tJ*} \circ (- \otimes \wedge^t \cM) \circ  \gamma_{\tJ} \circ q^* 
\simeq 
\bigoplus_{t_1 + t_2 = t} (\pi_{1*} \circ (- \otimes \wedge^{t_1} \cM_{1}) \circ \gamma_1) \otimes 
(\pi_{2*} \circ (- \otimes \wedge^{t_2} \cM_{2}) \circ \gamma_2) 
\end{equation*}
of functors $(\cA^1)^{\hpd} \sotimes (\cA^2)^{\hpd} \to \cA^1 \sotimes \cA^2$. 
\end{proposition}

\begin{proof}
The proof is similar to that of Proposition~\ref{proposition-pi-gammatJ-1}.
First, using the commutative diagram in Lemma~\ref{lemma:tj-functoriality}, we obtain an isomorphism 
\begin{equation*}
\tJ(\gamma_1,\gamma_2) \circ q^* \simeq q_\bH^* \circ (\gamma_1 \otimes \gamma_2).
\end{equation*}
Using the definition of $\gamma_\tJ$, this allows us to rewrite the left hand side in 
the statement of the proposition as 
\begin{equation*}
p_! \circ \pi_{\tJ*} \circ (- \otimes \wedge^t \cM) \circ \alpha_* \circ \tilde{p}^* \circ q_\bH^* \circ (\gamma_1 \sotimes \gamma_2).
\end{equation*}
By the projection formula, we can rewrite this as 
\begin{equation*}
p_! \circ \pi_{\tJ*} \circ \alpha_* \circ (- \otimes \wedge^t \tM) \circ \tilde{p}^* \circ q_\bH^* \circ (\gamma_1 \sotimes \gamma_2). 
\end{equation*}
Note that $p_! = p_* \circ (- \otimes \omegap)$ by~\eqref{eq:shriek-adjoints}.
Further, the pullback of $\omegap$ to $\tJJ(\bH_1, \bH_2)$ is $\omegatp$, since~$\tp$ is a base change of $p$. 
Hence again by the projection formula, 
we can rewrite the above composition as 
\begin{equation*}
p_* \circ \pi_{\tJ*} \circ \alpha_* \circ (- \otimes \omegatp \otimes \wedge^t \tM ) \circ \tilde{p}^* \circ q_\bH^* \circ (\gamma_1 \sotimes \gamma_2).
\end{equation*}
By commutativity of the diagrams~\eqref{eq:big-diagram} and~\eqref{tJJH1H2-diagram}
we have 
$p \circ \pi_{\tJ} \circ \alpha = (\pi_1 \times \pi_2) \circ u$, where recall~$u = q_{\bH} \circ \tp$. 
Hence we can rewrite the above composition as 
\begin{equation*}
(\pi_1 \times \pi_2)_* \circ u_* \circ (- \otimes \omegatp \otimes \wedge^t \tM ) \circ u^* \circ (\gamma_1 \sotimes \gamma_2).
\end{equation*}
By the projection formula, the composition $u_* \circ (- \otimes \omegatp \otimes \wedge^t \tM ) \circ u^*$
is equivalent to the functor given by tensoring with the object 
$u_*( \omegatp \otimes \wedge^t \tM)$. 
But by Lemma~\ref{lemma-wedge-tR} this object is isomorphic to $\wedge^t(\pr_1^*\cM_1 \oplus \pr_2^*\cM_2)$.
Therefore, the functor we are interested in is equivalent to the direct sum of the functors
\begin{equation*}
(\pi_1 \times \pi_2)_* \circ (- \otimes (\wedge^{t_1} \pr_1^*\cM_1 \otimes \wedge^{t_2} \pr_2^*\cM_2)) \circ (\gamma_1 \sotimes \gamma_2),
\end{equation*}
over all $t_1 + t_2 = t$.
It remains to note that each summand is isomorphic to the corresponding summand in the right hand side of the statement of the proposition. 
\end{proof}

\subsection{Proof of the main theorem}
\label{subsection-joins-HPD-proof} 
The categorical join $\cJ(\cA^1, \cA^2)$ fits into a $\bP(V_1 \oplus V_2)$-linear 
semiorthogonal decomposition~\eqref{sod-tJ}, which we can write in a simplified form as 
\begin{equation*}
\tJ(\cA^1, \cA^2) = \llangle \cJ(\cA^1, \cA^2), {}^{\perp}\cJ(\cA^1, \cA^2) \rrangle.
\end{equation*} 
By Lemma~\ref{lemma-bH-sod} we have a semiorthogonal decomposition 
\begin{equation*}
\bH(\tJ(\cA^1, \cA^2)) = 
\llangle \bH(\cJ(\cA^1, \cA^2)), \bH({}^{\perp}\cJ(\cA^1, \cA^2) ) \rrangle. 
\end{equation*}
The HPD category $\cJ(\cA^1, \cA^2)^{\hpd}$ is a $\bP(\vV_1 \oplus \vV_2)$-linear subcategory 
of $\bH(\cJ(\cA^1, \cA^2))$, and hence also of $\bH(\tJ(\cA^1, \cA^2))$. 

We will prove the following more precise version of Theorem~\ref{theorem-joins-HPD}. 
\begin{theorem} 
\label{theorem-joins-HPD-precise} 
Let $\cA^1$ and $\cA^2$ be right strong, moderate Lefschetz categories over $\bP(V_1)$ and~$\bP(V_2)$. 
The functor 
\begin{equation*}
\gamma_{\tJ} 
\colon {\tJv}((\cA^1)^{\hpd}, (\cA^2)^{\hpd}) \to \bH(\tJ(\cA^1,\cA^2)) 
\end{equation*}
defined by~\eqref{gammatJ} induces a Lefschetz equivalence between the subcategories 
\begin{equation*}
\cJ((\cA^1)^{\hpd}, (\cA^2)^{\hpd}) \subset {\tJv}((\cA^1)^{\hpd}, (\cA^2)^{\hpd}) 
\quad \text{and} \quad 
\cJ(\cA^1, \cA^2)^{\hpd} \subset \bH(\tJ(\cA^1,\cA^2))  . 
\end{equation*}
\end{theorem}

The following lemma guarantees that $\gamma_{\tJ}$ does indeed 
induce a functor between the categories of interest. 
\begin{lemma}
\label{lemma-image-PhiE}
The image of the functor $\gamma_\tJ \colon {\tJv}((\cA^1)^{\hpd}, (\cA^2)^{\hpd}) \to \bH(\tJ(\cA^1, \cA^2))$ 
is contained in the HPD category 
$\cJ(\cA^1, \cA^2)^{\hpd} \subset \bH(\tJ(\cA^1, \cA^2))$. 
\end{lemma}

\begin{proof}
By Lemmas~\ref{lemma-characterization-Cd} and~\ref{lemma-bH-sod} together 
with $\bP(\vV_1 \oplus \vV_2)$-linearity of the functor $\gamma_\tJ$, it suffices to show that the 
image of $\pi_{\tJ*} \circ \gamma_\tJ$ is contained in the Lefschetz 
center 
\begin{equation*}
\cJ(\cA^1, \cA^2)_0 \subset \cJ(\cA^1, \cA^2) \subset \tJ(\cA^1, \cA^2).
\end{equation*}  
By Proposition~\ref{proposition-pi-gammatJ-1}, we have 
\begin{equation*}
\pi_{\tJ*} \circ \gamma_\tJ \simeq p^* \circ ((\pi_{1*} \circ \gamma_1) \otimes (\pi_{2*} \circ \gamma_2)) \circ q_*.  
\end{equation*} 
By Lemma~~\ref{lemma-characterization-Cd} the image of $\pi_{k*} \circ \gamma_k$ is $\cA^k_0 \subset \cA^k$, 
hence the image of $\pi_{\tJ*} \circ \gamma_\tJ$ is contained in~$p^*(\cA^1_0 \sotimes \cA^2_0)$,
which is nothing but $\cJ(\cA^1, \cA^2)_0$ by definition~\eqref{eq:cat-join-center}.
\end{proof} 

By Lemma~\ref{lemma-image-PhiE}, the restriction of the functor 
$\gamma_\tJ$ to the $\bP(\vV_1 \oplus \vV_2)$-linear subcategory 
$\cJ((\cA^1)^{\hpd}, (\cA^2)^{\hpd}) \subset {\tJv}((\cA^1)^{\hpd}, (\cA^2)^{\hpd})$ induces a $\bP(\vV_1 \oplus \vV_2)$-linear functor 
\begin{equation*}
\phi \colon \cJ((\cA^1)^{\hpd}, (\cA^2)^{\hpd}) \to \cJ(\cA^1, \cA^2)^{\hpd} 
\end{equation*}
which fits into a commutative diagram 
\begin{equation}
\label{diagram:gamma-phi}
\vcenter{
\xymatrix{
\cJ((\cA^1)^{\hpd}, (\cA^2)^{\hpd}) \ar[d]_{j} \ar[r]^{\phi} & \cJ(\cA^1, \cA^2)^{\hpd}  \ar[d]^{\gamma} \\ 
{\tJv}((\cA^1)^{\hpd}, (\cA^2)^{\hpd}) \ar[r]^{\gamma_{\tJ}} & \bH(\tJ(\cA^1,\cA^2)) 
}
}
\end{equation}
where $j$ and $\gamma$ are the inclusions. 
Our goal is to show~$\phi$ is an equivalence of Lefschetz categories. 
We will prove this by verifying the criteria of Lemma~\ref{lemma-equivalence-lef-cat}.  

\begin{lemma}
\label{lemma-phiE-B0-in-Ad0}
The functor $\phi$ takes the Lefschetz center $\cJ((\cA^1)^{\hpd}, (\cA^2)^{\hpd})_0 \subset \cJ((\cA^1)^{\hpd}, (\cA^2)^{\hpd})$ to the Lefschetz center 
$\cJ(\cA^1, \cA^2)^{\hpd}_0 \subset \cJ(\cA^1, \cA^2)^{\hpd}$. 
\end{lemma} 

\begin{proof}
By Proposition~\ref{proposition-characterization-Ad0} 
we must show that for any  
$C \in \cJ((\cA^1)^{\hpd}, (\cA^2)^{\hpd})_0$ and $t \geq 1$, we have  
\begin{equation*}
\pi_{\tJ*}(\gamma_\tJ(C) \otimes \wedge^t \cM) \in {}^{\perp} \cJ(\cA^1, \cA^2)_0 . 
\end{equation*} 
Since $\cJ(\cA^1, \cA^2)_0 = p^*(\cA^1_0 \sotimes \cA^2_0)$, by 
adjunction the desired conclusion is equivalent to 
\begin{equation*}
p_!\pi_{\tJ*}(\gamma_\tJ(C) \otimes \wedge^t \cM) \in {}^{\perp}(\cA_0^1 \sotimes \cA^2_0). 
\end{equation*} 
Since $\cJ((\cA^1)^{\hpd}, (\cA^2)^{\hpd})_0  = 
q^*((\cA^1)^{\hpd}_0 \sotimes (\cA^2)^{\hpd}_0)$, by Lemma~\ref{lemma-generators-box-tensor}
we may assume that $C$ is of the form $C = q^*(C_1 \boxtimes C_2)$ for $C_k \in (\cA^k)^{\hpd}_0$. 
Then by Proposition~\ref{proposition-pi-gammatJ-2}, we have 
\begin{equation*}
p_!\pi_{\tJ*}(\gamma_\tJ(q^*(C_1 \boxtimes C_2)) \otimes \wedge^t \cM) \simeq 
\bigoplus_{t_1 + t_2 = t} ( \pi_{1*}(\gamma_1(C_1) \otimes \wedge^{t_1}\cM_1)) 
\otimes  ( \pi_{2*}(\gamma_2(C_2) \otimes \wedge^{t_2}\cM_2)) . 
\end{equation*}
By Proposition~\ref{proposition-characterization-Ad0} again, 
we have 
\begin{align*}
\pi_{1*}(\gamma_1(C_1) \otimes \wedge^{t_1} \cM_1) & 
\in {}^{\perp} \cA^1_0  \quad \text{ if } t_1 \geq 1 , \\
\pi_{2*}(\gamma_2(C_2) \otimes \wedge^{t_2} \cM_2) & 
\in {}^{\perp} \cA^2_0  \quad \text{ if } t_2 \geq 1 . 
\end{align*}
Since $t \geq 1$, if $t_1 + t_2 = t$ then either $t_1 \geq 1$ or $t_2 \geq 1$. 
It follows that if $t_1 \geq 1$ then the $(t_1, t_2)$ summand in the above expression 
lies in $({}^{\perp} \cA^1_0) \sotimes \cA^2$,  
and if $t_2 \geq 1$ then it lies in $\cA^1 \sotimes ({}^{\perp} \cA^2_0)$. 
We conclude by Lemma~\ref{lemma-sod-tensor} that every summand in  
the above expression lies in ${}^{\perp}(\cA_0^1 \sotimes \cA^2_0)$, 
and hence so does their sum 
$p_!\pi_{\tJ*}(\gamma_\tJ(q^*(C_1 \boxtimes C_2)) \otimes \wedge^t \cM)$. 
\end{proof}

\begin{lemma}
\label{lemma-phiE-B0-equivalent-Ad0}
The functor $\phi$ induces an equivalence $\cJ((\cA^1)^{\hpd}, (\cA^2)^{\hpd})_0 
\simeq \cJ(\cA^1, \cA^2)^{\hpd}_0$. 
\end{lemma}

\begin{proof}
Consider the diagram
\begin{equation*}
\xymatrix@C=4em{
&
\cJ((\cA^1)^{\hpd}, (\cA^2)^{\hpd}) \ar[d]_{j} \ar[r]^{\phi} & 
\cJ(\cA^1, \cA^2)^{\hpd}  \ar[d]^{\gamma} 
\\
&
{\tJv}((\cA^1)^{\hpd}, (\cA^2)^{\hpd}) \ar[r]^{\gamma_{\tJ}} \ar[dl]_{q_*} & 
\bH(\tJ(\cA^1,\cA^2))\ar[dr]^{\pi_{\tJ*}} 
\\
(\cA^1)^{\hpd} \sotimes (\cA^2)^{\hpd} \ar[r]^-{\gamma_1 \otimes \gamma_2} &
\bH(\cA_1) \otimes \bH(\cA_2) \ar[r]^-{\pi_{1*} \otimes \pi_{2*}} &
\cA_1 \sotimes \cA_2 \ar[r]^-{p^*} &
\tJ(\cA^1,\cA^2), 
}
\end{equation*}
where the top square is the commutative diagram~~\eqref{diagram:gamma-phi}, 
and the bottom part commutes by Proposition~\ref{proposition-pi-gammatJ-1}.
By Lemma~\ref{lemma-phiE-B0-in-Ad0}, $\phi$ induces a functor $\cJ((\cA^1)^{\hpd}, (\cA^2)^{\hpd})_0 \to \cJ(\cA^1, \cA^2)^{\hpd}_0$, and 
by Lemma~\ref{lemma-A0-Ad0-equivalence} the functor $\pi_{\tJ*} \circ \gamma$ 
induces an equivalence $\cJ(\cA^1, \cA^2)^{\hpd}_0 \simeq \cJ(\cA^1, \cA^2)_0$. 
Hence it suffices to show the composition 
\begin{equation*}
\pi_{\tJ*} \circ \gamma \circ \phi \simeq 
\pi_{\tJ*} \circ \gamma_\tJ \circ j \simeq
p^* \circ ((\pi_{1*} \circ \gamma_1) \otimes (\pi_{2*} \circ \gamma_2)) \circ q_* \circ j
\end{equation*}
induces an equivalence $\cJ((\cA^1)^{\hpd}, (\cA^2)^{\hpd})_0 \simeq \cJ(\cA^1, \cA^2)_0$. 
By the definitions of~$\cJ((\cA^1)^{\hpd}, (\cA^2)^{\hpd})_0$ and $\cJ(\cA^1, \cA^2)_0$, the functor $q_* \circ j$ induces 
an equivalence~$\cJ((\cA^1)^{\hpd}, (\cA^2)^{\hpd})_0 \simeq (\cA^1)_0^{\hpd} \sotimes (\cA^2)^{\hpd}_0$ 
(note that $q_* \circ q^* \simeq \id$ as~$q$ is a~$\bP^1$-bundle) and~$p^*$ induces 
an equivalence~$\cA^1_0 \sotimes \cA^2_0 \simeq \cJ(\cA^1, \cA^2)_0$. 
Hence it remains to observe $\pi_{k*} \circ \gamma_k$ induces an equivalence 
$(\cA^k)^{\hpd}_0 \simeq \cA^k_0$ for $k=1,2$, again by Lemma~\ref{lemma-A0-Ad0-equivalence}. 
\end{proof}

\begin{lemma}
\label{lemma-phi-adjoints}
The functor $\phi$ admits a left adjoint $\phi^* \colon \cJ(\cA^1, \cA^2)^{\hpd} \to \cJ((\cA^1)^{\hpd}, (\cA^2)^{\hpd})$.
\end{lemma}

\begin{proof}
Consider the diagram~\eqref{diagram:gamma-phi}. 
The functor $\gamma_\tJ$ has a left adjoint~$\gamma_\tJ^*$ by Lemma~\ref{lemma:gtj-adjoints}, 
and the functor $j$ has a left adjoint $j^*$ by Lemma~\ref{lemma-tJ}. 
Since the functor $\gamma$ is fully faithful, it follows that $j^* \circ \gamma_\tJ^* \circ \gamma$ is left adjoint to $\phi$. 
\end{proof} 

\begin{lemma}
\label{lemma-phiE-adjoint-Ad0-equivalent-B0}
The functor $\phi^*$ induces an equivalence $\cJ(\cA^1, \cA^2)^{\hpd}_0 \simeq  \cJ((\cA^1)^{\hpd}, (\cA^2)^{\hpd})_0$. 
\end{lemma}

\begin{proof}
By Lemma~\ref{lemma-A0-Ad0-equivalence} the functor 
\mbox{$\gamma^* \circ \pi_\tJ^* \colon \tJ(\cA^1, \cA^2) \to \cJ(\cA^1, \cA^2)^{\hpd}$} 
induces an equivalence \mbox{$\cJ(\cA^1, \cA^2)_0 \simeq \cJ(\cA^1, \cA^2)^{\hpd}_0$}. 
So, it suffices to show that the composition $\phi^* \circ \gamma^* \circ \pi_\tJ^*$
induces an equivalence $\cJ(\cA^1, \cA^2)_0 \simeq \cJ((\cA^1)^{\hpd}, (\cA^2)^{\hpd})_0$.

On the one hand, by taking left adjoints in the diagram~\eqref{diagram:gamma-phi} and composing with $\pi_{\tJ}^*$ we obtain
\begin{equation*}
\phi^* \circ \gamma^* \circ \pi_\tJ^* \simeq j^* \circ \gamma_{\tJ}^* \circ \pi_{\tJ}^*.
\end{equation*}
On the other hand, taking left adjoints in Proposition~\ref{proposition-pi-gammatJ-1} and composing with~$j^*$ gives
\begin{equation*}
j^* \circ  \gamma_\tJ^* \circ \pi_\tJ^* \simeq 
j^* \circ q^* \circ ((\gamma_1^* \circ \pi_1^*) \otimes (\gamma_2^* \circ \pi_2^*) ) \circ p_!.
\end{equation*} 
So it suffices to show the right side 
induces an equivalence $\cJ(\cA^1, \cA^2)_0 \simeq \cJ((\cA^1)^{\hpd}, (\cA^2)^{\hpd})_0$. 
By the definitions of $\cJ(\cA^1, \cA^2)_0$ and $\cJ((\cA^1)^{\hpd}, (\cA^2)^{\hpd}))_0$, 
the functor $p_!$ induces an equivalence~$\cJ(\cA^1, \cA^2)_0 \simeq \cA^1_0 \sotimes \cA^2_0$ 
(note that $p_! \circ p^* \simeq \id$ as $p$ is a~$\bP^1$-bundle) 
and $j^* \circ q^*$ induces an equivalence $(\cA^1)^{\hpd}_0 \sotimes (\cA^2)^{\hpd}_0 \simeq \cJ((\cA^1)^{\hpd}, (\cA^2)^{\hpd})_0$. 
Hence it remains to observe $\gamma^*_k \circ \pi_k^*$ induces an equivalence~$\cA^k_0 \simeq (\cA^k)^{\hpd}_0$ for $k=1,2$, 
again by Lemma~\ref{lemma-A0-Ad0-equivalence}.
\end{proof}

\begin{proof}[{Proof of Theorems~\textup{\ref{theorem-joins-HPD-precise}} and~\textup{\ref{theorem-joins-HPD}}}]
Lemmas \ref{lemma-phiE-B0-equivalent-Ad0}, \ref{lemma-phi-adjoints}, and \ref{lemma-phiE-adjoint-Ad0-equivalent-B0} 
verify the criteria of Lemma~\ref{lemma-equivalence-lef-cat} for the functor 
\begin{equation*}
\phi \colon \cJ((\cA^1)^{\hpd}, (\cA^2)^{\hpd}) \to \cJ(\cA^1, \cA^2)^{\hpd}
\end{equation*}
to be an equivalence of Lefschetz categories. 
This completes the proof of Theorem~\ref{theorem-joins-HPD-precise}, and hence also of Theorem~\ref{theorem-joins-HPD}.
\end{proof}


\section{Nonlinear HPD theorems}
\label{section-nonlinear-HPD}

In~\S\ref{subsection-nonlinear-HPD-theorem} we prove a nonlinear version of the main theorem of HPD (Theorem~\ref{theorem-HPD}).
We give an extension of this result in \S\ref{subsection-iterated-nonlinear-HPD-theorem}, 
which describes the tensor product of an arbitrary number of Lefschetz 
categories over a projective bundle in terms of a linear section of the categorical 
join of their HPD categories. 

\subsection{The nonlinear HPD theorem} 
\label{subsection-nonlinear-HPD-theorem}

Recall that if $\cA^1$ and $\cA^2$ are Lefschetz categories over $\bP(V_1)$ and~$\bP(V_2)$, 
then Theorem~\ref{theorem-join-lef-cat} provides their categorical join $\cJ(\cA^1,\cA^2)$ 
with the structure of a Lefschetz category over $\bP(V_1 \oplus V_2)$.
We denote by $\cJ_i$, $i \in \bZ$, the Lefschetz components of the categorical join $\cJ(\cA^1,\cA^2)$, 
defined by~\eqref{Ji} and~\eqref{Ji-left}.

\begin{lemma}
\label{lemma-cJL-HPD} 
Let $\cA^1$ and $\cA^2$ be Lefschetz categories over $\bP(V_1)$ and~$\bP(V_2)$. 
For $i \in \bZ$ let $\cJ_i$ be the Lefschetz components of the categorical join 
$\cJ(\cA^1, \cA^2)$. 
Set 
\begin{equation*}
m = \length(\cJ(\cA^1,\cA^2)) = \length(\cA^1) + \length(\cA^2). 
\end{equation*}
Let \mbox{$W \subset V_1 \oplus V_2$} be a subbundle of corank~$s$, and denote by $H$ the relative hyperplane class on $\bP(W)$. 
Then the functor $\cJ(\cA^1, \cA^2) \to \cJ(\cA^1, \cA^2)_{\bP(W)}$ induced by pullback 
along the embedding $\bP(W) \to \bP(V_1 \oplus V_2)$ is fully faithful on $\cJ_i$ for $|i| \geq s$, and 
there are semiorthogonal decompositions 
\begin{align*}
\cJ(\cA^1, \cA^2)_{\bP(W)} & = 
\llangle 
\cK_W(\cJ(\cA^1, \cA^2)), \cJ_s(H), \dots, \cJ_{m-1}((m-s)H) 
\rrangle,    
\\ 
& = 
\llangle 
\cJ_{1-m}((s-m)H), \dots, \cJ_{-s}(-H), \cK'_W(\cJ(\cA^1, \cA^2)) 
\rrangle . 
\end{align*} 
\end{lemma}

\begin{proof}
Follows by combining Theorem~\ref{theorem-join-lef-cat} and Lemma~\ref{lemma-linear-section-lc}. 
\end{proof}

Now assume that the compositions $W \to V_1 \oplus V_2 \to V_1$ and $W \to V_1 \oplus V_2 \to V_2$ are both inclusions,
that is~$\bP(W) \subset \bP(V_1 \oplus V_2)$ is contained in the complement of $\bP(V_1) \sqcup \bP(V_2)$.
Then Proposition~\ref{proposition-cJT} gives an equivalence
\begin{equation}
\label{eq:cj-pl-a1-a2}
\cJ(\cA^1, \cA^2)_{\bP(W)} \simeq \cA^1_{\bP(W)} \otimes_{\Perf(\bP(W))} \cA^2_{\bP(W)}. 
\end{equation} 
If $s$ denotes the corank of $W \subset V_1 \oplus V_2$, then 
by this equivalence and Lemma~\ref{lemma-cJL-HPD} for~$|i| \geq s$ we may consider 
the Lefschetz component $\cJ_i$ of the categorical join $\cJ(\cA^1, \cA^2)$ 
as a subcategory of $\cA^1_{\bP(W)} \otimes_{\Perf(\bP(W))} \cA^2_{\bP(W)}$.

\begin{remark}
\label{remark:cji-in-a1-a2}
Let us directly describe $\cJ_i$, $|i| \geq s$, as a subcategory of 
$\cA^1_{\bP(W)} \otimes_{\Perf(\bP(W))} \cA^2_{\bP(W)}$, without reference to categorical joins.
First note that $\cJ_i$ as a subcategory of $\cJ(\cA^1, \cA^2)$ is the image under the functor~$p^*$ 
of the subcategory $\barcJ_i \subset \cA^1 \otimes \cA^2$ 
described explicitly in Lemma~\ref{lemma-Ji-alternate}. 
Hence for $|i| \geq s$, the subcategory 
\begin{equation*}
\cJ_i \subset \cA^1_{\bP(W)} \otimes_{\Perf(\bP(W))} \cA^2_{\bP(W)}
\end{equation*}
is the fully faithful image of $\barcJ_i$ under the composition 
\begin{equation}
\label{cji-in-a1-a2-1}
\cA^1 \otimes \cA^2 \xrightarrow{\ p^*\ } \cJ(\cA^1, \cA^2) \to \cJ(\cA^1, \cA^2)_{\bP(W)} \xrightarrow{\ \sim\ }
\cA^1_{\bP(W)} \otimes_{\Perf(\bP(W))} \cA^2_{\bP(W)}, 
\end{equation}
where the second functor is given by base change along the inclusion $\bP(W) \to \bP(V_1 \oplus V_2)$. 
It remains to describe this composition without reference to categorical joins. 
The inclusions~$W \to V_1$ and $W \to V_2$ induce a 
morphism $\bP(W) \to \bP(V_1) \times \bP(V_2)$. 
Base changing the~$\bP(V_1) \times \bP(V_2)$-linear category $\cA^1 \otimes \cA^2$ along this morphism 
gives a functor 
\begin{equation}
\label{cji-in-a1-a2-2}
\cA^1 \otimes \cA^2 \to (\cA^1 \otimes \cA^2) \otimes_{\Perf(\bP(V_1) \times \bP(V_2))} \Perf(\bP(W)) 
\simeq \cA^1_{\bP(W)} \otimes_{\Perf(\bP(W))} \cA^2_{\bP(W)}, 
\end{equation}
where the equivalence is given by Corollary~\ref{corollary-base-change-along-diagonal}. 
Alternatively, \eqref{cji-in-a1-a2-2} is given by the tensor product of the restriction functors 
$\cA^1 \to \cA^1_{\bP(W)}$ and $\cA^2 \to \cA^2_{\bP(W)}$. 
Unwinding the definitions shows the functors \eqref{cji-in-a1-a2-1} and \eqref{cji-in-a1-a2-2} 
are isomorphic. 
\end{remark}

Combining Lemma~\ref{lemma-cJL-HPD} with the equivalence~\eqref{eq:cj-pl-a1-a2} we arrive at the following
nonlinear analogue of Lemma~\ref{lemma-linear-section-lc}, which describes linear sections of Lefschetz categories. 
Note that although by Remark~\ref{remark:cji-in-a1-a2} 
this result can be stated without categorical joins, 
the proof uses them.

\begin{corollary}
\label{corollary-fiber-product-lc}
In the setup of Lemma~\textup{\ref{lemma-cJL-HPD}}, 
assume that the subbundle $W \subset V_1 \oplus V_2$ is such that
the compositions $W \to V_1 \oplus V_2 \to V_1$ and $W \to V_1 \oplus V_2 \to V_2$ are both inclusions. 
Then there are semiorthogonal decompositions 
\begin{align*}
\cA^1_{\bP(W)} \otimes_{\Perf(\bP(W))} \cA^2_{\bP(W)} & = 
\llangle \cK_W(\cA^1, \cA^2), \cJ_s(H), \dots, \cJ_{m-1}((m-s)H) \rrangle ,    \\
& = 
\llangle 
\cJ_{1-m}((s-m)H), \dots, \cJ_{-s}(-H), \cK'_W(\cA^1, \cA^2) 
\rrangle . 
\end{align*}
\end{corollary}

\begin{remark}
If the maps $W \to V_1 \oplus V_2 \to V_1$ and $W \to V_1 \oplus V_2 \to V_2$ are isomorphisms 
and~$\cA^2 = \Perf(\bP(L))$ for a subbundle $L \subset V_2$, 
then Corollary~\ref{corollary-fiber-product-lc} reduces to Lemma~\ref{lemma-linear-section-lc}.  
\end{remark}

Now we arrive at the nonlinear HPD theorem. 
Like Corollary~\ref{corollary-fiber-product-lc}, the statement can be explained without appealing to categorical joins, 
but the proof uses them. 

\begin{theorem}
\label{theorem-nonlinear-HPD} 
Let $\cA^1$ and $\cA^2$ be right strong, moderate Lefschetz categories over projective bundles $\bP(V_1)$ and~$\bP(V_2)$. 
For $i , j \in \bZ$ let $\cJ_i$ and $\cJd_j$ be the Lefschetz components 
of the categorical joins $\cJ(\cA^1,\cA^2)$ and $\cJ((\cA^1)^{\hpd},(\cA^2)^{\hpd})$ respectively. 
Assume $V_1$ and $V_2$ have the same rank, and set 
\begin{equation*}
N = \rank(V_1) = \rank(V_2). 
\end{equation*} 
Let $W$ be a vector bundle on $S$ equipped with isomorphisms
\begin{alignat*}{2}
\xi_{k} &\colon W \xrightarrow{\, \sim \,} V_k, &&  k =1,2, 
\intertext{and let}
(\xi_{k}^{\svee})^{-1} &\colon W^{\svee} \xrightarrow{\, \sim \,}  V_k^{\svee}, \quad && k =1,2, 
\end{alignat*} 
be the inverse dual isomorphisms. 
Set 
\begin{equation*}
m = \length(\cA^1) + \length(\cA^2)
\qquad\text{and}\qquad 
n = \length((\cA^1)^{\hpd}) + \length((\cA^2)^{\hpd}). 
\end{equation*}
Denote by $H$ and $H'$ the relative hyperplane classes on~$\bP(W)$ and $\bP(W^{\svee})$.
Then there are semiorthogonal decompositions 
\begin{align*}
\cA^1_{\bP(W)} \otimes_{\Perf(\bP(W))} \cA^2_{\bP(W)} & = 
\llangle \cK_W(\cA^1, \cA^2), \cJ_N(H), \dots, \cJ_{m-1}((m-N)H) \rrangle ,    \\ 
(\cA^1)^{\hpd}_{\bP(W^{\svee})} {\otimes_{\Perf(\bP(W^{\svee}))}} (\cA^2)^{\hpd}_{\bP(W^{\svee})} & = 
\llangle \cJd_{1-n}((N-n)H'), \dots, \cJd_{-N}(-H'), \cK'_{W^{\svee}}((\cA^1)^{\hpd}, (\cA^2)^{\hpd}) \rrangle, 
\end{align*} 
where we consider $\cJ_i$ and $\cJd_j$ as subcategories in the left sides as explained in Remark~\textup{\ref{remark:cji-in-a1-a2}}.
Furthermore, we have an $S$-linear equivalence 
\begin{equation*}
\cK_W(\cA^1, \cA^2) \simeq  \cK'_{W^{\svee}}((\cA^1)^{\hpd}, (\cA^2)^{\hpd}) .  
\end{equation*}
\end{theorem}

\begin{proof}
Consider the inclusion of vector bundles 
\begin{equation*}
(\xi_1, \xi_2) \colon W \to V_1 \oplus V_2 = V.
\end{equation*}  
The orthogonal subbundle is given by the inclusion of 
vector bundles 
\begin{equation*}
((\xi^{\svee}_1)^{-1}, -(\xi^{\svee}_2)^{-1}) \colon W^{\svee} \to V_1^{\svee} \oplus V_2^{\svee} = V^{\svee}.  
\end{equation*} 
Now the semiorthogonal decompositions follow from Corollary~\ref{corollary-fiber-product-lc},
and the equivalence of categories follows from a combination of~\eqref{eq:cj-pl-a1-a2} 
with Theorem~\ref{theorem-joins-HPD} and Theorem~\ref{theorem-HPD}\eqref{HPD-linear},
noting that~$-(\xi^{\svee}_2)^{-1}$ and~$(\xi^{\svee}_2)^{-1}$ induce the same morphism 
$\bP(W^{\svee}) \to \bP(\vV_2)$.
\end{proof}

\begin{remark}
\label{remark:hpd-nonlinear-vs-linear}
As we already mentioned, Theorem~\ref{theorem-nonlinear-HPD} can be regarded as a nonlinear 
version of the main theorem of~HPD, i.e. Theorem~\ref{theorem-HPD}\eqref{HPD-linear}. 
Indeed, consider the linear duality of Example~\ref{ex:categorical-linear-hpd}, 
where~$\cA^2 = \Perf(\bP(L))$ for a subbundle $0 \subsetneq L \subsetneq V_2$ and 
$(\cA^2)^{\hpd} \simeq \Perf(\bP(L^{\perp}))$. 
Now if $W = V_1 = V_2$ and $\xi_1 = \xi_2 = \id$, then 
Theorem~\ref{theorem-nonlinear-HPD} reduces to Theorem~\ref{theorem-HPD}\eqref{HPD-linear}. 
\end{remark}

In the following remark, we explain how to deduce results for 
bounded derived categories of coherent sheaves in place of perfect complexes. 

\begin{remark}
\label{remark-Perf-to-Db}
Given a proper $T$-linear category $\cA$, 
where $T$ is noetherian over a field of characteristic $0$, 
in \cite[Definition 4.27]{NCHPD} a \emph{bounded coherent category} $\cA^{\coh}$ is defined. 
In case~$X \to T$ is a proper morphism of finite presentation, where $X$ is possibly a derived scheme, 
then~$\Perf(X)^{\coh}$ recovers $\Db(X)$. 

By~\cite[Proposition~4.28]{NCHPD}, a semiorthogonal decomposition of $\cA$ (with all 
components except possibly the first or last admissible) induces a semiorthogonal decomposition of $\cA^\coh$. 
This gives rise to a bounded coherent version of Theorem~\ref{theorem-nonlinear-HPD}. 
Namely, assume the categories $\cJ_i$ and $\cJd_j$ appearing in the semiorthogonal decompositions of 
\begin{equation*}
\cC = \cA^1_{\bP(W)} \otimes_{\Perf(\bP(W))} \cA^2_{\bP(W)} 
\quad \text{and} \quad 
\cD = (\cA^1)^{\hpd}_{\bP(W^{\svee})} {\otimes_{\Perf(\bP(W^{\svee}))}} (\cA^2)^{\hpd}_{\bP(W^{\svee})}
\end{equation*} 
are admissible. 
For instance, this is automatic if $\cA^1$ and $\cA^2$ are smooth and proper over $S$, by 
Lemma~\ref{lemma-cJ-smooth-proper} combined with \cite[Lemmas 4.15 and 4.13]{NCHPD}. 
Then there are semiorthogonal decompositions 
\begin{align}
\label{Ccoh-sod} \cC^{\coh} & = 
\llangle (\cK_W(\cA^1, \cA^2))^{\coh}, (\cJ_N)^{\coh}(H), \dots, (\cJ_{m-1})^{\coh}((m-N)H) \rrangle ,    \\ 
\label{Dcoh-sod} \cD^{\coh} & = 
\llangle (\cJd_{1-n})^{\coh}((N-n)H'), \dots, (\cJd_{-N})^{\coh}(-H'), (\cK'_{W^{\svee}}((\cA^1)^{\hpd}, (\cA^2)^{\hpd}))^\coh \rrangle, 
\end{align} 
and an $S$-linear equivalence 
\begin{equation*}
(\cK_W(\cA^1, \cA^2))^\coh \simeq  (\cK'_{W^{\svee}}((\cA^1)^{\hpd}, (\cA^2)^{\hpd}))^{\coh} .  
\end{equation*}

Note that if $\cA^k = \Perf(X_k)$ for a $\bP(V_k)$-scheme $X_k$, 
then there is an equivalence 
\begin{equation*}
\cC \simeq \Perf{ \left( (X_1)_{\bP(W)} \times_{\bP(W)} (X_2)_{\bP(W)} \right) } . 
\end{equation*}
Hence if our base $S$ is noetherian over a field of characteristic $0$ and $X_k \to \bP(V_k)$ is a proper  
morphism, then there is an equivalence  
\begin{equation*}
\cC^\coh \simeq \Db{\left( (X_1)_{\bP(W)} \times_{\bP(W)} (X_2)_{\bP(W)} \right)}. 
\end{equation*} 
Note that in both formulas above the fiber product is derived 
(and hence agrees with the usual fiber product of schemes when $\Tor$-independence holds).
\end{remark}

\subsection{An iterated nonlinear HPD theorem} 
\label{subsection-iterated-nonlinear-HPD-theorem}   

Theorem~\ref{theorem-nonlinear-HPD} describes the tensor product of two Lefschetz categories 
over a projective bundle in terms of their HPD categories. 
This generalizes to a description of the tensor product of an arbitrary 
number of Lefschetz categories over a projective bundle. 
The key point is to consider iterated categorical joins of a collection of Lefschetz categories. 

\begin{definition}
For $k=1,2, \dots, \ell$, let $\cA^k$ be a Lefschetz category over $\bP(V_k)$. 
The \emph{categorical join} of $\cA^1, \dots, \cA^{\ell}$ is the Lefschetz 
category over $\bP(V_1 \oplus \cdots \oplus V_{\ell})$ defined inductively by the formula 
\begin{equation*}
\cJ(\cA^1, \dots, \cA^{\ell}) = \cJ(\cJ(\cA^1, \dots, \cA^{\ell-1}), \cA^\ell). 
\end{equation*}
\end{definition}

\begin{remark}
As the notation suggests,
the operation of taking a categorical join is associative in the sense that there is a Lefschetz equivalence
\begin{equation*}
\cJ(\cJ(\cA^1, \cA^{2}), \cA^3) \simeq \cJ(\cA^1, \cJ(\cA^2,\cA^3)).
\end{equation*}
To see this one can define the universal ``triple resolved join'' $\tJ(\bP(V_1),\bP(V_2),\bP(V_3))$ as
\begin{equation*}
\tJ(\bP(V_1),\bP(V_2),\bP(V_3)) = \bP_{\bP(V_1)\times\bP(V_2)\times\bP(V_3)}(\cO(-H_1) \oplus \cO(-H_2) \oplus \cO(-H_3)),
\end{equation*}
where $H_k$ are the hyperplane classes of $\bP(V_k)$.
There are three sections
\begin{equation*}
\sigma_k \colon \bP(V_1)\times\bP(V_2)\times\bP(V_3) \hookrightarrow \tJ(\bP(V_1),\bP(V_2),\bP(V_3)), \qquad k = 1, 2, 3,
\end{equation*}
corresponding to the embeddings~$\cO(-H_k) \to \cO(-H_1) \oplus \cO(-H_2) \oplus \cO(-H_3)$.
Let 
\begin{equation*}
\tJ(\cA^1,\cA^2,\cA^3) = (\cA^1 \otimes \cA^2 \otimes \cA^3) \otimes_{\Perf(\bP(V_1)\times\bP(V_2)\times\bP(V_3))} 
\Perf(\tJ(\bP(V_1),\bP(V_2),\bP(V_3))),
\end{equation*}
and define the ``triple categorical join'' as
\begin{equation*}
\cJ(\cA^1,\cA^2,\cA^3) = \left\{ 
C \in \tJ(\cA^1,\cA^2,\cA^3) 
\ \left|\ 
\begin{aligned}
\sigma_1^*(C) &\in \cA^1 \sotimes \cA^2_0 \sotimes \cA^3_0 \subset \cA^1 \otimes \cA^2 \otimes \cA^3 , \\
\sigma_2^*(C) &\in \cA^1_0 \sotimes \cA^2 \sotimes \cA^3_0 \subset \cA^1 \otimes \cA^2 \otimes \cA^3 , \\
\sigma_3^*(C) &\in \cA^1_0 \sotimes \cA^2_0 \sotimes \cA^3 \subset \cA^1 \otimes \cA^2 \otimes \cA^3
\end{aligned}
\right.\right\}.
\end{equation*}
It is not a priori clear whether $\cJ(\cA^1,\cA^2,\cA^3)$ is an admissible subcategory in~$\tJ(\cA^1,\cA^2,\cA^3)$.
Note, however, that $\tJ(\cA^1,\cA^2,\cA^3)$ has a $\bP(V_1 \oplus V_2 \oplus V_3)$-linear structure induced 
by the natural morphism~$\tJ(\bP(V_1),\bP(V_2),\bP(V_3)) \to \bP(V_1 \oplus V_2 \oplus V_3)$, and 
that $\cJ(\cA^1,\cA^2,\cA^3) \subset \tJ(\cA^1,\cA^2,\cA^3)$ is a $\bP(V_1 \oplus V_2 \oplus V_3)$-linear 
subcategory. Moreover, if 
\begin{equation*} 
\rho \colon \tJ(\bP(V_1),\bP(V_2),\bP(V_3)) \to \bP(V_1) \times \bP(V_2) \times \bP(V_2)
\end{equation*} 
is the projection, then the induced functor $\rho^* \colon \cA^1 \otimes \cA^2 \otimes \cA^3 \to \tJ(\cA^1, \cA^2, \cA^3)$ gives a fully faithful embedding $\cA^1_0 \otimes \cA^2_0 \otimes \cA^3_0 \hookrightarrow \cJ(\cA^1,\cA^2,\cA^3)$, whose image we denote by $\cJ(\cA^1,\cA^2,\cA^3)_0$. 

There is a morphism
\begin{equation*}
\tJ(\tJ(\bP(V_1),\bP(V_2)),\bP(V_3)) \xrightarrow{\ \varpi_{3}\ }
\tJ(\bP(V_1),\bP(V_2),\bP(V_3)),
\end{equation*}
identifying the iterated resolved join $\tJ(\tJ(\bP(V_1),\bP(V_2)),\bP(V_3))$ with the blowup of the image of the section~$\sigma_3$. 
The key observation is that $\varpi_3^*$ induces a $\bP(V_1 \oplus V_2 \oplus V_3)$-linear equivalence 
\begin{equation*}
\cJ(\cA^1,\cA^2,\cA^3) \simeq \cJ(\cJ(\cA^1,\cA^2),\cA^3) , 
\end{equation*}
which takes the subcategory $\cJ(\cA^1, \cA^2, \cA^3)_0$ to the Lefschetz center $\cJ(\cJ(\cA^1,\cA^2),\cA^3)_0$. 
Indeed, it is straightforward to check that $\varpi_3^*$ gives a $\bP(V_1 \oplus V_2 \oplus V_3)$-linear 
fully faithful functor~$\cJ(\cA^1,\cA^2,\cA^3) \hookrightarrow \cJ(\cJ(\cA^1,\cA^2),\cA^3)$ 
that takes $\cJ(\cA^1, \cA^2, \cA^3)_0$ to $\cJ(\cJ(\cA^1,\cA^2),\cA^3)_0$, 
and then essential surjectivity follows from $\bP(V_1 \oplus V_2 \oplus V_3)$-linearity 
and the fact that the image contains $\cJ(\cJ(\cA^1,\cA^2),\cA^3)_0$. 
Therefore, $\cJ(\cA^1,\cA^2,\cA^3)$ has the structure of a Lefschetz category over $\bP(V_1 \oplus V_2 \oplus V_3)$ 
with center $\cJ(\cA^1, \cA^2, \cA^3)_0$, 
and there is a natural Lefschetz equivalence~$\cJ(\cA^1,\cA^2,\cA^3) \simeq \cJ(\cJ(\cA^1,\cA^2),\cA^3)$. 
It also follows that $\cJ(\cA^1, \cA^2, \cA^3)$ is admissible in~$\tJ(\cA^1, \cA^2, \cA^3)$, but we do not need this.

Replacing the image of $\sigma_3$ by the image of $\sigma_1$ and $\varpi_3$ by the analogous morphism~$\varpi_1$,
we also obtain a Lefschetz equivalence $\cJ(\cA^1,\cA^2,\cA^3) \simeq \cJ(\cA^1,\cJ(\cA^2,\cA^3))$.
A combination of these two equivalences proves associativity of the categorical join. 

Finally, we note that the associativity of categorical joins is also addressed in \cite{jiang-leung-joins}, 
and that our definition of the triple categorical join answers the question from \cite[Remark B.1]{jiang-leung-joins}. 
\end{remark}

As an immediate consequence of Theorem~\ref{theorem-joins-HPD}, we obtain the following. 
\begin{theorem}
\label{theorem-iterated-joins-HPD}
For $k=1, 2, \dots, \ell$, 
let $\cA^k$ be a right strong, moderate Lefschetz category over~$\bP(V_k)$. 
Then there is an equivalence 
\begin{equation*}
\cJ(\cA^1, \dots, \cA^{\ell})^{\hpd} \simeq \cJ((\cA^1)^{\hpd}, \dots, (\cA^\ell)^{\hpd}) 
\end{equation*}
of Lefschetz categories over $\bP(\vV_1 \oplus \cdots \oplus \vV_\ell)$. 
\end{theorem} 

For $k=1,2, \dots, \ell$, let $\cA^k$ be a Lefschetz category over $\bP(V_k)$. 
Let $\cJ_i$, $i \in \bZ$, be the Lefschetz components of the categorical join 
$\cJ(\cA^1, \dots, \cA^\ell)$. 
If $W \subset V_1 \oplus \cdots \oplus V_\ell$ is a subbundle of corank~$s$,
then by Lemma~\ref{lemma-linear-section-lc} the functor $\cJ_i \hookrightarrow \cJ(\cA^1,\dots,\cA^\ell) \to \cJ(\cA^1,\dots,\cA^\ell)_{\bP(W)}$ is fully faithful for $|i| \ge s$, 
so we can consider $\cJ_i$ as a subcategory of $\cJ(\cA^1,\dots,\cA^\ell)_{\bP(W)}$.
If moreover the composition $W \to V_1 \oplus \cdots \oplus V_{\ell} \to V_k$ is 
an inclusion of vector bundles for each~$k$, then Proposition~\ref{proposition-cJT} gives an equivalence
\begin{equation}
\label{eq:cj-pl-a1-a2-iterated}
\cJ(\cA^1, \dots, \cA^\ell)_{\bP(W)} \simeq \cA^1_{\bP(W)} \otimes_{\Perf(\bP(W))} \dots \otimes_{\Perf(\bP(W))} \cA^\ell_{\bP(W)} , 
\end{equation} 
so in this case we can consider $\cJ_i$ as a subcategory of $\cA^1_{\bP(W)} \otimes_{\Perf({\bP(W)})} \cdots \otimes_{\Perf({\bP(W)})} \cA^{\ell}_{\bP(W)}$ as soon as~$|i| \ge s$.
Finally, this subcategory can be described as in Remark~\ref{remark:cji-in-a1-a2} without appealing to 
categorical joins, as the image of an explicit subcategory of 
$\cA^1 \otimes \cdots \otimes \cA^\ell$ (defined along the lines of Lemma~\ref{lemma-Ji-alternate}) 
under the functor given by base change along the induced morphism~$\bP(W) \to \bP(V_1) \times \dots \times \bP(V_\ell)$.
Combining Lemma~\ref{lemma-linear-section-lc} with the equivalence~\eqref{eq:cj-pl-a1-a2-iterated}, 
we obtain the following iterated version of Corollary~\ref{corollary-fiber-product-lc}. 

\begin{proposition}
\label{proposition-iterated-fiber-product-lc}
For $k=1, 2, \dots, \ell$, let $\cA^k$ be a Lefschetz category over $\bP(V_k)$. 
For $i \in \bZ$ let $\cJ_i$ be the Lefschetz components of the categorical join $\cJ(\cA^1, \dots, \cA^\ell)$. 
Let $W$ be a vector bundle on $S$ equipped with inclusions of vector 
bundles $W \to V_k$ for all $k$. 
Set 
\begin{equation*}
m = \sum_k \length(\cA^k) 
\qquad\text{and}\qquad
s =  \sum_k \rank(V_k) - \rank(W) . 
\end{equation*}
Denote
by $H$ the relative hyperplane class on $\bP(W)$.
Then there are semiorthogonal decompositions 
\begin{align*}
\cA^1_{\bP(W)}{ \otimes_{\Perf(\bP(W))}} \cdots {\otimes_{\Perf(\bP(W))}} \cA^\ell_{\bP(W)} & = 
\llangle \cK_W(\cA^1, \dots, \cA^\ell), \cJ_s(H), \dots, \cJ_{m-1}((m-s)H) \rrangle ,    \\
&  \hspace{-3.5pt} = \hspace{-3pt} 
\llangle 
\cJ_{1-m}((s-m)H), \dots, \cJ_{-s}(-H), \cK'_W(\cA^1, \dots, \cA^\ell)
\rrangle . 
\end{align*}
\end{proposition}

Now we state the iterated nonlinear HPD theorem, which 
describes $\cK_W(\cA^1, \dots, \cA^\ell)$ in terms of the 
HPD categories $(\cA^k)^{\hpd}$ and reduces to Theorem~\ref{theorem-nonlinear-HPD} 
when $\ell = 2$. 
When $\ell > 2$, the description is in terms 
of the categorical join of the categories $(\cA^k)^{\hpd}$, and cannot be 
expressed in terms of a tensor product of the $(\cA^k)^{\hpd}$
over a projective bundle. 

\begin{theorem}
\label{theorem-iterated-nonlinear-HPD} 
For $k=1, 2, \dots, \ell$, let $\cA^k$ be a 
right strong, moderate Lefschetz category over $\bP(V_k)$. 
For $i , j \in \bZ$ let $\cJ_i$ and $\cJd_j$ be the Lefschetz 
components of the categorical joins~$\cJ(\cA^1, \dots, \cA^\ell)$ and $\cJ((\cA^1)^{\hpd}, \dots, (\cA^\ell)^{\hpd})$. 
Let $W$ be a vector bundle on $S$ equipped with inclusions of vector bundles 
\begin{equation*}
\xi_{k} \colon W \to V_k, ~ k = 1, \dots, \ell, 
\end{equation*} 
and let  
\begin{equation*}
W^{\perp} = \set{ 
(\theta_1, \dots, \theta_\ell) \in V_1^{\svee} \oplus \cdots \oplus V_{\ell}^{\svee} ~ \st ~
\textstyle{\sum_k} \, \theta_k \circ \xi_k = 0 \in W^{\svee} } \subset \vV_1 \oplus \cdots \oplus \vV_\ell 
\end{equation*}
be the orthogonal to the induced inclusion $W \to V_1 \oplus \cdots \oplus V_\ell$. 
Let $H$ and $H'$ denote the relative hyperplane classes on 
$\bP(W)$ and $\bP(W^{\perp})$, and set 
\begin{equation*}
r = \rank(W), ~ s = \rank(W^{\perp}), ~ m = \sum_k \length(\cA^k), ~ n = \sum_{k} \length((\cA^k)^\hpd) . 
\end{equation*}
Then there are semiorthogonal decompositions 
\begin{align*}
\cA^1_{\bP(W)} \otimes_{\Perf(\bP(W))} \cdots \otimes_{\Perf(\bP(W))} \cA^\ell_{\bP(W)} = 
\llangle \cK_W(\cA^1, \dots, \cA^\ell), \cJ_{s}(H), \dots, \cJ_{m-1}((m-s)H) \rrangle ,    \\ 
\cJ((\cA^1)^{\hpd}, \dots, (\cA^{\ell})^{\hpd})_{\bP(W^{\perp})} = 
\llangle \cJd_{1-n}((r - n)H'), \dots, \cJd_{-r}(-H'), \cK'_{W^{\perp}}(\cJ((\cA^1)^{\hpd}, \dots, (\cA^\ell)^{\hpd})) \rrangle, 
\end{align*} 
and an $S$-linear equivalence 
\begin{equation*}
\cK_W(\cA^1, \dots, \cA^\ell) \simeq  \cK'_{W^{\perp}}(\cJ((\cA^1)^{\hpd}, \dots, (\cA^\ell)^{\hpd})) . 
\end{equation*}
\end{theorem}

\begin{proof}
The argument of Theorem~\ref{theorem-nonlinear-HPD} works, using Proposition~\ref{proposition-iterated-fiber-product-lc}, the equivalence~\eqref{eq:cj-pl-a1-a2-iterated}, and Theorem~\ref{theorem-iterated-joins-HPD}   
in place of the corresponding results for $\ell=2$.
\end{proof}

\begin{remark}
In Theorem~\ref{theorem-iterated-nonlinear-HPD} we do not require $\xi_k \colon W \to V_k$ 
to be an isomorphism, as in Theorem~\ref{theorem-nonlinear-HPD}. 
The reason is that this assumption does not lead to a simplification in the 
statement of the conclusion when $\ell > 2$. 
\end{remark}


\section{Applications}
\label{section:applications}

In this section, we discuss some applications of  
HPD for categorical joins (Theorem~\ref{theorem-joins-HPD}) and the nonlinear HPD theorem (Theorem~\ref{theorem-nonlinear-HPD}).  
For simplicity, we assume the base scheme $S$ is the spectrum of an algebraically closed field~$\bk$ of characteristic~$0$. 

\subsection{Intersections of two Grassmannians} 
\label{subsection:intersection-gr25}
Let $V_5$ be a $5$-dimensional vector space and let $\Gr(2,V_5)$ be the Grassmannian of 2-dimensional vector subspaces of $V_5$.
Note that we have~$\Gr(2,V_5) \subset \bP(\wedge^2 V_5) \cong \bP^9$ via the Pl\"{u}cker embedding.

\begin{theorem}
\label{theorem-Gr-intersection} 
Let $V_5$ be a $5$-dimensional vector space. 
Let $W$ be a vector space equipped with two isomorphisms 
\begin{equation*}
\xi_1 \colon W \xrightarrow{\sim} \wedge^2 V_5 , \quad 
\xi_2 \colon W \xrightarrow{\sim} \wedge^2 V_5,  
\end{equation*}
and let 
\begin{equation*}
\xi_1^{\svee} \colon  \wedge^2 V_5^{\svee} \xrightarrow{\sim} W^{\svee} , \quad 
\xi_2^{\svee} \colon \wedge^2 V_5^{\svee} \xrightarrow{\sim} W^{\svee},  
\end{equation*}
be the dual isomorphisms. 
Consider the derived fiber products
\begin{align*}
X & = \xi_1^{-1}(\Gr(2,V_5)) \times_{\bP(W)} \xi_{2}^{-1}(\Gr(2,V_5))  , \\ 
Y & = \xi_1^{\svee}(\Gr(2,V_5^{\svee})) \times_{\bP(W^{\svee})} \xi_{2}^{\svee}(\Gr(2,V_5^{\svee})) . 
\end{align*}
Then there are equivalences of categories 
\begin{equation*}
\Perf(X) \simeq \Perf(Y) \quad \text{and} \quad \Db(X) \simeq \Db(Y). 
\end{equation*}
\end{theorem}

\begin{proof}
By \cite[Section~6.1 and Theorem~1.2]{kuznetsov2006hyperplane}, the Grassmannian 
$\Gr(2,V_5)$ is homologically projectively self-dual. 
More precisely, 
let $\cU$ and $\cU'$ be the tautological rank~$2$ subbundles on~$\Gr(2,V_5)$ and~$\Gr(2,\vV_5)$. 
Then $\Perf(\Gr(2,V_5))$ and $\Perf(\Gr(2,\vV_5))$ have the structure of 
strong, moderate Lefschetz categories over $\bP(\wedge^2V_5)$ and $\bP(\wedge^2\vV_5)$ 
of length $5$, with Lefschetz components given by 
\begin{equation*}
\cA_i  = \langle \cO , \cUv \rangle \quad \text{and} \quad 
\cA'_i  = \langle \cU', \cO \rangle   
\end{equation*} 
for $|i| \leq 4$, and there is an equivalence 
\begin{equation*}
\Perf(\Gr(2,V_5))^{\hpd} \simeq \Perf(\Gr(2,\vV_5))
\end{equation*}
of Lefschetz categories over $\bP(\wedge^2\vV_5)$. 
Now the first equivalence of the theorem follows by applying Theorem~\ref{theorem-nonlinear-HPD}, 
and the second follows by Remark~\ref{remark-Perf-to-Db}. 
\end{proof} 

\begin{remark}
When smooth of the expected dimension $3$, the varieties $X$ and $Y$ 
in Theorem~\ref{theorem-Gr-intersection} are Calabi--Yau threefolds (and the fiber products are underived). 
For a generic choice of the isomorphisms $\xi_1$ and $\xi_2$, this pair of varieties 
was recently shown to give the first example of deformation equivalent, 
derived equivalent, but non-birational Calabi--Yau threefolds, and as a consequence 
the first counterexample to the birational Torelli problem for Calabi--Yau threefolds
\cite{GPK3, jj-torelli}.  
\end{remark}

There is a similar construction with $\Gr(2,V_5)$ replaced by an orthogonal Grassmannian.
For a vector space $V$ of even dimension $2n$ with a  
nondegenerate quadratic form $q \in \Sym^2\vV$, the 
Grassmannian of $n$-dimensional isotropic subspaces of $V$ 
has two connected components, which are abstractly isomorphic. 
We denote by $\OGrp(n, V)$ one of these components and by~$\OGrm(n, V)$ the other. 
The Pl\"{u}cker embedding $\OGrp(n,V) \to \bP(\wedge^nV)$ is 
given by the square of the generator of $\Pic(\OGrp(n,V))$; the 
generator itself gives an embedding 
\begin{equation*}
\OGrp(n,V) \to \bP(\sS_{2^{n-1}}), 
\end{equation*}
where $\sS_{2^{n-1}}$ 
is a $2^{n-1}$-dimensional half-spinor representation of $\Spin(V)$. 

In the case of a $10$-dimensional vector space $V_{10}$, the orthogonal 
Grassmannian in its spinor embedding $\OGrp(5, V_{10}) \subset \bP(\sSs)$ 
shares a very special property with the Grassmannian~$\Gr(2, V_5) \subset \bP(\wedge^2V_5)$: both are projectively self-dual, 
and even homologically projectively self-dual.
To be more precise, the classical projective dual of $\OGrp(5, V_{10}) \subset \bP(\sSs)$ 
is given by the spinor embedding $\OGrm(5, V_{10}) \subset \bP(\sSs^{\svee})$. 
Like the case of $\Gr(2, V_5)$, this persists at the level of HPD 
by \cite[Section~6.2 and Theorem~1.2]{kuznetsov2006hyperplane}. 
The same argument as in Theorem~\ref{theorem-Gr-intersection} then proves the following 
spin analogue.

\begin{theorem}
\label{theorem-OGr-intersection} 
Let $V_{10}$ be a $10$-dimensional vector space, and let $\sSs$ be a 
half-spinor representation of $\Spin(V_{10})$. 
Let $W$ be a vector space equipped with two isomorphisms 
\begin{equation*}
\xi_1 \colon W \xrightarrow{\sim} \sSs , \quad 
\xi_2 \colon W \xrightarrow{\sim} \sSs ,  
\end{equation*}
and let  
\begin{equation*}
\xi_1^{\svee} \colon  \sSs^{\svee} \xrightarrow{\sim} W^{\svee} , \quad 
\xi_2^{\svee} \colon \sSs^{\svee} \xrightarrow{\sim} W^{\svee},  
\end{equation*}
be the dual isomorphisms. 
Consider the derived fiber products
\begin{align*}
X & = \xi_1^{-1}(\OGrp(5,V_{10})) \times_{\bP(W)} \xi_{2}^{-1}(\OGrp(5,V_{10}))  , \\ 
Y & = \xi_1^{\svee}(\OGrm(5,V_{10})) \times_{\bP(W^{\svee})} \xi_{2}^{\svee}(\OGrm(5,V_{10})) . 
\end{align*}
Then there are equivalences of categories 
\begin{equation*}
\Perf(X) \simeq \Perf(Y) \quad \text{and} \quad \Db(X) \simeq \Db(Y). 
\end{equation*}
\end{theorem}

\begin{remark}
When smooth of the expected dimension $5$, the varieties $X$ and $Y$ 
in Theorem~\ref{theorem-OGr-intersection} are Calabi--Yau fivefolds 
(and the fiber products are underived), which were 
recently studied in \cite{double-spinor}. There, following \cite{GPK3, jj-torelli} 
it is shown that for generic $\xi_1$ and $\xi_2$, these varieties are non-birational.   
\end{remark}

\subsection{Enriques surfaces}
\label{subsection-Enriques} 
The goal of this subsection is to prove that for a general Enriques surface $\Sigma$, 
there is a stacky projective plane $\ccP$ (with stack structure 
along the union of two cubic curves), such that the subcategory $\langle \cO_\Sigma \rangle^{\perp} \subset \Perf(\Sigma)$ 
is equivalent to the orthogonal of an exceptional object in the twisted derived category of $\ccP$.  
The precise statement is Theorem~\ref{theorem-enriques} below. 
As we will see, the result falls out naturally by considering the categorical join of two Veronese surfaces. 

Let $W$ be a $3$-dimensional vector space, and let $V = \Sym^2W$. 
The double Veronese embedding $\bP(W) \to \bP(V)$, given by the linear system $|\cO_{\bP(W)}(2)|$, 
endows $\Perf(\bP(W))$ with a~$\bP(V)$-linear structure. 
Below we will consider $\Perf(\bP(W))$ as a Lefschetz category over~$\bP(V)$ of length $2$, 
with right Lefschetz components given by 
\begin{equation*}
\cA_i = \begin{cases}
\langle \cO_{\bP(W)}, \cO_{\bP(W)}(1) \rangle & \text{for $i = 0$} , \\ 
\langle \cO_{\bP(W)} \rangle & \text{for $i = 1$} .
\end{cases}
\end{equation*}
We call this the \emph{double Veronese Lefschetz structure} on $\bP(W)$, 
to distinguish it from the standard Lefschetz structure on a projective space 
from Example~\ref{example-projective-bundle-lc}. 

We need the description of the HPD of $\Perf(\bP(W))$ from \cite{kuznetsov08quadrics}. 
The universal family of conics in $\bP(W)$ is a conic fibration over $\bP(\vV)$. 
Associated to this fibration are the sheaves~$\Cl_0$ and $\Cl_1$ of even and odd parts of the corresponding  
Clifford algebra on $\bP(\vV)$, which as sheaves of $\cO_{\bP(\vV)}$-modules are given by 
\begin{alignat}{2}
\Cl_0 & = \hspace{18pt} \cO_{\bP(\vV)} && \oplus \left(\wedge^2W \otimes \cO_{\bP(\vV)}(-1) \right) , \\ 
\Cl_1 & = \left(W \otimes \cO_{\bP(\vV)} \right) && \oplus \left(\wedge^3W \otimes \cO_{\bP(\vV)}(-1) \right)  . 
\end{alignat}
Note that $\Cl_0$ is a sheaf of $\cO_{\bP(\vV)}$-algebras via Clifford multiplication,  
and $\Cl_1$ is a (locally projective) module over $\Cl_0$. 
Before continuing, we need a brief digression on the noncommutative scheme 
associated to a sheaf of algebras. 

\begin{notation}
\label{notation-sheaves-algebras}
Suppose $X$ is a scheme (or stack) equipped with a sheaf $\cR$ 
of $\cO_{X}$-algebras, such that 
$\cR$ is finite locally free over $\cO_{X}$. 
We denote by $\QCoh(X, \cR)$ the unbounded derived category of quasi-coherent sheaves of $\cR$-modules, and 
by $\Perf(X, \cR) \subset \QCoh(X, \cR)$ the full subcategory of objects which are perfect as complexes of $\cR$-modules. 
\end{notation} 

\begin{remark}
\label{remark-tensor-sheaves-algebras} 
In the above situation, $\Perf(X, \cR)$ naturally has the structure of an $X$-linear category. 
Moreover, there is a geometric description of tensor products of categories of this form. 
Namely, let $(X_1, \cR_1)$ and $(X_2, \cR_2)$ be two pairs as in Notation~\ref{notation-sheaves-algebras}, 
and assume $X_1$ and~$X_2$ are defined over a common scheme $T$, so that 
both $\Perf(X_1, \cR_1)$ and $\Perf(X_2, \cR_2)$ can be considered as $T$-linear categories. 
Assume $X_1$ and $X_2$ are perfect stacks in the sense of~\cite{bzfn}. 
Then there is an equivalence 
\begin{equation*}
\Perf(X_1, \cR_1) \otimes_{\Perf(T)} \Perf(X_2, \cR_2) \simeq \Perf{\left( X_1 \times_T X_2, \cR_1 \boxtimes \cR_2 \right)} , 
\end{equation*}
where the fiber product on the right side is derived.
This is a mild generalization of the case where~\mbox{$\cR_1 = \cO_{X_1}$} and~\mbox{$\cR_2 = \cO_{X_2}$} 
proved in \cite{bzfn}, and follows by the same argument.
See also \cite{kuznetsov2006hyperplane} for the case where $\cR_1$ and $\cR_2$ are 
Azumaya algebras. 
\end{remark}

We consider the $\bP(\vV)$-linear category $\Perf(\bP(\vV), \Cl_0)$. 
From $\Cl_0$ and $\Cl_1$ we obtain exceptional objects $\Cl_i \in \Perf(\bP(\vV), \Cl_0)$ for 
all $i \in \bZ$ by the prescription  
\begin{equation*}
\Cl_{i+2} = \Cl_i \otimes \cO_{\bP(\vV)}(1). 
\end{equation*} 

\begin{theorem}[{\cite[Theorem 5.4]{kuznetsov08quadrics}}]
\label{theorem-HPD-veronese}
The category $\Perf(\bP(\vV), \Cl_0)$ has a Lefschetz structure 
over $\bP(\vV)$ of length $5$, 
with left Lefschetz components given by 
\begin{equation*}
\cB_i = \begin{cases}
\langle \Cl_0 \rangle & \text{for $i = -4$} , \\
\langle \Cl_{-1}, \Cl_0 \rangle & \text{for $-3 \leq i \leq 0$} . 
\end{cases}
\end{equation*} 
Moreover, there is an equivalence 
\begin{equation*}
\Perf(\bP(W))^{\hpd} \simeq \Perf(\bP(\vV), \Cl_0) 
\end{equation*} 
of Lefschetz categories over $\bP(\vV)$, where $\bP(W)$ is considered 
with its double Veronese Lefschetz structure. 
\end{theorem}

\begin{remark}
In \cite{kuznetsov08quadrics}, an HPD theorem is proved more generally for the double Veronese 
embedding of $\bP(W)$ where $W$ is of arbitrary dimension. 
\end{remark}

Now we consider the categorical join of two copies of the above data. 
Namely, for $k=1,2$, let $W_k$ be a $3$-dimensional vector space,  
let $V_k = \Sym^2 W_k$, and let $\Cl^k_i$, $i \in \bZ$, denote the Clifford sheaves on $\bP(\vV_k)$ from above.  
We set 
\begin{equation*}
\cA^k = \Perf(\bP(W_k))
\end{equation*} 
with the double Veronese Lefschetz structure over $\bP(V_k)$, and 
\begin{equation*}
\cB^k = \Perf(\bP(\vV_k), \Cl^k_0) 
\end{equation*}
with the Lefschetz structure over $\bP(\vV_k)$ from above. 

By Theorem~\ref{theorem-join-lef-cat}, 
the categorical join $\cJ(\cA^1, \cA^2)$ is a Lefschetz category over $\bP(V_1 \oplus V_2)$ of length $4$, 
and $\cJ(\cB^1, \cB^2)$ is a Lefschetz category over $\bP(\vV_1 \oplus \vV_2)$ 
of length $10$. 
Moreover, it follows from Lemma~\ref{lemma-Ji-alternate} that the right Lefschetz components 
of $\cJ(\cA^1, \cA^2)$ are given by 
\begin{equation*}
\cJ(\cA^1, \cA^2)_{i}  =  
\begin{cases}
\llangle \cO, \cO(1,0), \cO(0,1), \cO(1,1) \rrangle  & \text{if $i = 0, 1$,} \\ 
\llangle \cO, \cO(1,0), \cO(0,1) \rrangle  & \text{if $i = 2$,} \\ 
\llangle \cO \rrangle  & \text{if $i = 3$,} 
\end{cases} 
\end{equation*} 
and the left Lefschetz components of $\cJ(\cB^1, \cB^2)$ are given by 
\begin{equation*}
\cJ(\cB^1, \cB^2)_{i} = 
\begin{cases}
\hphantom{\Cl_{-1,-1}, \Cl_{-1,0}, \Cl_{0,-1},} \llangle \Cl_{0,0}  \rrangle  & \text{if $i = -9$,} \\
\hphantom{\Cl_{-1,-1},} \llangle \Cl_{-1,0}, \Cl_{0,-1}, \Cl_{0,0} \rrangle  & \text{if $i = -8$,} \\ 
\llangle \Cl_{-1,-1}, \Cl_{-1,0}, \Cl_{0,-1}, \Cl_{0,0} \rrangle  & \text{if $-7 \leq i \leq 0$,} 
\end{cases}
\end{equation*}
where we write $\Cl_{i_1,i_2} = \Cl^1_{i_1} \boxtimes \Cl^2_{i_2}$,
and when we write $\cO(i_1,i_2)$ or $\Cl_{i_1,i_2}$ in the right-hand sides of the above equalities we mean
the pullbacks of the corresponding 
objects of~$\Perf(\bP(W_1) \times \bP(W_2))$ and~$\Perf(\bP(\vV_1) \times \bP(\vV_2), \Cl^1_0 \boxtimes \Cl^2_0)$. 
By Theorem~\ref{theorem-joins-HPD} combined with Theorem~\ref{theorem-HPD-veronese}, 
there is an equivalence 
\begin{equation*}
\cJ(\cA^1, \cA^2)^{\hpd} \simeq \cJ(\cB^1, \cB^2)
\end{equation*}
of Lefschetz categories over $\bP(\vV_1 \oplus \vV_2)$.

Now let $L \subset V_1 \oplus V_2$ be a vector subspace of codimension $3$. 
Then combining the above with Theorem~\ref{theorem-HPD}\eqref{HPD-linear} (and twisting appropriately), 
we obtain the following.

\begin{corollary} 
\label{corollary-join-veronese-cliff}
There are semiorthogonal decompositions 
\begin{align*}
\cJ(\cA^1, \cA^2)_{\bP(L)} & = \langle \cK_L, \cO \rangle ,  \\ 
\cJ(\cB^1, \cB^2)_{\bP(L^{\perp})} & = \langle \Cl^1_0 \boxtimes \Cl^2_0 , \cK'_L \rangle, 
\end{align*}
and an equivalence $\cK_L \simeq \cK'_L$. 
\end{corollary}

Our goal below is to describe geometrically the linear section categories appearing in 
Corollary~\ref{corollary-join-veronese-cliff} for generic $L$. 

First we consider the category $\cJ(\cA^1, \cA^2)_{\bP(L)} = \cJ(\bP(W_1), \bP(W_2))_{\bP(L)}$. 
The resolved join $\tJ(\bP(W_1), \bP(W_2))$ of $\bP(W_1)$ and $\bP(W_2)$ with 
respect to their double Veronese embeddings is a $\bP(V_1 \oplus V_2)$-scheme. 
We define $\Sigma_L$ as its base change 
along $\bP(L) \subset \bP(V_1 \oplus V_2)$, i.e. 
\begin{equation*}
\Sigma_L = \tJ(\bP(W_1), \bP(W_2))_{\bP(L)} . 
\end{equation*} 

\begin{lemma}
\label{lemma-JL-SigmaL} 
Assume $\bP(L)$ does not intersect 
$\bP(W_1) \sqcup \bP(W_2)$ in $\bP(V_1 \oplus V_2)$. 
Then there is an equivalence 
\begin{equation*}
\cJ(\cA^1, \cA^2)_{\bP(L)} \simeq \Perf(\Sigma_L). 
\end{equation*}
\end{lemma}

\begin{proof}
Follows from Proposition~\ref{proposition-cJT}. 
\end{proof} 

The assumption of Lemma~\ref{lemma-JL-SigmaL} holds generically. 
In this case, $\Sigma_L$ is a familiar variety: 
\begin{lemma}[{\cite[Lemma~3]{kuznetsov2018embedding}}]
For generic $L \subset V_1 \oplus V_2$ of codimension~$3$, the scheme $\Sigma_L$ is an Enriques surface. 
Moreover, a general Enriques surface is obtained in this way. 
\end{lemma}

Now we turn to $\cJ(\cB^1, \cB^2)_{\bP(L^{\perp})}$. 
Note that $\dim(L^\perp) = 3$, so $\bP(L^\perp) \cong \bP^2$.

\begin{lemma}
\label{lemma-JL-Cl}
For $k=1,2$, assume the composition $L^{\perp} \to \vV_1 \oplus \vV_2 \to \vV_k$ is an inclusion. 
Let $\Cl_0^k\vert_{\bP(L^{\perp})}$ denote the pullback of $\Cl_0^k$ along the induced 
embedding $\bP(L^\perp) \to \bP(\vV_k)$. 
Then there are equivalences 
\begin{align}
\label{clifford-join-Lperp-1}
\cJ(\cB^1, \cB^2)_{\bP(L^{\perp})} & \simeq 
\Perf {\left(\bP(L^{\perp}), \Cl_0^1\vert_{\bP(L^{\perp})} \right)} \otimes_{\Perf(\bP(L^{\perp}))}  \Perf {\left(\bP(L^{\perp}), \Cl_0^2\vert_{\bP(L^{\perp})} \right)} \\ 
\label{clifford-join-Lperp-2}
& \simeq 
\Perf{ \left(\bP(L^{\perp}), (\Cl_0^1 \boxtimes \Cl_0^2)\vert_{\bP(L^{\perp})} \right)} . 
\end{align}
\end{lemma}

\begin{proof}
The first equivalence follows from Proposition~\ref{proposition-cJT} combined with Remark~\ref{remark-tensor-sheaves-algebras}, 
and the second follows by applying Remark~\ref{remark-tensor-sheaves-algebras} again. 
\end{proof} 

The assumption of Lemma~\ref{lemma-JL-Cl} is equivalent to $\bP(L^{\perp})$ not meeting  
$\bP(\vV_1)$ or $\bP(\vV_2)$ in~$\bP(\vV_1 \oplus \vV_2)$, and holds generically for dimension reasons. 
In this case, we will give a more geometric description of $\cJ(\cB^1, \cB^2)_{\bP(L^{\perp})}$ 
by rewriting the factors in the tensor product in \eqref{clifford-join-Lperp-1}. 
Being the base of the universal family of conics $\ccX_k \to \bP(\vV_k)$ in $\bP(W_k)$, 
the space~$\bP(\vV_k)$ has a stratification 
\begin{equation*}
\bP(\vW_k) \subset D_k \subset \bP(\vV_k), 
\end{equation*}
where $D_k$ is the discriminant locus parameterizing degenerate conics 
and $\bP(\vW_k) \subset \bP(\vV_k)$ is the double Veronese embedding, which parameterizes non-reduced conics (double lines). 
Note that $D_k \subset \bP(\vV_k)$ is a cubic hypersurface, with singular locus $\bP(\vW_k)$. 

Under the assumption of Lemma~\ref{lemma-JL-Cl}, for $k=1,2$, we have an embedding 
\begin{equation*}
\xi_{k} \colon \bP(L^{\perp}) \to \bP(\vV_k). 
\end{equation*}
The stratification of $\bP(L^{\perp})$ associated to the 
pullback family of conics $(\ccX_k)_{\bP(L^\perp)} \to \bP(L^\perp)$ is the preimage  
of the stratification of $\bP(\vV_k)$, i.e. 
\begin{equation*}
\xi_k^{-1}(\bP(\vW_k)) \subset \xi_k^{-1}(D_k) \subset  \bP(L^\perp). 
\end{equation*} 
We write $C_k = \xi_k^{-1}(D_k)$ for the discriminant locus. 
Note that for $L \subset V_1 \oplus V_2$ generic, the locus $\xi_k^{-1}(\bP(\vW_k))$ is empty,  
$C_k$ is a smooth cubic curve in the projective plane $\bP(L^\perp)$, and the curves $C_1$ and $C_2$ intersect transversally.
We define 
\begin{equation*}
\ccP_k = \bP(L^{\perp})(\sqrt{C_k})
\end{equation*}
as the square root stack (see \cite[\S2.2]{root-stack-cadman} or \cite[Appendix B]{root-stack-vistoli}) 
of the divisor $C_k \subset \bP(L^{\perp})$.  
Note that $\ccP_k$ is a Deligne--Mumford stack with coarse moduli space 
\begin{equation*}
\rho_k \colon \ccP_k \to \bP(L^{\perp}), 
\end{equation*} 
where $\rho_k$ is an isomorphism over $\bP(L^{\perp}) \setminus C_k$ and a $\bZ/2$-gerbe over $C_k$. 

\begin{lemma} 
\label{lemma-Clk-Zk} 
For $k=1,2$, assume the composition $L^{\perp} \to \vV_1 \oplus \vV_2 \to \vV_k$ is an inclusion 
and $C_k \ne \bP(L^{\perp})$.
Then there is a finite locally free sheaf of algebras $\cR_k$ on $\ccP_k$ such that $\rho_{k*} \cR_k \simeq  \Cl_0^k\vert_{\bP(L^{\perp})}$ 
and the induced functor 
\begin{equation*}
\rho_{k*} \colon \Perf(\ccP_k ,\cR_k) \xrightarrow{\ \sim \ } \Perf {\left(\bP(L^{\perp}), \Cl_0^k\vert_{\bP(L^{\perp})} \right)}
\end{equation*} 
is an equivalence. Moreover, $\cR_k$ is Azumaya
over the complement of 
$\xi_k^{-1}(\bP(\vW_k))$ in $\bP(L^{\perp})$.
\end{lemma}

\begin{proof}
Follows from \cite[\S3.6]{kuznetsov08quadrics}, cf. \cite[Proposition 1.20]{ci-quadrics}. 
\end{proof}

In the situation of the lemma, we define 
\begin{equation*}
\ccP = \ccP_1 \times_{\bP(L^{\perp})} \ccP_2. 
\end{equation*} 
This space carries a finite locally free sheaf of algebras given by 
\begin{equation*}
\cR =  \cR_1 \boxtimes \cR_2. 
\end{equation*} 

\begin{lemma}
\label{lemma-JL-Z1Z2}
Under the assumption of Lemma~\textup{\ref{lemma-Clk-Zk}}, 
there is an equivalence 
\begin{equation*}
\cJ(\cB^1, \cB^2)_{\bP(L^{\perp})} \simeq \Perf(\ccP, \cR), 
\end{equation*} 
where $\cR$ is Azumaya over 
the complement of $\xi_1^{-1}(\bP(\vW_1)) \cup \xi_2^{-1}(\bP(\vW_2))$ in~$\bP(L^{\perp})$. 
\end{lemma}

\begin{proof}
Follows from Lemma~\ref{lemma-JL-Cl}, Lemma~\ref{lemma-Clk-Zk}, and 
Remark~\ref{remark-tensor-sheaves-algebras}. 
\end{proof}

\begin{remark}
The space $\ccP$ is a stacky projective plane.
More precisely, consider the stratification 
\begin{equation*}
C_1 \cap C_2 \subset C_1 \cup C_2 \subset \bP(L^\perp). 
\end{equation*} 
Then the canonical morphism $\rho \colon \ccP \to \bP(L^{\perp})$ 
can be described over the open strata as follows: 
\begin{itemize}
\item $\rho$ is an isomorphism over $\bP(L^{\perp}) \setminus (C_1 \cup C_2)$. 
\item $\rho$ is a $\bZ/2$-gerbe over $(C_1 \cup C_2) \setminus (C_1 \cap C_2)$. 
\item $\rho$ is a $\bZ/2 \times \bZ/2$-gerbe over $C_1 \cap C_2$. 
\end{itemize} 
When $\xi_1^{-1}(\bP(\vW_1))$ and $\xi_2^{-1}(\bP(\vW_2))$ are empty --- which holds for 
generic $L \subset V_1 \oplus V_2$ --- Lemma~\ref{lemma-JL-Z1Z2} thus gives a satisfactory geometric interpretation 
of $\cJ(\cB^1, \cB^2)_{\bP(L^{\perp})}$. 
Indeed, then~$\Perf(\ccP, \cR)$
is an \'{e}tale form of the 
geometric category $\Perf(\ccP)$.
More precisely, both of these categories are $\ccP$-linear, 
and there is an \'{e}tale cover $\tilde{\ccP} \to \ccP$  
such that after base 
change to~$\tilde{\ccP}$
the two categories are equivalent as $\tilde{\ccP}$-linear categories. 
Namely, $\cR$ 
is Azumaya by Lemma~\ref{lemma-JL-Z1Z2}, and an \'{e}tale 
cover over which it becomes a matrix algebra will do. 
\end{remark}

The following result summarizes our work above. 
Note that the structure sheaf of an Enriques surface is an exceptional object of its
derived category. 

\begin{theorem} 
\label{theorem-enriques} 
Let $\Sigma$ be a general Enriques surface. 
Then there exist \textup(non-canonical\textup) smooth, transverse cubic plane curves $C_1$ and $C_2$ in $\bP^2$ 
and Azumaya algebras $\cR_1$ and $\cR_2$ on the square root stacks $\bP^2(\sqrt{C_1})$ and $\bP^2(\sqrt{C_2})$, 
such that if 
\begin{equation*}
\ccP = \bP^2(\sqrt{C_1}) \times_{\bP^2} \bP^2(\sqrt{C_2}), 
\end{equation*} 
then $\cR = \cR_1 \boxtimes \cR_2 \in \Perf(\ccP, \cR)$ is an exceptional object and there is 
an equivalence between the subcategories 
\begin{equation*}
\langle \cO_{\Sigma} \rangle^{\perp} \subset \Perf(\Sigma) 
\qquad \text{and} \qquad 
{^{\perp}}\langle \cR \rangle \subset \Perf(\ccP, \cR). 
\end{equation*} 
\end{theorem}

\begin{remark}
Theorem~\ref{theorem-enriques} can be thought of as an algebraization of the logarithmic transform, that creates an 
Enriques surface from a rational elliptic surface with two marked fibers. 
\end{remark}


\appendix

\section{Linear categories} 
\label{section:linear-cats}

In this appendix, we collect some results on linear categories that 
are needed in the body of the text. 
As in \cite[Part I]{NCHPD}, in this section we will be considering general linear categories, 
as opposed to Lefschetz categories or categories linear over a projective bundle.  
To emphasize this we tend to denote categories with the letters 
$\cC$ or $\cD$ as opposed to $\cA$ or $\cB$. 

First recall that if $T$ is a scheme, then by Definition~\ref{definition-linear-category} a 
$T$-linear category is a small idempotent-complete stable $\infty$-category equipped with a $\Perf(T)$-module structure. 
The basic example of such a category is $\cC = \Perf(X)$ where $X$ is a scheme over $T$; 
in this case, the action functor $\cC \times \Perf(T) \to \cC$ is given by $(C, F) \mapsto C \otimes \pi^*(F)$, 
where $\pi \colon X \to T$ is the structure morphism.

Given $T$-linear categories $\cC$ and $\cD$, we can form their tensor product 
\begin{equation*}
\cC \otimes_{\Perf(T)} \cD, 
\end{equation*} 
which is a $T$-linear category characterized by the property that for any $T$-linear category 
$\cE$, the $T$-linear functors $\cC \otimes_{\Perf(T)} \cD \to \cE$ classify ``bilinear functors'' $\cC \times \cD \to \cE$ 
(see \cite[\S4.8]{lurie-HA}). 
In particular, there is a canonical functor 
\begin{equation*}
\cC \times \cD \to \cC \otimes_{\Perf(T)} \cD,
\end{equation*}
whose action on objects we denote by $(C,D) \mapsto C \boxtimes D$. 

\begin{lemma}[{\cite[Lemma 2.7]{NCHPD}}]
\label{lemma-generators-box-tensor}
The category $\cC \otimes_{\Perf(T)} \cD$ is thickly generated by objects of the form $C \boxtimes D$ for 
$C \in \cC$, $D \in \cD$, i.e. the smallest idempotent-complete triangulated subcategory containing all of these 
objects is $\cC \otimes_{\Perf(T)} \cD$ itself. 
\end{lemma}

A fundamental result is that in the geometric case, tensor products of linear categories correspond to 
fiber products of schemes (see~\S\ref{subsection-conventions} for a discussion of derived fiber products). 

\begin{theorem}[{\cite[Theorem 1.2]{bzfn}}]
\label{theorem-bzfn} 
Let $X \to T$ and $Y \to T$ be morphisms of schemes. 
Then there is a canonical equivalence 
\begin{equation*}
\Perf {\left( X \times_T Y \right) } \simeq \Perf(X) \otimes_{\Perf(T)} \Perf(Y), 
\end{equation*} 
where $X \times_T Y$ is the derived fiber product. 
\end{theorem} 

\begin{remark}
In \cite{bzfn}, the theorem is formulated for $X_1, X_2,$ and $T$ being so-called perfect stacks. 
Any quasi-compact, separated derived scheme is a perfect 
stack \cite[Proposition 3.19]{bzfn}, so with our conventions from \S\ref{subsection-conventions} 
any derived scheme is perfect. 
\end{remark} 

Let $\phi_1 \colon \cC_1 \to \cD_1$ and $\phi_2 \colon \cC_2 \to \cD_2$ be $T$-linear functors. 
They induce a $T$-linear functor between the tensor product categories $\cC_1 \otimes_{\Perf(T)} \cC_2 \to \cD_1 \otimes_{\Perf(T)} \cD_2$,
which we denote by~$\phi_1 \otimes \phi_2$.
This operation is compatible with adjunctions.

\begin{lemma}[{\cite[Lemmas 2.12]{NCHPD}}] 
\label{lemma-adjoints}
Let $\phi_1 \colon \cC_1 \to \cD_1$ and $\phi_2 \colon \cC_2 \to \cD_2$ be $T$-linear functors. 
\begin{enumerate}
\item 
If $\phi_1$ and $\phi_2$ both admit left adjoints $\phi_1^*$ and $\phi_2^*$ \textup(or right adjoints $\phi_1^!$ and $\phi_2^!$\textup), 
then the functor $\phi_1 \otimes \phi_2 \colon \cC_1 \otimes_{\Perf(T)} \cC_2 \to \cD_1 \otimes_{\Perf(T)} \cD_2$ has a 
left adjoint given by $\phi_1^* \otimes \phi_2^*$ \textup(or right adjoint given by $\phi_1^! \otimes \phi_2^!$\textup). 
\item 
If $\phi_1$ and $\phi_2$ both admit left or right adjoints and are fully faithful, then so is 
$\phi_1 \otimes \phi_2$. 

\item 
If $\phi_1$ and $\phi_2$ are both equivalences, then so is $\phi_1 \otimes \phi_2$. 
\end{enumerate}
\end{lemma} 

Semiorthogonal decompositions and admissible subcategories 
of linear categories are defined as in the usual triangulated case \cite[Definitions 3.1 and 3.5]{NCHPD}. 
We will frequently need to use that they behave well under tensor products. 

\begin{lemma}[{\cite[Lemma 3.17]{NCHPD}}]
\label{lemma-admissible-tensor}
Let $\cC$ and $\cD$ be $T$-linear categories. 
If $\cA \subset \cC$ is a left \textup(or right\textup) admissible $T$-linear subcategory, then 
so is $\cA \otimes_{\Perf(T)} \cD \subset \cC \otimes_{\Perf(T)} \cD$. 
\end{lemma}

\begin{lemma}[{\cite[Lemma 3.15]{NCHPD}}]
\label{lemma-sod-tensor}
Let $\cC = \llangle \cA_1, \dots, \cA_m \rrangle$ and $\cD = \llangle \cB_1, \dots, \cB_{n} \rrangle$  
be $T$-linear semiorthogonal decompositions. 
Then the tensor product of the embedding functors 
\begin{equation*}
\cA_i \otimes_{\Perf(T)} \cB_j \to \cC \otimes_{\Perf(T)} \cD
\end{equation*}
is fully faithful for all $i,j$. 
Moreover, there is a semiorthogonal decomposition 
\begin{equation*}
\cC \otimes_{\Perf(T)} \cD = \llangle \cA_i \otimes_{\Perf(T)} \cB_j \rrangle_{1 \leq i \leq m,\, 1 \leq j \leq n} 
\end{equation*} 
where the ordering on the set $\set{ (i,j) \st  1 \leq i \leq m , 1 \leq j \leq n }$ is any 
one which extends the coordinate-wise partial order. 
The projection functor onto the $(i,j)$-component of this decomposition is given by 
\begin{equation*}
\pr_{\cA_i} \otimes \pr_{\cB_j} \colon \cC \otimes_{\Perf(T)} \cD \to  \cA_i \otimes_{\Perf(T)} \cB_j, 
\end{equation*} 
where $\pr_{\cA_i} \colon \cC \to \cA_i$ and $\pr_{\cB_j} \colon \cD \to \cB_j$ are the projection 
functors for the given decompositions.
\end{lemma}

Finally, in the paper we need a couple formal tensor product identities. 
Before stating them, we note the following. 
Let $T_1, T_2$, and $X$ be schemes, and for $k=1,2$, let $\cC_k$ be a $T_k \times X$-linear category.  
Then the tensor product
\begin{equation*}
\cC_1 \otimes_{\Perf(X)} \cC_2 
\end{equation*}
is naturally a $T_1 \times T_2$-linear category via 
the equivalence 
$\Perf(T_1 \stimes T_2) \simeq \Perf(T_1) \sotimes \Perf(T_2).$ 

\begin{lemma}
\label{lemma-fiber-product-categories}
Let $T_1, T_2, X$, and $Y$ be schemes, and for $k=1,2$, let $\cC_k$ be a $T_k \times X$-linear category 
and let $\cD_k$ be a $T_k \times Y$-linear category.
Then there is an equivalence 
\begin{equation*}
\left( \cC_1 \otimes_{\Perf(X)} \cC_2 \right) \otimes_{\Perf(T_1 \stimes T_2)} \left(  \cD_1 \otimes_{\Perf(Y)} \cD_2 \right) 
\! \simeq \!
\left(  \cC_1 \otimes_{\Perf(T_1)}  \cD_1 \right) \otimes_{\Perf(X \stimes Y)} \left(  \cC_2 \otimes_{\Perf(T_2)} \cD_2 \right) . 
\end{equation*}
\end{lemma}

\begin{proof}
The equivalence is induced by the transposition of the middle two factors. 
\end{proof} 

Recall that if $\cC$ is a $T$-linear category and $T' \to T$ is a morphism, 
then we denote by~\mbox{$\cC_{T'} = \cC \otimes_{\Perf(T)} \Perf(T')$} the $T'$-linear base change category. 

\begin{corollary}
\label{corollary-base-change-along-diagonal}
For $k=1,2$, let $T_k$ be a scheme and let $\cC_k$ be a $T_k$-linear category. 
Let $Y$ be a scheme with a morphism $Y \to T_1 \stimes T_2$ corresponding 
to morphisms $Y \to T_1$ and $Y \to T_2$. 
Then there is an equivalence 
\begin{equation*}
\left( \cC_1 \sotimes \cC_2 \right) \otimes_{\Perf(T_1 \stimes T_2)} \Perf(Y) 
\simeq 
(\cC_{1})_Y \otimes_{\Perf(Y)} (\cC_2)_Y 
\end{equation*}
\end{corollary}

\begin{proof}
There is a canonical equivalence $\Perf(Y) \simeq \Perf(Y) \otimes_{\Perf(Y)} \Perf(Y)$. 
Now the result follows from Lemma~\ref{lemma-fiber-product-categories} by 
taking $X = S$ to be our base scheme and $\cD_1 = \cD_2 =  \Perf(Y)$. 
\end{proof}  


\section{Projected categorical joins}
\label{section-future}

Given closed subvarieties $X_1 \subset \bP(V)$ and~$X_2 \subset \bP(V)$ of the \emph{same} projective space, 
the classical join $\bJ(X_1, X_2)$ as we have defined it is a subvariety of~$\bP(V \oplus V)$.  
In this situation, it is more common to consider the join
\begin{equation*}
\bJ_V(X_1, X_2) \subset \bP(V)
\end{equation*}
\emph{inside} $\bP(V)$, 
defined as the Zariski closure of the 
union of all the lines between points of $X_1$ and~$X_2$ in $\bP(V)$. 
It is easy to see that $\bJ_V(X_1, X_2)$ is isomorphic to the image of {the classical join} $\bJ(X_1,X_2)$ under the linear projection
along the sum map~$V \oplus V \to V$. 
Accordingly, we call $\bJ_V(X_1, X_2)$ the \emph{projected join} of $X_1$ and $X_2$. 
Note also that we allow the possibility that~$X = X_1 = X_2$ coincide, 
in which case 
\begin{equation*}
\Sec(X) =  \bJ_V(X, X)
\end{equation*}
is known as the \emph{secant variety} of~$X$.

Let $V_1$ and $V_2$ be the linear spans of $X_1$ and $X_2$ in~$\bP(V)$ and
consider $X_1$ and $X_2$ as subvarieties of $\bP(V_1)$ and $\bP(V_2)$.
If $\bP(V_1)$ and $\bP(V_2)$ do not intersect in $\bP(V)$, the natural embedding~$V_1 \oplus V_2 \to V$
identifies the join $\bJ(X_1,X_2) \subset \bP(V_1 \oplus V_2)$ with the projected join~$\bJ_V(X_1,X_2)$. 
Otherwise, if $K = V_1 \cap V_2$ and we consider the diagonal embedding \mbox{$K \to K \oplus K \to V_1 \oplus V_2$}, then 
we have $\bJ(X_1, X_2) \cap \bP(K) = X_1 \cap X_2$, 
and the projection induces a regular surjective map 
\begin{equation}
\label{projected-join-map} 
\Bl_{X_{1} \cap X_{2}} (\bJ(X_1, X_2)) \to \bJ_V(X_1, X_2).
\end{equation}  
Note that this morphism is often generically finite. 
Thus, $\Bl_{X_{1} \cap X_{2}} \bJ(X_1, X_2)$ can be regarded as an approximation to $\bJ_V(X_1, X_2)$. 
In this appendix, we use this observation to categorify the projected join, 
and show that under HPD it corresponds to the operation of taking fiber products. 

\subsection{Linear projections and HPD} 

Let $\tV \to V$ be a surjection of vector bundles over $S$, with kernel $K$. 
Linear projection from $K$ gives a blowup diagram 
\begin{equation}
\label{diagram-tV-V-projection}
\vcenter{\xymatrix{
\tbE \ar[r]^-{\teps} \ar[d]_{b_{\tbE}} &
\Bl_{\bP(K)}(\bP(\tV)) \ar[d]^b \ar[r]^-{q} &
\bP(V) \\
\bP(K) \ar[r] & 
\bP(\tV)
}}
\end{equation}
with exceptional divisor $\tbE$. 
If $\cA$ is a $\bP(\tV)$-linear category, we define the \emph{linear projection of~$\cA$ along 
$\bP(K)$} by 
\begin{equation*}
\Bl_{\bP(K)}(\cA) = \cA \otimes_{\Perf(\bP(\tV))} \Perf(\Bl_{\bP(K)}\bP(\tV)), 
\end{equation*}
which we regard as a $\bP(V)$-linear category via the morphism $q \colon \Bl_{\bP(K)}(\bP(\tV)) \to \bP(V)$.

\begin{proposition}
\label{proposition-HPD-projection-general}
Let $\cA$ be a Lefschetz category over $\bP(\tV)$ with center $\cA_0$. 
Assume $\tV \to V$ is a surjection with kernel $K$, and 
$\length(\cA) < \rank(V)$. 
Then $\Bl_{\bP(K)}(\cA)$ has the structure of a moderate 
Lefschetz category over $\bP(V)$ of length $\rank(V) - 1$ with center 
\begin{equation*}
\Bl_{\bP(K)}(\cA)_0 = \llangle 
b^* \cA_0, \teps_* b_{\tbE}^*(\cA_{\bP(K)})
\rrangle, 
\end{equation*} 
which is right \textup(or left\textup) strong if $\cA$ is. 
Moreover, there is a $\bP(\vV)$-linear equivalence 
\begin{equation*}
(\Bl_{\bP(K)}(\cA)/\bP(V))^{\hpd} \simeq (\cA/\bP(\tV))^{\hpd} \otimes_{\Perf(\bP(\tV^{\svee}))} \Perf(\bP(\vV)). 
\end{equation*}
\end{proposition}

\begin{proof}
In the case where $\cA$ is given by the derived category of a variety, this is the main result of 
\cite{carocci2015homological}. 
The proof carries over directly to our setting. 
We note that out of preference, the Lefschetz center we have used for $\Bl_{\bP(K)}(\cA)$ 
is a twist by $\cO(-(\rank(V)-1)\tbE)$ of the one from \cite{carocci2015homological}. 
We also note that there is an explicit formula
\begin{equation*}
\Bl_{\bP(K)}(\cA)_i = \begin{cases}
\llangle 
b^* \cA_{i}, \teps_* b_{\tbE}^*(\cA_{\bP(K)})
\rrangle & 
\text{for $0 \leq i \leq \length(\cA) - 1$} , \\ 
\teps_* b_{\tbE}^*(\cA_{\bP(K)}) & 
\text{for $\length(\cA) \leq i \leq \rank(V) - 2$}, 
\end{cases}
\end{equation*} 
for the right Lefschetz components of $\Bl_{\bP(K)}(\cA)$. 
\end{proof} 

\subsection{Projected categorical joins and HPD}
Now let $V_1$ and $V_2$ be nonzero vector bundles on $S$, 
equipped with morphisms $V_1 \to V$ and $V_2 \to V$ such that $\tV = V_1 \oplus V_2 \to V$ is surjective. 
Let $K \subset \tV$ be the kernel of this surjection. 

\begin{definition}
\label{definition-relative-join}
Let $\cA^1$ and $\cA^2$ be Lefschetz categories over the projective bundles $\bP(V_1)$ and~$\bP(V_2)$ 
such that~$\length(\cA^1) + \length(\cA^2) < \rank(V)$. 
Then the \emph{projected categorical join} of~$\cA^1$ and~$\cA^2$ over~$\bP(V)$ is 
the Lefschetz category over~$\bP(V)$ defined by 
\begin{equation*}
\cJ_V(\cA^1, \cA^2) = \Bl_{\bP(K)}(\cJ(\cA^1, \cA^2)). 
\end{equation*} 
\end{definition}

Note that by Theorem~\ref{theorem-join-lef-cat} we have $\length(\cJ(\cA^1, \cA^2)) = \length(\cA^1) + \length(\cA^2)$, 
so by Proposition~\ref{proposition-HPD-projection-general} 
the category $\cJ_V(\cA^1, \cA^2)$ does indeed have a natural Lefschetz structure 
over~$\bP(V)$. 
If $V = V_1 \oplus V_2$, then $\cJ_V(\cA^1, \cA^2) = \cJ(\cA^1, \cA^2)$ agrees with the usual categorical join. 

\begin{remark}
\label{remark-projected-join-center} 
If $V_1 \to V$ and $V_2 \to V$ are embeddings, the Lefschetz center of $\cJ_V(\cA^1, \cA^2)$ can 
be described explicitly as follows. 
In this case $K \cong V_1 \cap V_2 \subset V$
and $\bP(K) \subset \bP(V_1 \oplus V_2)$ does not intersect $\bP(V_1) \sqcup \bP(V_2)$, hence 
by Proposition~\ref{proposition-cJT} we have an equivalence 
\begin{equation*}
\cJ(\cA^1, \cA^2)_{\bP(K)} \simeq \cA^1_{\bP(K)} \otimes_{\Perf(\bP(K))} \cA^2_{\bP(K)} . 
\end{equation*}
Thus, by Proposition~\ref{proposition-HPD-projection-general} and the definition (Definition~\ref{definition-lef-center-join}) 
of the center of $\cJ(\cA^1, \cA^2)$, the center of $\cJ_V(\cA^1, \cA^2)$ is given by 
\begin{equation*}
\cJ_V(\cA^1, \cA^2)_0 = \llangle \cA_0^1 \otimes \cA_0^2, \cA^1_{\bP(K)} \otimes_{\Perf(\bP(K))} \cA^2_{\bP(K)} \rrangle, 
\end{equation*} 
where we have suppressed the embedding functors of the components. 
If $V_1 = V_2 = V$, this formula simply reads 
\begin{equation*}
\cJ_V(\cA^1, \cA^2)_0 = \llangle \cA_0^1 \otimes \cA_0^2 , \cA^1 \otimes_{\Perf(\bP(V))} \cA^2 \rrangle. 
\end{equation*}
\end{remark}

If $\cA^1$ and $\cA^2$ are Lefschetz categories over $\bP(V)$, then we write 
$\cJ_V(\cA^1, \cA^2)$ for the projected categorical join over $\bP(V)$ where we 
take $V_1 = V_2 = V$. 
In this case the assumptions of Remark~\ref{remark-projected-join-center} are satisfied and $K = V$.
Our results imply an appealing formula for the HPD category in this setting.

\begin{corollary}
\label{corollary-relative-joins} 
Let $\cA^1$ and $\cA^2$ be right strong, moderate Lefschetz categories over 
$\bP(V)$ such that $\length(\cA^1) + \length(\cA^2) < \rank(V)$. 
Then there is a $\bP(\vV)$-linear equivalence 
\begin{equation*}
\cJ_V(\cA^1, \cA^2)^{\hpd} \simeq (\cA^1)^{\hpd} \otimes_{\Perf(\bP(\vV))} (\cA^2)^{\hpd}. 
\end{equation*} 
\end{corollary}

\begin{proof}
Combine Theorem~\ref{theorem-joins-HPD}, Proposition~\ref{proposition-HPD-projection-general}, 
and Proposition~\ref{proposition-cJT}. 
\end{proof}

\begin{remark}
Jiang and Leung \cite{jiang-leung-joins} also recently proved a version of Corollary~\ref{corollary-relative-joins}. 
While we deduce this as a direct consequence of our main result, Theorem~\ref{theorem-joins-HPD}, 
they give an argument that does not rely on (but is closely related to the proof of) 
Theorem~\ref{theorem-joins-HPD}. 
\end{remark}

\begin{remark}
\label{remark-cJv-problems}
Recall that two key properties of categorical joins are that they preserve smoothness, and in the geometric case give a categorical resolution of the classical join. 
In contrast, if $\cA^1$ and $\cA^2$ in Corollary~\ref{corollary-relative-joins} are smooth and proper over $S$, 
their projected categorical join $\cJ_V(\cA^1, \cA^2)$ will not be smooth over $S$ unless they are transverse over $\bP(V)$, 
i.e. unless the category~$\cA^1 \otimes_{\Perf(\bP(V))} \cA^2$ is smooth. 
Further, if {$\cA^k = \Perf(X_k)$} for closed 
subvarieties~\mbox{$X_k \subset \bP(V)$}, 
the category $\cJ_V(\cA^1, \cA^2)$ will often not be 
birational to~$\bJ_V(X_1, X_2)$ over~$\bP(V)$ in the sense 
of Remark~\ref{remark-cJ-vs-classical}, 
corresponding to the fact that 
the map 
\begin{equation*}
\Bl_{X_{1} \cap X_{2}} \bJ(X_1, X_2) \to \bJ_V(X_1, X_2) 
\end{equation*}
is often only generically finite. 
It would be interesting to find a modification of~$\cJ_V(\cA^1, \cA^2)$ that fixes the above two issues. 
On the other hand, an advantage of $\cJ_V(\cA^1,\cA^2)$ as defined above is that 
by Corollary~\ref{corollary-relative-joins}  its HPD has a nice description. 
\end{remark}

\begin{remark}
If $S$ is the spectrum of a field, 
for a closed subvariety $X \subset \bP(V)$ 
the classical secant variety $\Sec(X) = \bJ_V(X, X) \subset \bP(V)$ 
is given by the join of $X$ with itself inside~$\bP(V)$. 
Note that the map $\Bl_{X} \bJ(X, X) \to \Sec(X)$ (take $X_1 = X_2 = X$ in~\eqref{projected-join-map}) 
is equivariant for the $\bZ/2$-action on the source induced by 
the $\bZ/2$-action on $V \oplus V$ swapping the summands. 
The quotient of $\Bl_{X} \bJ(X, X)$ by this action can thus be considered as a closer 
approximation to~$\Sec(X)$. 
Similarly, in the situation of Definition~\ref{definition-relative-join}, if 
$\cA = \cA^1 = \cA^2$ and $V = V_1 = V_2$, then there is a~$\bZ/2$-action on 
$\cJ_V(\cA, \cA)$ induced by the one on $\cA \otimes \cA$ given by swapping the 
two factors. 
The equivariant category $\cJ_V(\cA, \cA)^{\bZ/2}$ is a natural candidate for a 
\emph{categorical secant variety}, whose HPD can be described using Corollary~\ref{corollary-relative-joins}.
\end{remark}


\providecommand{\bysame}{\leavevmode\hbox to3em{\hrulefill}\thinspace}
\providecommand{\MR}{\relax\ifhmode\unskip\space\fi MR }
\providecommand{\MRhref}[2]{%
  \href{http://www.ams.org/mathscinet-getitem?mr=#1}{#2}
}
\providecommand{\href}[2]{#2}


\end{document}